\documentclass[reqno]{amsart}
\usepackage{etex}
\usepackage[a4paper,hmarginratio=1:1]{geometry}
\usepackage{amssymb,amsfonts,amsmath}
\usepackage{comment} 
\usepackage[all,arc]{xy}
\usepackage{enumerate}
\usepackage{mathrsfs,mathtools}
\usepackage{todonotes,booktabs}
\usepackage{stmaryrd}
\usepackage{marvosym}
\usepackage{graphicx}
\usepackage{pdflscape} 	

\usepackage{bbm}
\usepackage{tikz}
\usepackage{tikz-cd}
\usetikzlibrary{trees}
\usetikzlibrary[shapes]
\usetikzlibrary[arrows]
\usetikzlibrary{patterns}
\usetikzlibrary{fadings}
\usetikzlibrary{backgrounds}
\usetikzlibrary{decorations.pathreplacing}
\usetikzlibrary{decorations.pathmorphing}
\usetikzlibrary{positioning}
	
\def\biblio{\bibliography{duality}\bibliographystyle{alpha}}

\usepackage{xcolor} 
\usepackage{graphicx}

\usepackage[pagebackref]{hyperref}

\definecolor{dark-red}{rgb}{0.5,0.15,0.15}
\definecolor{dark-blue}{rgb}{0.15,0.15,0.6}
\definecolor{dark-green}{rgb}{0.15,0.6,0.15}
\hypersetup{
    colorlinks, linkcolor=dark-red,
    citecolor=dark-blue, urlcolor=dark-green
}

\newcommand{\cW}{\cal{W}}

\renewcommand*{\backref}[1]{}
\renewcommand*{\backrefalt}[4]{%
  \ifcase #1 %
No citations.
  \or
(cit. on p. #2).%
  \else
(cit on pp. #2).%
  \fi%
}
\usepackage{subfiles}

\usepackage[nameinlink,capitalise,noabbrev]{cleveref}

\newtheorem{thm}{Theorem}[section]
\newtheorem{cor}[thm]{Corollary}
\newtheorem{prop}[thm]{Proposition}
\newtheorem{lem}[thm]{Lemma}

\theoremstyle{definition}
\newtheorem{defn}[thm]{Definition}

\newtheorem{ex}[thm]{Example}

\theoremstyle{remark}
\newtheorem{rem}[thm]{Remark}
 
\theoremstyle{theorem}
\newtheorem*{thm*}{Theorem}

\makeatletter
\let\c@equation\c@thm
\makeatother
\numberwithin{equation}{section}

\DeclareMathOperator{\Sp}{Sp}
\DeclareMathOperator{\Hom}{Hom}

\DeclareMathOperator{\End}{End}
\DeclareMathOperator{\colim}{colim}
\DeclareMathOperator{\cA}{\mathcal{A}}
\DeclareMathOperator{\cC}{\mathcal{C}}
\DeclareMathOperator{\cD}{\mathcal{D}}
\DeclareMathOperator{\cE}{\mathcal{E}}

\DeclareMathOperator{\cS}{\mathcal{S}}
\DeclareMathOperator{\cT}{\mathcal{T}}

\DeclareMathOperator{\cF}{\mathcal{F}}
\DeclareMathOperator{\cG}{\mathcal{G}}
\DeclareMathOperator{\cV}{\mathcal{V}}

\DeclareMathOperator{\Rep}{Rep}

\DeclareMathOperator{\Spec}{Spec}
\DeclareMathOperator{\Mod}{Mod}

\DeclareMathOperator{\StMod}{StMod}

\DeclareMathOperator{\Loc}{Loc}
\DeclareMathOperator{\Coloc}{Coloc}

\DeclareMathOperator{\Thick}{Thick}
\DeclareMathOperator{\Ind}{Ind}

\DeclareMathOperator{\res}{res}

\newcommand{\N}{\mathbb{N}}
\newcommand{\Q}{\mathbb{Q}}
\DeclareMathOperator{\Res}{Res}
\DeclareMathOperator{\Coind}{Coind}

\DeclareMathOperator{\id}{id}

\DeclareMathOperator{\Mor}{Mor}
\DeclareMathOperator{\Spin}{Spin}

\DeclareMathOperator{\supp}{supp}
\DeclareMathOperator{\cosupp}{cosupp}

\newcommand{\kos}[2]{{#1}/\!\!/{#2}}

\newcommand{\cal}{\mathcal}
\newcommand{\xr}{\xrightarrow}
\newcommand{\cH}{\mathcal{H}}

\newcommand{\Z}{\mathbb{Z}}

\Crefname{figure}{Figure}{Figures}
\Crefname{assu}{Assumption}{Assumptions}
\Crefname{lem}{Lemma}{Lemmas}
\Crefname{thm}{Theorem}{Theorems}
\Crefname{prop}{Proposition}{Propositions}

\renewcommand{\frak}{\mathfrak}

\DeclareMathOperator{\Inj}{Inj}

\newcommand{\fp}{\mathfrak{p}}
\newcommand{\fq}{\mathfrak{q}}
\newcommand{\fr}{\mathfrak{r}}

\newcommand{\recollement}[5]{
\xymatrix{{#1} \ar[r]|-{#2} & #3 \ar[r]|-{#4} \ar@<1ex>[l]^-{{#2}_!} \ar@<-1ex>[l]_-{{#2}^*} & #5, \ar@<1ex>[l]^-{{#4}!} \ar@<-1ex>[l]_-{{#4}^*}
}}
\let\lim\relax

\DeclareMathOperator{\lim}{lim}
\DeclareMathOperator{\Proj}{Proj}
\newcommand{\mm}{/\!\!/}

\DeclareMathOperator{\op}{op}

\newcommand{\cZ}{\mathcal{Z}}
\newcommand{\cU}{\mathcal{U}}

\newcommand{\F}{\mathbb{F}}

\newcommand{\cL}{\mathcal{L}}

\newcommand{\pc}[1]{{#1}_{p}^{\wedge}}

\DeclareMathOperator{\map}{map}

\newcommand{\Gorenstein}{absolute Gorenstein }
\newcommand{\Normalization}{Gorenstein normalization }

\title[Stratification and duality for homotopical groups]{Stratification and duality for homotopical groups}

\author{Tobias Barthel}
\address{Max-Planck-Institut f\"ur Mathematik, Vivatsgasse 7, 53111 Bonn, Germany}
\email{barthel.tobi@gmail.com}

\author{Nat{\`a}lia Castellana}
\address{Departament de Matem\`atiques, Universitat Aut\`onoma de Barcelona and BGSMath, 08193 Bellaterra, Spain}
\email{natalia@mat.uab.cat}

\author{Drew Heard}
\address{Fakult{\"a}t f{\"u}r Mathematik, Universit{\"a}t Regensburg}
\email{drew.k.heard@gmail.com}

\author{Gabriel Valenzuela}
\address{Max-Planck-Institut f\"ur Mathematik, Vivatsgasse 7, 53111 Bonn, Germany}
\email{gvalenzuela@mpim-bonn.mpg.de}

\date{\today}


\begin{document}

\begin{abstract}
We generalize Quillen's $F$-isomorphism theorem, Quillen's stratification theorem, the stable transfer, and the finite generation of cohomology rings from finite groups to homotopical groups. As a consequence, we show that the category of module spectra over $C^*(B\cG,\F_p)$ is stratified and costratified for a large class of $p$-local compact groups $\cG$ including compact Lie groups, connected $p$-compact groups, and $p$-local finite groups, thereby giving a support-theoretic classification of all localizing and colocalizing subcategories of this category. Moreover, we prove that $p$-compact groups admit a homotopical form of Gorenstein duality.
\end{abstract}

\maketitle

\def\biblio{}

\setcounter{tocdepth}{1}
\tableofcontents

\section{Introduction}

Let $G$ be a finite group and suppose $k$ is a field whose characteristic $p$ divides the order of $G$. Modular representation theory studies the $k$-linear representations of $G$, or equivalently the abelian category of modules over the group algebra $kG$. Since classifying these objects up to isomorphism is in general impossible, one may consider two simplifications: Firstly divide out by the projective representations to form the stable module category $\StMod_{kG}$ of $G$, and secondly classify objects up to a coarser equivalence relation than isomorphism. This leads to the question of how to classify thick subcategories of compact objects and localizing subcategories of $\StMod_{kG}$. 

The classification of thick subcategories was given in \cite{bcr_thick} using support theoretic techniques. Building on their work, Benson, Iyengar, and Krause \cite{benson_local_cohom_2008,bik11} developed the notion of stratification of a triangulated category by a Noetherian commutative ring that captures both the classification of thick and localizing subcategories. Their work then culminated in the statement that the canonical action of $H^*(BG,k)$ stratifies $\StMod_{kG}$, thus giving a complete  classification of all localizing ideals in terms of subsets of the support variety $\Proj(H^*(BG,k))$.

In \cite{blo_pcompact}, Broto, Levi, and Oliver introduced the powerful concept of $p$-local compact groups as a common generalization of the notions of $p$-compact group \cite{dwyerwilkerson_finloop} as well as fusion systems $\cF$ on a finite group \cite{blo_fusion}. A $p$-local compact group $\cG = (S,\cF,\cL)$ consists of a saturated fusion system on a discrete $p$-toral group $S$ together with a centric linking system $\cL$. Informally speaking, this definition provides a combinatorial model of the $p$-local structure of a compact Lie group $(S,\cF)$. The extra structure $\cL$ makes it possible to construct a ($p$-completed) classifying space $B\cG$ associated to $\cG$ which is uniquely determined \cite{chermak_existence,boboliver_existence,ll_uniqueness}, thus making saturated fusion systems amenable to homotopical techniques. Examples are given by compact Lie groups with no restriction on the group of components \cite{blo_pcompact} as well as $p$-completions of finite loop spaces \cite{BLO_finiteloopspaces}. The main result of the present paper is: 

\begin{thm*}
If $\cG  = (S,\cF,\cL)$ is a $p$-local compact group such that $S$ is geometric in the sense defined below, then the category $\Mod_{C^*(B\cG,\F_p)}$ of module spectra over the cochains $C^*(B\cG,\F_p)$ is canonically stratified and costratified.
\end{thm*}

In particular, this result contains as special cases the main theorems of \cite{bik_finitegroups} for $p$-groups and \cite{bg_stratifyingcompactlie} for compact Lie groups $G$ with $\pi_0G$ a finite $p$-group, but applies to many other examples of homotopical groups as well, such as connected $p$-compact groups and $p$-local finite groups. In order to prove it, we employ techniques from unstable and stable homotopy theory to establish generalizations of Quillen's $F$-isomorphism and Quillen's stratification theorem, Chouinard's theorem, and the support theory of Benson, Iyengar, and Krause to the context of homotopical groups.

\subsection*{Summary of results and methods}

From a formal point of view, the key difference between finite $p$-groups or connected Lie groups and general homotopical groups is that the canonical morphism induced by a Sylow subgroup inclusion $\phi_{\cG}\colon C^*(B\cG, \F_p) \to C^*(BS,\F_p)$ is not necessarily finite, i.e., $C^*(BS,\F_p)$ is not necessarily compact as a $C^*(B\cG,\F_p)$-module. In fact, we characterize those $p$-local compact groups for which $\phi_{\cG}$ is finite as the $p$-compact groups. Therefore, our first goal is to develop an abstract generalization of the methods from \cite{bik_finitegroups} which allows descent along certain nonfinite morphisms. 

To this end, we prove several general base-change formulae (e.g., \Cref{thm:suppbasechange}) for support along morphisms of ring spectra, which extend the theory for finite morphisms developed in \cite{bik12}. We then isolate sufficient conditions for descent of stratifications and costratifications: a morphism $f\colon R \to S$ is said to satisfy Quillen lifting if certain prime ideals in $\Spec^h(\pi_*R)$ can be lifted along $f$; for a precise definition, see \Cref{defn:qlifting}. 

\begin{thm*}[\Cref{thm:stratdescent}, \Cref{thm:costratdescent}]
Suppose that $f \colon R \to S$ is a morphism of Noetherian ring spectra satisfying Quillen lifting and such that induction and coinduction along $f$ are conservative. If $\Mod_S$ is canonically stratified, then so is $\Mod_R$. If $f$ additionally admits an $R$-module retract, then canonical costratification descends along $f$ as well. 
\end{thm*}

Checking that the conditions of the theorem are satisfied for $\phi_{\cG}\colon C^*(B\cG, \F_p) \to C^*(BS, \F_p)$ crucially relies on three ingredients: A generalization of Chouinard's theorem to discrete $p$-toral groups, the existence of approximations of $p$-local compact groups by $p$-local finite groups due to Gonzalez \cite{gonzalez_approxpcompact}, and the existence of a stable transfer $\Sigma^{\infty}_+ B\cal{G} \to \Sigma^{\infty}_+ BS$ whenever $\cal{G}$ is a $p$-local finite group (i.e., when $S$ is a finite $p$-group), as proven by Ragnarsson \cite{ragnarsson_transfer}. The first ingredient  is  proven for discrete $p$-toral groups which are geometric in the sense that they are $p$-discrete approximations  of $p$-toral groups. As a first step, in order to construct a general transfer, we have to deal with the subtle problem that the stable transfers constructed in the finite case are not necessarily compatible with the given morphisms in the approximation tower, which prevented progress in this area to date. Overcoming this issue requires the study of phantom maps via Brown--Comenetz duality, and is therefore an intrinsically stable result. This motivates the terminology stable transfer for the resulting $C^*(B\cG, \F_p)$-module retract of $\phi_{\cG}$.

\begin{thm*}[\Cref{prop:plocalcompactchouniard}]
Any $p$-local compact group $\cG = (S,\cF)$ admits a stable transfer $C^*(BS, \F_p) \to C^*(B\cG, \F_p)$ of $C^*(B\cG, \F_p)$-modules.
\end{thm*}

One consequence of this theorem is that the cohomology ring $H^*(B\cG,\F_p) = \pi_{-*}C^*(B\cG, \F_p)$ is a finitely generated $\F_p$-algebra for any $p$-local compact group $\cG$, generalizing the fundamental result for finite groups and compact Lie groups proven in \cite{evens,venkov}, and for $p$-compact groups \cite{dwyerwilkerson_finloop}. 

\begin{thm*}[\Cref{cor:noetherian}]
The cohomology ring $H^*(B\cG,\F_p)$ is Noetherian for any $p$-local compact group $\cG$.
\end{thm*}

In a second step, we then use the stable elements formula of Gonzalez for $H^*(B\cG,\F_p)$ \cite[Thm. 2]{gonzalez_approxpcompact} as well as work of Rector \cite{rector_quillenstrat} and Broto--Zarati \cite{brotozarati_steenrod} to extend the $F$-isomorphism theorem for $p$-local finite groups proven in \cite{blo_fusion} to $p$-local compact groups. This allows us to verify that Rector's general formalism applies to arbitrary $p$-local compact groups to deduce a strong form of Quillen stratification from the $F$-isomorphism theorem, closely following Quillen's original argument~\cite{quillen_stratification}. 

\begin{thm*}[\Cref{thm:fiso}, \Cref{thm:qstratification}]
The $F$-isomorphism theorem holds for any $p$-local compact group $\cG$, i.e., there is an $F$-isomorphism
\[
\xymatrix{H^*(B\cG,\F_p) \ar[r]^-{ } & \varprojlim \limits_{\cF^e}H^*(BE,\F_p),}
\]
where $\cF^e$ is the full subcategory of $\cF$ on the elementary abelian subgroups of $S$. Moreover, the variety of $\cG$ admits a strong form of Quillen stratification:
\[
\cV_{\cG} \cong \coprod_{E\in \cE(\cG)}\cV_{\cG,E}^+,
\]
where $\cE(\cG)$ denotes a set of representatives of $\cF$-isomorphism classes of elementary abelian subgroups of $S$.
\end{thm*}

The Quillen stratification of this theorem generalizes previous work of Linckelmann \cite{linckelman_stratification}, and in fact gives an alternative proof that does not rely on the existence of certain bisets for $p$-local finite groups. 

The final ingredient is a generalization of Chouinard's theorem to discrete $p$-toral groups $S$ which satisfy an additional condition: $S$ is called geometric if there exists a compact Lie group $S'$ and an equivalence $BS^{\wedge}_p \simeq (BS')^{\wedge}_p$ of $p$-completed classifying spaces. We deduce from the classification of connected $p$-compact groups due to Andersen, Grodal, M{\o}ller, and Viruel~\cite{AGMV_classification1,AG_classification2} as well as previous work of Benson and Greenlees~\cite{bg_stratifyingcompactlie} that a large class of homotopical groups is geometric in this sense. When combined with the stable transfer constructed above, we establish:

\begin{thm*}[\Cref{prop:plocalcompactchouniard}, \Cref{cor:geometricexamples}]
If $\cG = (S,\cF)$ is a $p$-local compact group with geometric $S$, then Chouinard's theorem holds for $\cG$, in the sense that a $C^*(B\cG,\F_p)$-module is trivial if and only if it is trivial after induction (resp.~coinduction) along $C^*(B\cG,\F_p) \to C^*(BE,\F_p)$ for each elementary abelian subgroup $E \le S$. Moreover, $S$ is geometric if $\cG$ belongs to one of the following classes:
	\begin{enumerate}
		\item compact Lie groups, or
		\item connected $p$-compact groups, or
		\item $p$-local finite groups.
	\end{enumerate}
\end{thm*}

Applying our abstract descent technique in connection with the previous three theorems, we thus obtain our main result:

\begin{thm*}[\Cref{thm:plocalstrat}, \Cref{thm:plocalcostrat}, \Cref{thm:costratpgroupcons}]
If $\cG $ is a $p$-local compact group with geometric $S$, then $\Mod_{C^*(B\cG, \F_p)}$ is canonically stratified and costratified. In particular, there are bijections
\[
\begin{Bmatrix}
\text{Localizing subcategories} \\
\text{of } \Mod_{C^*(B\cG, \F_p)}
\end{Bmatrix} 
\xymatrix@C=2pc{ \ar@{<->}[r]^-{\sim} &}
\begin{Bmatrix}
\text{Subsets of}  \\
\Spec^h(H^*(B\cG,\F_p))
\end{Bmatrix}
\xymatrix@C=2pc{ \ar@{<->}[r]^-{\sim} &}
\begin{Bmatrix}
\text{Colocalizing subcategories}  \\
\text{of } \Mod_{C^*(B\cG, \F_p)}
\end{Bmatrix}
\]
as well as
\[
\begin{Bmatrix}
\text{Thick subcategories} \\
\text{of } \Mod_{C^*(B\cG, \F_p)}^{\mathrm{compact}}
\end{Bmatrix} 
\xymatrix@C=2pc{ \ar@{<->}[r]^-{\sim} &}
\begin{Bmatrix}
\text{Specialization closed subsets of}  \\
\Spec^h(H^*(B\cG,\F_p))
\end{Bmatrix}.
\]
Moreover, the telescope conjecture holds in $\Mod_{C^*(B\cG, \F_p)}$.
\end{thm*}

Finally, we turn to the duality theory for homotopical groups. The proof of Benson's conjecture by Benson and Greenlees \cite{bg_localduality} reveals that $C^*(BG,\F_p)$ is an absolute Gorenstein ring spectrum for any finite group $G$, i.e., admits an algebraic form of Poincar\' e duality. We use our general methods from \cite{bhv2} to extend this result to $p$-compact groups. 

\begin{thm*}[\Cref{thm:p-comp_goren}]
Let $\cG$ be a $p$-compact group of dimension $w$, then $\cG$ is absolute Gorenstein, i.e., for each $\fp \in \Spec^h(H^*(B\cG,\F_p))$ of dimension $d$, there is an isomorphism
\[
\pi_*(\Gamma_{\fp}C^*(B\cG,\F_p)) \cong I_{\fp}[w+d],
\]
where $I_{\fp}$ denotes the injective hull of $(H^*(B\cG,\F_p))/\fp$, and $\Gamma_{\fp}$ is the local cohomology functor constructed by Benson, Iyengar, and Krause. 
\end{thm*}

As an immediate corollary this implies the existence of a local cohomology spectral sequence for $p$-compact groups, see \Cref{cor:pgroupss}.

\subsection*{Outline of the document}

After establishing the conventions that will be in place throughout the document, we start in \Cref{sec:recollections} by revisiting the relevant aspects of support theory and stratifications. In particular, we state the main consequences of stratification and costratification as proven by \cite{bik11}. The next section contains our general base-change formulae for support and cosupport, leading to the proof of our main abstract descent result. \Cref{sec:homotopicalgroups} begins with a review of some background material on $p$-local compact groups and some of their basic properties.  Motivated by the examples coming from homotopical groups, we introduce the notion of a geometric discrete $p$-toral group and prove that these satisfy Chouinard's theorem. We then establish a general criterion which allows us to produce stable transfers for arbitrary $p$-local compact groups, a key step in generalizing Chouinard's theorem to this setting. The main theorems about $p$-local compact groups are then proven in \Cref{sec:stratification}, namely the $F$-isomorphism theorem, Quillen stratification of the cohomology ring, and the stratification of the category $\Mod_{C^*(B\cG,\F_p)}$. This section also contains a discussion of finiteness properties of homotopical groups, leading to a characterization of $p$-compact groups among all $p$-local compact groups. In the final section, we discuss  duality properties of homotopical groups and show that $p$-compact groups satisfy a homotopical version of Gorenstein duality.

\subsection*{Conventions}
Throughout this paper we work in setting of stable $\infty$-categories \cite{ha}. A subcategory $\cT \subseteq \cC$ of a stable $\infty$-category is called thick if it is closed under finite colimits, retracts, and desuspensions, and $\cT$ is called localizing (respectively colocalizing) if it is closed under all filtered colimits (respectively all filtered limits) as well. For a collection of objects $\cS \subseteq \cC$, we denote the smallest localizing (respectively colocalizing) subcategory of $\cC$ containing $\cS$ by $\Loc_{\cC}(\cS)$ (respectively $\Coloc_{\cC}(\cS)$). Upon passage to homotopy categories, these definitions correspond to the analogous notions in the setting of triangulated categories. If the ambient category is clear from the context, we will omit the subscript $\cC$.

In particular, we work with the stable $\infty$-category of spectra $\Sp$ \cite{ha}. However, our work is not sensitive to a particular choice of model; a good symmetric monoidal point set model of spectra, such as $S$-algebras \cite{ekmm} or symmetric spectra \cite{sym_spectra} would also suffice. 

Given such a model, we refer to a commutative monoid object as a commutative ring spectrum. For a commutative ring spectrum $R$ we write $\Mod_R$ for the category of $R$-modules and $\Mod_R^{\omega}$ for the full subcategory of compact objects, i.e., the smallest stable subcategory of $\Mod_R$ containing $R$ which is closed under retracts. For $M, N \in \Mod_R$ we write $M \otimes_R N$ for the symmetric monoidal product, and $\Hom_R(M,N)$ for the spectrum of $R$-module morphism between $M$ and $N$. 	

Given a space $X$ we write $X_+$ for the suspension spectrum $\Sigma^\infty_+X$ and $F(X_+,R)$ for the mapping spectrum of $X_+$ and $R$. If $R$ is a commutative ring spectrum, then so is $F(X_+,R)$. Often $R = Hk$ will be the Eilenberg--MacLane spectrum of a discrete commutative ring $k$; in this case we simply write $C^*(X,k)$ for $F(X_+,Hk)$, the spectrum of $k$-valued cochains on $X$. 

The $p$-completion of a space $X$, denoted $\pc{X}$, always refers to Bousfield--Kan $p$-completion \cite{bousfield_homotopy_1972}. A space $X$ is called $p$-complete if the $p$-completion map $X \to \pc{X}$ is a homotopy equivalence, and is called $\F_p$-finite if $H^*(X,\F_p)$ is finite. A space is $p$-good if the natural completion map $X \to \pc{X}$ induces an isomorphism on mod $p$ cohomology.

All discrete rings $R$ in this paper are assumed to be commutative and graded, and all ring-theoretic notions are implicitly graded. In particular, an $R$-module $M$ refers to a graded $R$-module and we write $\Mod_R$ for the abelian category of discrete graded $R$-modules. Prime ideals in $R$ will be denoted by fraktur letters $\fp,\fq,\fr$ and are always homogeneous, so that $\Spec^h(R)$ refers to the Zariski spectrum of homogeneous prime ideals in $R$. 

Finally, our grading conventions are homological; differentials always decrease degree. Thus, for example, $\pi_{-i}C^*(BG,k) \cong H^{i}(BG,k)$, i.e., $C^*(BG,k)$ is a coconnective commutative ring spectrum. 

\subsection*{Acknowledgments}

We would like to thank John Greenlees and Wolfgang Pitsch for helpful conversations, Tilman Bauer for allowing us to include his appendix in this paper, and James Cameron for pointing out a mistake in an earlier version of this manuscript. We also thank the referees for many useful comments. The first-named author was partially supported by the DNRF92, the second-named author was partially supported by FEDER-MEC grant MTM2016-80439-P and acknowledges financial support from the Spanish Ministry of Economy and Competitiveness, through the  ``Mar\'ia de Maeztu'' Programme for Units of Excellence in R\&D (MDM-2014-0445),
and the third-named author was partially supported by the SPP 1786. The fourth-named author would like to thank the Max Planck Institute for Mathematics for its hospitality.

\section{Recollections on support theory and stratifications}\label{sec:recollections}

\subsection{Preliminaries on ring spectra and modules}
In this short section we recall some basic properties of commutative ring spectra and modules over them, and collect several results that will be used throughout the document.

We start by introducing Noetherian ring spectra. By \cite[Thm.~1.1]{gy_noetheriangraded}, for a $\Z$-graded commutative ring $A_*$, the following conditions are equivalent:
\begin{enumerate}
  \item The underlying ungraded ring $A$ is Noetherian.
  \item The ring $A_0$ is Noetherian and $A$ is finitely generated as an $A_0$-algebra.
  \item Every homogeneous ideal of $A$ is finitely generated.  
\end{enumerate}
Therefore, it is not necessary to distinguish between Noetherian and graded Noetherian. 

\begin{defn}
A commutative ring spectrum $R$ is called Noetherian if $\pi_*R$ is Noetherian.
\end{defn}\sloppy
 A morphism $f\colon R \to S$ of commutative ring spectra induces a triple of adjoints $(\Ind_R^S,\Res_R^S,\Coind_R^S)$ between $\Mod_{R}$ and $\Mod_S$, where $\Ind_R^S \colon \Mod_R \to \Mod_S$ is induction $ -\otimes_R S$, $\Res_R^S\colon \Mod_S \to \Mod_R$ is restriction along $f$, and $\Coind_R^S\colon \Mod_R \to \Mod_S$ is coinduction, given by $\Hom_R(S,-)$. We often denote these adjoints as follows:
\[
\Ind = \Ind_R^S \quad  \Res = \Res_R^S \quad \Coind = \Coind_R^S.
\]

We recall that a functor $F\colon \cC \to \cD$ of stable $\infty$-categories is said to be conservative if it reflects equivalences. By stability, this is equivalent to the statement that, for any $X \in \cC$, $FX \simeq 0$ if and only if $X \simeq 0$.
\begin{lem}[Projection formula]\label{lem:projformula}
For any morphism $f\colon R \to S$ of commutative ring spectra there exists a canonical and natural equivalence
\[
\xymatrix{M \otimes_R \Res(N) \ar[r]^-{\sim} & \Res(\Ind (M) \otimes_S N)}
\]
for $M \in \Mod_R$ and $N \in \Mod_S$. In particular, $\Ind (M) \otimes_S N\simeq 0$ if and only if $M \otimes_R \Res(N)\simeq 0$.
\end{lem}
\begin{proof}
The proof is standard, see for example \cite[Prop.~2.15]{bds_wirth}; to wit: The required natural transformation is adjoint to
\[
\xymatrix{\Ind(M \otimes_R \Res(N)) \simeq \Ind (M) \otimes_S \Ind\Res(N) \ar[r] & \Ind (M) \otimes_S N,}
\] 
where the first map uses that $\Ind$ is symmetric monoidal. This map is an equivalence for $M=R$ and arbitrary $N$, so for all $M$ because both $\Ind$ and $\Res$ preserve colimits. The final claim follows because $\Res$ is conservative.
\end{proof}

The following lemma shows that if $f$ admits an $R$-module retract, then $\Ind$ and $\Coind$ are always conservative functors. 
\begin{lem}\label{lem:retract}
Assume $f\colon R \to S$ is a morphism of commutative ring spectra which admits an $R$-module retract, i.e., a map $g\colon S \to R$ of $R$-modules such that $g\circ f \simeq \id_R$. If $M \in \Mod_R$, then $M$ is an $R$-module retract of $\Res\Ind M$ and $\Res\Coind M$. In particular, $\Ind$ and $\Coind$ are conservative.
\end{lem}
\begin{proof}
Indeed, $\Res\Coind M \simeq \Res\Hom_R(S,M)$ admits $\Hom_R(R,M)\simeq M$ as an $R$-module retract, so we are done. The claim about induction is proven similarly. 
\end{proof}

\begin{lem}\label{lem:noethretract}
Suppose $f\colon R \to S$ is a morphism of commutative ring spectra which admits an $R$-module retract. If $S$ is Noetherian, then $R$ is Noetherian as well.
\end{lem}
\begin{proof}
By definition, the claim reduces to the analogous claim for (discrete) commutative rings, which is well-known, see for example \cite[Lem.~2.4]{dwyerwilkerson_finloop}. For the convenience of the reader, we give a quick alternative proof. Noetherian rings are characterized by the property that any direct sum of injective modules is injective. Let $(I_{\alpha})_{\alpha}$ be a set of injective $R$-modules. It follows that $\Coind I_{\alpha}$ is an injective $S$-module for all $\alpha$, hence $\bigoplus_{\alpha}\Coind I_{\alpha}$ is injective because $S$ is Noetherian. Therefore, $\Res\bigoplus_{\alpha}\Coind I_{\alpha} \cong \bigoplus_{\alpha}\Res\Coind I_{\alpha}$ is an injective $R$-module. Since $f$ is split as a map of $R$-modules, $\bigoplus_{\alpha}\Res\Coind I_{\alpha}$ admits $\bigoplus_{\alpha}I_{\alpha}$ as a retract by 
\Cref{lem:retract}, hence the latter object is injective as well. 
\end{proof}

We will also make repeated use of the following standard result. 

\begin{lem}\label{lem:preservecoloc}
Suppose $F\colon \cC \to \cD$ is a functor between presentable stable $\infty$-categories which preserves colimits or limits, respectively. If $X,Y \in \cC$ are objects such that $X \in \Loc(Y)$ or $X \in \Coloc(Y)$, then $FX \in \Loc(FY)$ or $FX \in \Coloc(FY)$, respectively.
\end{lem}
\begin{proof}
First note that such a functor $F$ is necessarily exact, in the sense that it preserves both finite limits and colimits, see \cite[Prop.~1.1.4.1]{ha} for example. With this in mind, we will prove the claim about localizing subcategories, the one about colocalizing subcategories being proven similarly. Consider the full subcategory $\cL \subseteq \cC$ consisting of all objects $U \in \cC$ such that $FU \in \Loc(FY)$. Since $\Loc(FY)$ is a localizing subcategory of $\cD$ and $F$ preserves colimits, $\cL \subseteq \cC$ is localizing as well. In more detail, it is clear that $\cL$ is thick, so consider a collection of objects $U_i \in \cL$ for some indexing set $I$. By definition, $FU_i \in \Loc(FY)$, hence $F(\bigoplus_{i\in I}U_i) \simeq \bigoplus_{i \in I}FU_i \in \Loc(FY)$ as $F$ preserves colimits. It follows that $\bigoplus_{i \in I}U_i \in \cL$, so $\cL$ is closed under all colimits. By assumption, $Y \in \cL$, so $X \in \Loc(Y) \subseteq \cL$. 
\end{proof}

\subsection{Local cohomology and support}\label{sec:localcohom}
In this section we summarize some of the constructions of \cite{benson_local_cohom_2008}; for an alternative approach, one can apply the techniques of \cite{bhv2}. We refer to these sources for the basic terminology of localization and colocalization sequences. An endofunctor $F$ on $\Mod_R$ will be called smashing if it preserves colimits, or equivalently if $F(M) \simeq F(R) \otimes_R M$ for all $M$, see \cite[Def.~3.3.2]{hps_axiomatic}.

Let $R$ be a Noetherian commutative ring spectrum, and suppose that $A$ is a discrete graded-commutative Noetherian ring that acts on $\Mod_R$, in the sense of \cite[Sec.~4]{benson_local_cohom_2008}. Usually $A$ will be $\pi_*R$ itself, acting in the evident way, in which case we say that the action is canonical. Given a specialization closed\footnote{A subset of $\Spec^h(A)$ is called specialization closed if it is a union of Zariski closed subsets of $\Spec^h(A)$. }  subset $\cV \subset \Spec^h(A)$ we can construct local cohomology and homology functors, denoted $\Gamma^A_{\cV}$ and $\Lambda_A^{\cV}$, with the following properties, see \cite[Sec.~4]{benson_local_cohom_2008} or \cite[Thm.~3.9]{bhv2}. 
\begin{thm}\label{thm:localfunctors}
 	Let $\cV \subset \Spec^h(A)$ be specialization closed. Then there exist two pairs of  adjoint functors $(\Gamma_{\cV}^A,\Lambda^{\cV}_A)$ and $(L_{\cV}^A,\Delta^{\cV}_A)$, and cofiber sequences 
 	\[
\Gamma_{\cV}^AM \to M \to L^A_{\cV}M \quad \text{ and } \quad \Delta_A^{\cV}M \to M \to \Lambda_A^{\cV}M. 
 	\] 
 	Moreover, $\Gamma^A_{\cV}$ is a smashing colocalization functor, $L^A_{\cV}$ is a smashing localization functor, and $\Lambda_A^{\cV}$ is a localization functor. 
 \end{thm} 

\begin{rem}\label{rem:orthogonal}
 	We will let $\Gamma_{\cal{V}}^A\Mod_R$ denote the essential image of $\Gamma_{\cal{V}}^A$, and similar for the other functors. By construction, we have that $M \in L_{\cal{V}}^A\Mod_R$ if and only if $\Hom_R(N,M) \simeq 0$ for all $N \in \Gamma_{\cal{V}}^A\Mod_R$, see for example, \cite[Thm.~2.4 and Thm.~3.9]{bhv2}.
 \end{rem}
In case the ring is clear from the context, we will drop the subscript or superscript. The functor $\Gamma_{\cV}$ is characterized by the property that $M \in \Gamma_{\cV}\Mod_R$ if and only if the algebraic $\fp$-localization $(\pi_*M)_{\fp} = 0$ for all $\fp \in \Spec^h(A) \setminus \cV$ \cite[Prop.~4.5]{benson_local_cohom_2008}. 

 Given $\fp \in \Spec^h(A)$ and $M \in \Mod_R$ we can define a Koszul object $\kos{M}{\fp} \in \Mod_R$ as in \cite[Sec.~5]{benson_local_cohom_2008} or \cite[Sec.~3]{bhv2} by modding out $\fp$ homotopically. More precisely, for any $x \in \pi_dR$ we define $\kos{R}{x}$ as the cofiber of the $R$-linear map $ R \xrightarrow{x} \Sigma^{-d}R$. Picking a sequence of generators $(x_1,\ldots,x_m)$ of the homogeneous prime ideal $\fp$, we then construct 
\[
\kos{M}{\fp} = M \otimes_R \kos{R}{x_1} \otimes_R \kos{R}{x_2} \otimes_R \ldots \otimes_R \kos{R}{x_m}.
\]
While this definition depends on the chosen sequence of generators, the thick subcategory generated by $\kos{M}{\fp}$ does not. If we consider the full subcategory $\Mod_R^{\cV-\text{tors}}$ of $\Mod_R$ given by $\Loc(\kos{R}{\fp} \mid \fp \in \cV)$, then the functor $\Gamma_{\cV}$ is the right adjoint to the inclusion of $\Mod_R^{\cV-\text{tors}}$ in $\Mod_R$, and the existence of the other functors follows from a formal argument as in \cite[Thm.~2.21]{bhv}. It follows from \cite[Prop.~3.21]{bhv2} that the functors constructed in this way agree with those constructed by Benson, Iyengar, and Krause in \cite{benson_local_cohom_2008}. 

The most important examples of specialization closed subsets we consider are 
\[
\cV(\fp) = \{  \fq \in \Spec^h(A) \mid \fq \supseteq \fp \} \quad \text{ and } \quad \cZ(\fp) = \{\fq \in \Spec^h(A) \mid \fq \not \subseteq \fp \}. 
\]
The latter gives rise to the localization functor $L_{\cZ(\fp)}$ satisfying $\pi_*L_{\cZ(\fp)}M \cong (\pi_*M)_{\fp}$, see \cite[Thm.~4.7]{benson_local_cohom_2008}. 
\begin{defn}
	For a homogeneous prime ideal $\fp \in \Spec^h(A)$ we define the local cohomology and homology functors with respect to $\fp$ to be 
	\[
\Gamma_{\fp} = \Gamma_{\cV(\fp)}L_{\cZ(\fp)} \quad \text{ and } \quad \Lambda^{\fp} = \Lambda^{\cV(\fp)}\Delta^{\cZ(\fp)}. 
	\]
\end{defn}

\begin{rem}\label{rem:localcohomfunctors}
	By \cite[Thm.~6.2]{benson_local_cohom_2008} one can replace $\cV(\fp)$ and $\cZ(\fp)$ with any specialization closed subsets $\cV$ and $\cW$ such that $\cV \setminus \cW = \{ \fp \}$. 
\end{rem}

This leads to the definition of the support and cosupport of a module $M \in \Mod_R$. 
\begin{defn}
	The support of an $R$-module $M$ is defined by
	 \[\supp_R(M) = \{ \fp \in \Spec^h(A) \mid \Gamma_{\fp}M \not \simeq 0 \}.\]
	 Similarly, the cosupport of $M$ is defined to be 
	 \[
	 \cosupp_R(M) = \{ \fp \in \Spec^h(A) \mid \Lambda^{\fp}M \not \simeq 0 \}.\] 
If $\cD \subseteq \Mod_R$ is a subcategory, then we set $\supp_R(\cD) = \bigcup_{M \in \cD}\supp_R(M)$, and similarly for $\cosupp_R(\cD)$.  
\end{defn}

Support and cosupport give useful characterizations of the images of $\Gamma_{\cV}$ and $\Lambda^{\cV}$. For the following, see \cite[Cor.~5.7]{benson_local_cohom_2008} and \cite[Cor.~4.8]{bik12}. 

\begin{prop}[Benson--Iyengar--Krause]
	Let $\cV \subseteq \Spec^h(A)$ be specialization closed subset and $M \in \Mod_R$. Then $M$ is in $ \Gamma_{\cV}\Mod_R$ if and only $\supp_R(M) \subseteq \cV$, and $M$ is in $\Lambda^{\cV}\Mod_R$ if and only if $\cosupp_R(M) \subseteq \cV$. 
\end{prop}

The next result shows that support and cosupport detect zero objects, and is proven in \cite[Thm.~2]{benson_local_cohom_2008} and \cite[Thm.~4.5]{bik12}.

\begin{thm}[Benson--Iyengar--Krause]\label{thm:cosupptrivialobjects}
	For any $M \in \Mod_R$, we have $M \simeq 0$ if and only if $\supp_R(M) = \varnothing$ if and only if $\cosupp_R(M) = \varnothing$. 
\end{thm}

Finally, we summarize some of the properties of the local cohomology and local homology functors $\Gamma_{\fp}$ and $\Lambda^{\fp}$ that will be used repeatedly below. From now on, we will restrict to the case that $A = \pi_*R$ acts canonically. 

\begin{prop}\label{prop:orthogonality}
Let $\fp, \fq \in \Spec^h(\pi_*R)$ be two homogeneous prime ideals.
	\begin{enumerate}
		\item For any object $M \in \Mod_R$ we have:
			\[
			\supp_R(\Gamma_{\fp}M) \subseteq \{\fp\} \quad \text{and} \quad \cosupp_R(\Lambda^{\fp}M) \subseteq \{\fp\}.
			\]
			If $M = R$, then $\supp_R(R) = \Spec^h(\pi_*R)$ and thus $\supp_R(\Gamma_{\fp}R) = \{\fp\}$. 
		\item The functor $\Gamma_{\fp}$ is left adjoint to $\Lambda^{\fp}$. 
		\item The following orthogonality relations are satisfied:
\[
\Gamma_{\fp}\Gamma_{\fq} \simeq 
\begin{cases}
\Gamma_{\fp} & \text{if } \fp = \fq, \\
0 & \text{otherwise,}
\end{cases}
\quad \text{and} \quad
\Lambda^{\fp}\Lambda^{\fq} \simeq 
\begin{cases}
\Lambda^{\fp} & \text{if } \fp = \fq, \\
0 & \text{otherwise.}
\end{cases}
\]
	\end{enumerate}
\end{prop}
\begin{proof}
The two inclusions of the first item are proven in \cite[Cor.~5.9]{benson_local_cohom_2008} and \cite[Sec.~5, Page 13]{bik12}, respectively, while the equalities for $M = R$ require a short argument: Combining \cite[Thm.~5.5(1)]{benson_local_cohom_2008} and \cite[Lem.~2.2(1)]{benson_local_cohom_2008}, we get equalities
\[
\supp_R(R) = \supp_R(\pi_*R) = \Spec^h(\pi_*R),
\]
hence $\Gamma_{\fp}R \not\simeq 0$ for all $\fp \in \Spec^h(\pi_*R)$. Item (2) is \cite[Sec.~4]{bik12}.

The last item is also implicitly contained in the works of Benson, Iyengar, and Krause, but we add some details for the convenience of the reader. By \Cref{thm:localfunctors}, the functors $\Gamma_{\fp}$ and $\Gamma_{\fq}$ are smashing, so it is enough to show that $\Gamma_{\fp}R \otimes_R \Gamma_{\fq}R \simeq 0$ for $\fp \neq \fq$. Furthermore, the same result and the definition of support imply that 
\[
\supp_R(M \otimes_{R} N) \subseteq \supp_R(M) \cap \supp_R(N)
\]
for any $M,N \in \Mod_R$. We may then use Part (1) to compute
\[
\supp_R(\Gamma_{\fp}R \otimes_R \Gamma_{\fq}R) \subseteq \supp_R(\Gamma_{\fp}R) \cap \supp_R(\Gamma_{\fq}R) = \{\fp\} \cap \{\fq\} = \varnothing. 
\]
Therefore, $\Gamma_{\fp}R \otimes_R \Gamma_{\fq}R \simeq 0$
 by \Cref{thm:cosupptrivialobjects}. 
 
Finally, the orthogonality relations for the local homology functors $\Lambda^{\fp}$ and $\Lambda^{\fq}$ now follow by adjunction: By Part (2), for any $M \in \Mod_R$ we have
\[
\Lambda^{\fp}\Lambda^{\fq}M \simeq \Hom_R(R,\Lambda^{\fp}\Lambda^{\fq}M) \simeq \Hom_R(\Gamma_{\fq}\Gamma_{\fp}R,M),
\]
so the claim reduces to the orthogonality relation for $\Gamma_{\fp}$ and $\Gamma_{\fq}$.
\end{proof}

\subsection{Stratification and costratification}
In order to classify the localizing subcategories of the stable module category of a finite group, Benson, Iyengar, and Krause~\cite{bik11} introduced the notion of stratification for a tensor triangulated category with an action by a graded-commutative Noetherian ring $A$. Since we will only consider the case when $A = \pi_*R$ acts canonically, we specialize their definition to this case. 
\begin{defn}\label{defn:stratification}
For a commutative ring spectrum $R$, the module category $\Mod_R$ is said to be canonically stratified if it is stratified by the canonical action of $\pi_*R$ on $\Mod_R$, i.e., if the following two conditions are satisfied:
	\begin{enumerate}
		\item The local to global principle for localizing subcategories holds, that is, for each $M \in \Mod_R$ there is an equality
		\begin{equation}\label{eq:locgloballoc}
\Loc(M) = \Loc(\{\Gamma_{\fp}M \mid \fp \in \Spec^h( \pi_*R) \}).
		\end{equation}
		\item For any homogeneous prime ideal $\fp \in \Spec^h( \pi_*R)$, the category $\Gamma_{\fp}\Mod_R$ is minimal as a localizing subcategory of $\Mod_R$ (i.e., it has no proper non-zero localizing subcategories).  
	\end{enumerate}
\end{defn}
\begin{rem}\label{rem:nonzero}
	At first this may appear to be different to the definition given in \cite[Sec.~4]{bik11} for stratification of a triangulated category $\cT$ by a commutative Noetherian ring $R$, where they require that $\Gamma_{\fp}\cT$ is either zero or minimal, however they are in fact equivalent in our case. Indeed, in the case of a canonical ring action by $\pi_*R$ on $\Mod_R$, $\Gamma_{\fp}\Mod_R$ is non-zero by \Cref{prop:orthogonality}(1).
\end{rem}

In order to classify colocalizing subcategories, we have the dual notation of costratification. 

\begin{defn}\label{defn:costratification}
For a commutative ring spectrum $R$, the module category $\Mod_R$ is said to be canonically costratified if it is costratified by the canonical action of $\pi_*R$ on $\Mod_R$, i.e., if the following two conditions are satisfied:
	\begin{enumerate}
		\item The local to global principle for colocalizing subcategories holds, that is, for each $M \in \Mod_R$ there is an equality
		\begin{equation}\label{eq:locglocalcoloc}
\Coloc(M) = \Coloc(\{\Lambda^{\fp}M \mid \fp \in \Spec^h( \pi_*R) \}).
		\end{equation}
		\item For any homogeneous prime ideal $\fp \in \Spec^h( \pi_*R)$, the category $\Lambda^{\fp}\Mod_R$ is minimal as a colocalizing subcategory of $\Mod_R$. 
	\end{enumerate}
	\end{defn}
	Once again, in this case it is not possible that $\Lambda^{\fp}\Mod_R$ is zero, by \cite[Prop.~5.4]{bik12}. Moreover, the local to global principle for stratification and costratification always holds for Noetherian ring spectra, see \cite[Thm.~7.2]{bik11} and \cite[Rem.~8.8]{bik12}:
	
\begin{thm}[Benson--Iyengar--Krause]\label{thm:localtoglobal}
If $R$ is a Noetherian commutative ring spectrum, then the local to global principle for localizing and colocalizing subcategories holds.
\end{thm}	
	
We have the following consequences of (co)stratifications. For proofs, see \cite[Thm.~4.2]{bik11} and \cite[Cor.~9.2]{bik12}.
	\begin{thm}[Benson--Iyengar--Krause] Let $R$ be a Noetherian commutative ring spectrum.
		\begin{enumerate}
			\item If $\Mod_R$ is canonically stratified, then there is a bijection between localizing subcategories of $\Mod_R$ and subsets of $\Spec^h(\pi_*R)$, which takes a subset $\cal{U} \subseteq \Spec^h(\pi_*R)$ to the full subcategory of $R$-modules whose support is contained in $\cal{U}$. 
			\item If $\Mod_R$ is canonically costratified, then there is a bijection between colocalizing subcategories of $\Mod_R$ and subsets of $\Spec^h(\pi_*R)$, which takes a subset $\cal{U} \subseteq \Spec^h(\pi_*R)$ to the full subcategory of $R$-modules whose cosupport is contained in $\cal{U}$.  
			\item If $\Mod_R$ is canonically costratified, then the map sending a localizing subcategory $\cL$ of $\Mod_R$ to its left orthogonal $\cL^{\perp} = \{M \in \Mod_R\mid \Hom_R(L,M) =0 \text{ for all } L\in\cL\}$ induces a bijection between localizing and colocalizing subcategories of $\Mod_R$.
		\end{enumerate}
	\end{thm}
	In the last result we have used the fact that if $\Mod_R$ is costratified, then it is also stratified \cite[Thm.~9.7]{bik12}.
\subsection{Consequences}
Apart from the classification of localizing and colocalizing subcategories, there are further consequences of (co)stratification that have been determined by Benson, Iyengar, and Krause, which we review here. Another application, \Cref{prop:abstractsubgroupthm}, is obtained as a consequence of both stratification and  a base-change formula for support.

	The first result is the classification of all thick subcategories of compact objects, which follows from the classification of all localizing subcategories \cite[Thm.~6.1]{bik11}. 
\begin{thm}[Benson--Iyengar--Krause]
	Suppose $\Mod_R$ is canonically stratified by $\pi_*R$. Then, there is a bijection between thick subcategories of compact $R$-modules and specialization closed subsets of $\Spec^h(\pi_*R)$, which takes a specialization closed subset $\cal{U}$ to the full-subcategory of compact $R$-modules whose support is contained in $\cal{U}$. 
\end{thm}

The following is the analog of the tensor product theorem in modular representation theory, as well as its hom-object version. Note that the inclusions $\subseteq$ in this theorem hold unconditionally as seen for the first one in the proof of \Cref{prop:orthogonality}.

\begin{thm}[Benson--Iyengar--Krause]\label{thm:tensor_hom_(co)supp}
Let $R$ be a Noetherian commutative ring spectrum and let $M$ and $N$ be $R$-modules. If $\Mod_R$ is canonically stratified, then \[
\supp_R(M \otimes_{R} N) = \supp_R(M) \cap \supp_R(N)\]
		and
		\[ \cosupp_R(\Hom_R(M,N)) = \supp_R(M) \cap \cosupp_R(N).\] 
In particular, $M \otimes_{R} N \simeq 0$ if and only if $\supp_R(M) \cap \supp_R(N) = \varnothing$, and similarly for $\Hom_R(M,N)$. 
	\end{thm}
\begin{proof}
	The first is \cite[Thm.~7.3]{bik11}, while the second is \cite[Thm.~9.5]{bik12}. The final claim is then an immediate consequence of \Cref{thm:cosupptrivialobjects}.
\end{proof}
 
Finally, by \cite[Thm.~6.3]{bik11}, stratification implies that the telescope conjecture holds for $\Mod_R$; see also \cite[Thm.~11.13]{bik_finitegroups}.
\begin{thm}[Benson--Iyengar--Krause]\label{thm:tel_conj}
	Suppose $\Mod_R$ is canonically stratified by $\pi_*R$. Let $\cal{U}$ be a localizing subcategory of $\Mod_{R}$. Then the following conditions are equivalent. 
	\begin{enumerate}
		\item The localizing subcategory $\cal{U}$ is smashing, i.e., the associated localizing endofunctor on $\Mod_R$ preserves colimits.
		\item The localizing subcategory $\cal{U}$ is generated by compact objects in $\Mod_R$. 
		\item The support of $\cal{U}$ is specialization closed. 
	\end{enumerate}
\end{thm}

\section{Base-change formulae and descent}\label{sec:basechange}
In this section we generalize the base-change formulae for support and cosupport given in \cite{bik12} to not necessarily finite morphisms of ring spectra. 
\subsection{Base-change for support and cosupport}

 We start by defining precisely what we mean by a finite morphism of ring spectra. 
\begin{defn}\label{defn:finite}
   Let $f \colon R \to S$ be a morphism of Noetherian commutative ring spectra. We say that $f$ is a finite morphism if $S$ is a compact $R$-module via $f$. 
 \end{defn}
 
 \begin{rem}\label{rem:finite_fg}
 If $f \colon R \to S$ is a finite morphism, then $\pi_*S$ is finitely generated over $\pi_*R$. This follows from the fact that in a cofiber sequence of $R$-modules, if two of them have finitely generated homotopy as $\pi_*(R)$-modules, the same holds for the third one. Since any compact $R$-module is built from finitely many cofiber sequences, it follows that its homotopy is a finitely generated $\pi_*(R)$-module (e.g. see the proof of \cite[Lem.~10.2(i)]{greenlees_hi}). But the converse fails in general.
 \end{rem}

  Under this finiteness assumption the results of Benson, Iyengar, and Krause \cite[Sec.~7]{bik12} specialize to give base-change results for support and cosupport under induction, coinduction, and restriction. The purpose of this section is to generalize these results to not necessarily finite morphisms of commutative ring spectra, which is crucial for our main theorem. We start with two results of Benson, Iyengar, and Krause, for the first see {\cite[Prop.~7.5]{bik12}}, while the second is a consequence of {\cite[Thm.~5.6 and Prop.~6.1]{benson_local_cohom_2008}} combined with \cite[Prop.~4.7]{bik12}.

\begin{lem}\label{lem:7.5}
Let $A$ be a graded commutative Noetherian ring, and $\alpha\colon A \to \pi_*S$ be a ring homomorphism with induced map $\alpha^*\colon \Spec^h(\pi_*S) \to \Spec^h(A)$, so that we may consider the induced $A$-linear action on $\Mod_S$. If $\cV \subseteq \Spec^h(A)$ is specialization closed and $\cW = (\alpha^*)^{-1}\cV$, then $\Gamma_{\cV}^S \simeq \Gamma_{\cW}^S$ and $L_{\cV}^S \simeq L_{\cW}^S$. 
\end{lem}

\begin{lem}\label{lem:6.1}
Let $\cV, \cW \subseteq \Spec^h(\pi_*S)$ be specialization closed and set $\cU = \cV \setminus \cW$. If $\fq \in \cU$, then $\Gamma_{\fq}^SL_{\cW}^S\Gamma_{\cV}^S \simeq \Gamma_{\fq}^S$, and zero otherwise. In particular, given $M \in \Mod_S$, then $\supp_S(L_{\cW}^S\Gamma_{\cV}^S M)=\supp_S(M)\cap \cU$. Dually, $\cosupp_S(\Lambda^{\cV}\Delta^{\cW}M) = \cosupp_S(M) \cap \cU$.
\end{lem}

The ring homomorphism $\pi_*f\colon \pi_*R \to \pi_*S$ induces an action of $\pi_*R$ on $\Mod_S$, so we can form the Koszul object $S \mm \fp$ for any $\fp \in \Spec^h(\pi_*R)$ (see \Cref{sec:localcohom}). 

\begin{lem}\label{lem:key}
With notation as above, there is an equivalence $\Ind(R\mm\fp) \simeq S\mm\fp$ of $S$-modules. 
\end{lem}
\begin{proof}
Because induction is symmetric monoidal, by construction of $R\mm\fp$ we can reduce to the case that $\fp$ is principal on an element $x\in\pi_nR$. Let $y=(\pi_nf)(x)\in\pi_nS$ and write $y\colon  S \to\Sigma^{-n} S$ for the corresponding $S$-module map. Explicitly, viewing $x$ as a map $R \to \Sigma^{-n}R$, the map $y$ is given by 
\[
\xymatrix{S \otimes_R R \ar[r]^-{1\otimes x} & S \otimes_R \Sigma^{-n}R \ar[r]^-{1\otimes \Sigma^{-n}f} & S\otimes_R\Sigma^{-n}S\simeq \Sigma^{-n}(S\otimes_RS)\ar[r]^-{\Sigma^{-n}\mu} & \Sigma^{-n}S,}
\]
where $\mu\colon S\otimes_R S \to S$ denotes the multiplication on $S$. It suffices to show that there is an equivalence $\Ind(R\mm(x)) \simeq S\otimes_R R\mm(x) \simeq S\mm(y)$ as $S$-modules. Consider the diagram of $R$-modules
\[
\xymatrix{S\otimes_R R \ar[r]^-{1\otimes x}\ar[d]_{\simeq} & S\otimes_R \Sigma^{-n}R \ar[r]\ar[d]|{\Sigma^{-n}\mu \circ (1\otimes \Sigma^{-n}f)} & S\otimes_R R\mm(x) \ar@{-->}[d]^{\phi}\\
S \ar[r]_{y} & \Sigma^{-n}S\ar[r] & S\mm(y).}
\]
Notice that the square on the left commutes by construction of $y$. Furthermore, both the left and the middle vertical maps in the diagram are equivalences as $f$ is the unit of $S$. It follows that the induced map $\phi\colon S\otimes_R R\mm(x)\to S\mm(y)$ is an equivalence. To conclude that $\phi$ is an equivalence of $S$-modules, observe that in fact all the maps in the diagram above are $S$-module maps.
\end{proof}

\begin{lem}\label{lem:inducedaction}
With respect to the induced $\pi_*R$-linear action on $\Mod_S$, there are equivalences
\[
\Ind\Gamma_{\cV}^R \simeq \Gamma_{\cV}^S\Ind \quad \text{and} \quad \Ind L_{\cV}^R \simeq L_{\cV}^S\Ind
\]
for any specialization closed subset $\cV \subseteq \Spec^h(\pi_*R)$. 
\end{lem}
\begin{proof}
Consider the fiber sequence
\[
\xymatrix{\Ind\Gamma_{\cV}^R \ar[r] & \Ind \ar[r] & \Ind L_{\cV}^R.}
\]
By the construction of $\Gamma_{\cV}^S$ and $L_{\cV}^S$ as in \cite[Sec.~4]{benson_local_cohom_2008}, it suffices to show that $\Ind\Gamma_{\cV}^RM \in \Gamma_{\cV}^S\Mod_S$ and $\Ind L_{\cV}^RM \in L_{\cV}^S\Mod_S$ for all $M \in \Mod_R$: Indeed, consider the following solid diagram of fiber sequences
\[
\xymatrix{\Ind\Gamma_{\cV}^R \ar[r] \ar@{-->}[d]_{\alpha} & \Ind \ar[r] \ar[d]^{=} & \Ind L_{\cV}^R \ar@{-->}[d]^{\beta} \\
\Gamma_{\cV}^S\Ind \ar[r] & \Ind \ar[r] & L_{\cV}^S\Ind.}
\]
If $\Ind\Gamma_{\cV}^RM \in \Gamma_{\cV}^S\Mod_S$, then the canonical composite $\Ind \Gamma_{\cV}^R \to \Ind \to L_{\cV}^S\Ind$ is trivial, as there are no non-trivial maps from an object in $\Gamma_{\cV}^S\Mod_S$ to an object in $L_{\cV}^S\Mod_S$, see \Cref{rem:orthogonal}. It then follows from the universal property of the bottom fiber sequence that there exists a map $\alpha$, which in turn induces a map $\beta$ making the above diagram commute.

Applying $L_{\cV}^S$ to this diagram then induces an equivalence $L_{\cV}^S\beta\colon L_{\cV}^S\Ind L_{\cV}^R \to L_{\cV}^SL_{\cV}^S\Ind \simeq L_{\cV}^S\Ind$. If $\Ind L_{\cV}^RM \in L_{\cV}^S\Mod_S$ as well, then this means that $L_{\cV}^S\Ind L_{\cV}^R \simeq \Ind L_{\cV}^R$, so $\beta$ itself is an equivalence. Therefore, $\alpha$ has to be an equivalence, furnishing the claim.

Recall that $\Gamma_{\cV}^R\Mod_R = \Loc(R\mm\fp\mid \fp \in \cV)$ and similarly for $\Gamma_{\cV}^S\Mod_S$. By \Cref{lem:key} and because $\Ind$ preserves colimits, \Cref{lem:preservecoloc} implies that there exists a factorization
\[
\xymatrix{\Gamma_{\cV}^R\Mod_R \ar[r] \ar@{-->}_{\Ind}[d] & \Mod_R \ar[d]^{\Ind} \\
\Gamma_{\cV}^S\Mod_S \ar[r] & \Mod_S.}
\]
Therefore, $\Ind\Gamma_{\cV}^RM \in \Gamma_{\cV}^S\Mod_S$ for any $M \in \Mod_R$. For the second claim, recall that $L_{\cV}^S\Mod_S$ is by definition the full subcategory of $\Mod_S$ spanned by the $S$-modules $Y$ with the property that $\Hom_S(X,Y) \simeq 0$ for all $X \in \Gamma_{\cV}^S\Mod_S$, see \Cref{rem:orthogonal}. Let $M \in \Mod_R$ and $\fp \in \cV$, then \Cref{lem:key} and the fact that $L_{\cV}^R$ is smashing by \Cref{thm:localfunctors} give
\begin{align*}
  \Hom_S(S\mm\fp,\Ind L_{\cV}^RM) & \simeq \Hom_S(\Ind(R\mm\fp),\Ind L_{\cV}^RM) \\
  & \simeq \Hom_R(R\mm\fp,\Res\Ind L_{\cV}^RM) \\
  & \simeq \Hom_R(R\mm\fp,S\otimes_RL_{\cV}^RM) \\
  & \simeq \Hom_R(R\mm\fp,L_{\cV}^R(S\otimes_RM)) \\
  & \simeq 0, 
\end{align*}
hence $\Hom_S(W,\Ind L_{\cV}^RM) \simeq 0$ for all $W \in \Gamma_{\cV}^S\Mod_S$ by a localizing subcategory argument. It follows that $\Ind L_{\cV}^RM \in L_{\cV}^S\Mod_S$, thereby completing the proof.
\end{proof}

A morphism $f\colon R \to S$ of commutative ring spectra induces a map $\res_f = f^*\colon \Spec^h(\pi_*S)\to\Spec^h(\pi_*R)$, which we will often denote simply by $\res$ if the map $f$ is clear from context. 

\begin{prop}\label{prop:basicbasechange}
Suppose $f\colon R \to S$ is a morphism of Noetherian commutative ring spectra. Let $\cV \subseteq \Spec^h(\pi_*R)$ and $\widetilde{\cV} = \res^{-1}(\cV) \subseteq \Spec^h(\pi_*S)$, then there are natural equivalences 
\[
\Ind\Gamma_{\cV}^R \simeq \Gamma_{\widetilde{\cV}}^S\Ind \quad \text{and} \quad \Ind L_{\cV}^R \simeq L_{\widetilde{\cV}}^S\Ind
\]
as well as
\[
\Lambda^{\cV}_R\Res \simeq \Res\Lambda_S^{\widetilde{\cV}} \quad \text{and} \quad \Delta_R^{\cV}\Res \simeq \Res\Delta_S^{\widetilde{\cV}}.
\]
\end{prop}
\begin{proof}
The first two equivalences follow from \Cref{lem:inducedaction} together with \Cref{lem:7.5}, for instance:
\[
\Ind\Gamma_{\cV}^R \simeq \Gamma_{\cV}^S\Ind \simeq \Gamma_{\widetilde{\cV}}^S\Ind.
\]
The second set of equivalences is a consequence of the first one, using the adjunctions $(\Ind \dashv \Res)$, $(\Gamma_{\cV} \dashv \Lambda^{\cV})$, and $(L_{\cV} \dashv \Delta^{\cV})$, as displayed in the following commutative diagram,
\[
\xymatrix{\Mod_R \ar@<0.5ex>[r]^-{\Ind} \ar@<0.5ex>@{-->}[d]^-{\Lambda^{\cV}} & \Mod_S \ar@<0.5ex>@{-->}[l]^-{\Res} \ar@<0.5ex>@{-->}[d]^-{\Lambda^{\widetilde{\cV}}} & & \Mod_R \ar@<0.5ex>[r]^-{\Ind} \ar@<0.5ex>@{-->}[d]^-{\Delta^{\cV}} & \Mod_S \ar@<0.5ex>@{-->}[l]^-{\Res}  \ar@<0.5ex>@{-->}[d]^-{\Delta^{\widetilde{\cV}}} \\
\Mod_R \ar@<0.5ex>[r]^-{\Ind} \ar@<0.5ex>[u]^-{\Gamma_{\cV}} & \Mod_S \ar@<0.5ex>@{-->}[l]^-{\Res} \ar@<0.5ex>[u]^-{\Gamma_{\widetilde{\cV}}} & & \Mod_R \ar@<0.5ex>[r]^-{\Ind} \ar@<0.5ex>[u]^-{L_{\cV}} & \Mod_S, \ar@<0.5ex>@{-->}[l]^-{\Res} \ar@<0.5ex>[u]^-{L_{\widetilde{\cV}}}}
\]
in which the dashed arrows indicate the right adjoints.
\end{proof}

\begin{cor}\label{prop:bcsupport}
  Let $\cV, \cW \subseteq \Spec^h(\pi_*R)$ be specialization closed subsets such that $\cV \setminus \cW = \{\fp\}$.\footnote{For example, we may take $\cV = \cV(\fp)$ and $\cW = \cZ(\fp)$.} Let $\widetilde{\cV} = \res^{-1}(\cV)$ and $\widetilde{\cW} = \res^{-1}(\cW)$, then there is an equivalence $\Ind \Gamma_\fp\simeq L_{\widetilde{\cW}}\Gamma_{\widetilde{\cV}}\Ind$. In particular, for any $N \in \Mod_S$, we have an equality 
  \[
  \supp_S(\Ind (\Gamma_\fp R) \otimes_S N)=\supp_S(N)\cap \cU,
  \]
where  $\cU := \widetilde{\cV}\setminus \widetilde{\cW} = \res^{-1}(\{\fp\})$. 
\end{cor} 
\begin{proof}
For $M \in \Mod_R$, using \Cref{rem:localcohomfunctors} and  applying \Cref{prop:basicbasechange} twice shows
\[
\Ind\Gamma_{\fp}^RM  \simeq \Ind L_{\cW}^R\Gamma_{\cV}^RM  \simeq L_{\widetilde{\cW}}^S\Gamma_{\widetilde{\cV}}^S\Ind M.
\]
As a consequence, we obtain
\[
\supp_S(\Ind(\Gamma_{\fp}R)\otimes_S N) = \supp_S((L_{\widetilde{\cW}}\Gamma_{\widetilde{\cV}}S) \otimes_S N) = \supp_S(L_{\widetilde{\cW}}\Gamma_{\widetilde{\cV}}N),
\]
because the local cohomology functors $L_{\widetilde{\cV}}$ and $\Gamma_{\widetilde{\cW}}$ are smashing. The second part of the result thus follows from \Cref{lem:6.1}.
\end{proof}

\begin{cor}\label{cor:7.8}
Let $M \in \Mod_R$ and $N \in \Mod_S$. Then, we have
\[
\res\supp_S(\Ind M) \subseteq \supp_R (M) \quad \text{and} \quad  \res\cosupp_S(N) = \cosupp_R(\Res N).
\]
Additionally, if $\Ind$ is conservative, then the first inclusion is an equality. 
\end{cor}
\begin{proof}
This is a formal consequence of \Cref{prop:basicbasechange} and the fact that $\Res$ is conservative, following the same argument as in \cite[Cor.~7.8]{bik12}. For the convenience of the reader, we outline the argument. By \Cref{prop:bcsupport}, for any $\fr \in \Spec^h(\pi_*R)$ and $M \in \Mod_R$ there are equalities
\[
\supp_S(\Ind\Gamma_{\fr}M) = \supp_S(\Ind(\Gamma_{\fr}R) \otimes_S \Ind M) = \supp_S(\Ind M) \cap \res^{-1}(\{\fr\}).
\]
Let $\fp \in \res\supp_S(\Ind M)$, then there exists $\fq \in \supp_S(\Ind M)$ with $\res(\fq) = \fp$, so $\supp_S(\Ind M) \cap \res^{-1}(\{\fp\})$ is non-empty. Therefore, $\Ind\Gamma_{\fp}M \not\simeq 0$ by \Cref{thm:cosupptrivialobjects}, hence $\fp \in \supp_R (M)$, which verifies the first claim. Assume now that $\Ind$ is conservative, then $\fp \in \supp_R (M)$ implies $\Ind\Gamma_{\fp}M \not\simeq 0$. As before, this is equivalent to the existence of a prime ideal $\fq \in \supp_S(\Ind M) \cap \res^{-1}(\{\fp\})$, thus $\fp \in \res\supp_S(\Ind M)$. 

As in the proof of \Cref{prop:bcsupport} and with notation as above, \Cref{prop:basicbasechange} gives an equivalence
\[
\Lambda^{\fp}\Res N \simeq \Res\Lambda^{\widetilde{\cV}}\Delta^{\widetilde{\cW}}N
\]
for any $\fp \in \Spec^h(\pi_*R)$. Therefore $\fp \in \cosupp_R\Res N$ holds if and only if $\Lambda^{\widetilde{\cV}}\Delta^{\widetilde{\cW}}N \not\simeq 0$ because $\Res$ is conservative. Since $ \widetilde{\cV}\setminus \widetilde{\cW} = \res^{-1}(\{\fp\})$, by \Cref{thm:cosupptrivialobjects} and \Cref{lem:6.1} this is equivalent to $\cosupp_S(N) \cap \res^{-1}(\{\fp\}) \neq \varnothing$, i.e., $\fp \in \res\cosupp_S(N)$.
\end{proof}

We say that a subset $\cU \subseteq \Spec^h(\pi_*S)$ is discrete if $\frak r \subseteq \fq$ implies $\frak r = \fq$ for each pair of primes $\fr,\fq \in \cU$. 
\begin{cor}\label{cor:inddecomposition}
  Suppose $\pi_*S$ is a finitely generated $\pi_*R$-module, and let $\cU = \res^{-1}(\{\fp\})$ for $\fp \in \Spec^h(\pi_*R)$. Then, there are equivalences
\[
  \Ind \Gamma_\fp \simeq \coprod_{\fq \in \cU} \Gamma_\fq \Ind \quad \text{ and } \quad \Lambda^\fp \Res \simeq \prod_{\fq \in \cU} \Res \Lambda^\fq.
\]
\end{cor}
\begin{proof}
 Given \Cref{prop:basicbasechange}, this is the same argument as for \cite[Cor.~7.10]{bik12}, where our assumption on $\pi_*S$ implies that the subset $\cU \subseteq \Spec^h(\pi_*S)$ is discrete. 
\end{proof}

\begin{lem}\label{lem:cosuppcoind}
For $M \in \Mod_R$, there is an inclusion
\[
\res\cosupp_S (\Coind M) = \cosupp_R(\Res\Coind M) \subseteq \cosupp_R(M).
\]
If the map $R \to S$ admits an $R$-module retract, then this is an equality. 
\end{lem}
\begin{proof}
The first equality is \Cref{cor:7.8}. By construction, $\Res\Coind M \simeq \Hom_R(S,M) \in \Coloc(M)$, which yields the inclusion. Assuming $R \to S$ is split, the reverse inclusion is a consequence of \Cref{lem:retract}.
\end{proof}

We are now ready to prove the key base-change formula.

\begin{prop}[Base-change formula for support]\label{thm:suppbasechange}
Suppose $f\colon R\to S$ is a morphism of commutative ring spectra, then $\supp_R(\Res N) = \res\supp_S(N)$ for any $N\in\Mod_S$. 
\end{prop}
\begin{proof}
By the projection formula \Cref{lem:projformula} and using that $\Gamma_{\fp}$ is smashing, there are equivalences
\[
\Gamma_{\fp}\Res N \simeq \Gamma_{\fp}R \otimes_R \Res N \simeq \Res(\Ind(\Gamma_{\fp}R) \otimes_S N). 
\]
Therefore, a homogeneous prime ideal $\fp \in \Spec^h(\pi_*R)$ is in $\supp_R(\Res N)$ if and only if $0 \not\simeq \Ind(\Gamma_{\fp}R)\otimes_S N$. \Cref{prop:bcsupport} and \Cref{thm:cosupptrivialobjects} show that this is equivalent to
\[
\varnothing \neq \supp_S(\Ind(\Gamma_{\fp}R)\otimes_S N) = \supp_S(N) \cap \res^{-1}(\{\fp\}),
\]
i.e., $\fp \in \res\supp_S(N)$.
\end{proof}

\begin{cor}\label{cor:1}
Suppose that $f \colon R \to S$ is a morphism of Noetherian commutative ring spectrum. For any $N \in \Mod_S$, we have $\Gamma_{\fp}^R\Res N \in \Loc(\Res\Gamma_{\fq}^SN\mid \res(\fq) = \fp)$. 
\end{cor}
\begin{proof}
By \Cref{thm:suppbasechange}, for all $\fq \in \Spec^h(\pi_*S)$ we have 
\[
\supp_R(\Res \Gamma^S_{\fq}N) = \res \supp_S(\Gamma^S_{\fq}N) \subseteq \res(\fq)
\] 
and hence by \cite[Cor.~5.9]{benson_local_cohom_2008} for any $\fp \in \Spec^h(\pi_*R)$,
\[
\Gamma^R_{\fp}\Res\Gamma^S_{\fq}N \simeq 
\begin{cases}
  \Res \Gamma^S_{\fq}N & \text{if } \res(\fq) = \fp, \\
  0 & \text{otherwise}. 
\end{cases}
\]
The local to global principle for $\Mod_S$ (\Cref{thm:localtoglobal}) gives $N \in \Loc(\Gamma_{\fq}^S N \mid \fq \in \Spec^h(\pi_*S))$. The fact that $\Res$ and $\Gamma^R_{\fp}$ preserve colimits then implies by \Cref{lem:preservecoloc} that 
\[
\begin{split}
\Gamma^R_{\fp}\Res N &\in \Loc(\Gamma^R_{\fp}\Res \Gamma_{\fq}^S N \mid \fq \in \Spec^h(\pi_*S)) \\
&= \Loc(\Res\Gamma_{\fq}^SN\mid \res(\fq) = \fp),
\end{split}
\]
as required.
\end{proof}

Under the assumption that the base category is canonically stratified, we can use \Cref{thm:suppbasechange} to prove an abstract version of the subgroup theorem: Recall that as a consequence of stratification, for a subgroup $H$ of a finite group $G$, Benson, Iyengar, and Krause \cite[Thm.~11.2]{bik_finitegroups} establish a precise relationship between the support of $M \in \StMod(kG)$ and the support of the restriction of $M$ in $\StMod(kH)$. Similarly, in \cite[Thm.~11.11]{bik12}, they proved a cosupport version of this theorem. These results can be generalized in the following way, and will give rise to a version of the subgroup theorem for a large class of $p$-local compact groups. 

\begin{prop}\label{prop:abstractsubgroupthm}
Let $f\colon R \to S$ be a morphism of Noetherian commutative ring spectra and suppose $\Mod_R$ is canonically stratified. If $M \in \Mod_R$, then 
\[
\supp_{S}(\Ind M) = \res^{-1}\supp_{R}(M) \quad \mathrm{and} \quad  \cosupp_S(\Coind M) = \res^{-1}\cosupp_R (M).
\] 
\end{prop}
\begin{proof}
Let $\fq\in \Spec^h(\pi_*S)$ and $\fp = \res(\fq)$. By \Cref{thm:tensor_hom_(co)supp}, stratification of $\Mod_R$ implies that 
\[
\supp_R(\Res(\Gamma_{\fq}S)\otimes_RM) = \supp_R(\Res\Gamma_{\fq}S) \cap \supp_R(M).
\]
Because \Cref{thm:suppbasechange} and \Cref{prop:orthogonality}(1) give  equalities
\[
\supp_{R}(\Res \Gamma_{\fq}S) = \res\supp_S(\Gamma_{\fq}S)  = \{\res(\fq)\}= \{ \fp \},
\]
using the projection formula we see that 
\[
\Res\Gamma_{\fq}\Ind M \simeq \Res(\Gamma_{\fq}S)\otimes_RM \not\simeq 0
\]
if and only if $\fp \in \supp_R(M)$. Since $\Res$ is conservative, this is equivalent to $\fq \in \supp_S(\Ind M)$.

For the second claim, let $M$ be an $R$-module. By adjunction, for any $\fq \in \Spec^h(\pi_*S)$ there are equivalences
  \begin{align*}
    \Lambda^{\fq}\Coind M & \simeq \Hom_S(S,\Lambda^{\fq}\Coind M) \\
    & \simeq \Hom_S(\Gamma_{\fq}S,\Coind M) \\
    & \simeq \Hom_R(\Res\Gamma_{\fq}S,M).
  \end{align*}
Since $\Mod_R$ is stratified, it thus follows from \Cref{thm:tensor_hom_(co)supp} that $\fq \not\in \cosupp_S(\Coind M)$ if and only if 
\begin{equation}\label{eq:intersectioncondition}
\supp_R(\Res\Gamma_{\fq}S) \cap \cosupp_R(M) = \varnothing.
\end{equation}
The base-change formula \Cref{thm:suppbasechange} gives $\supp_R(\Res\Gamma_{\fq}S) = \{\res(\fq)\}$, so \eqref{eq:intersectioncondition} is in turn equivalent to $\res(\fq) \not\in \cosupp_R(M)$.
\end{proof}

\subsection{Abstract descent along finite morphisms}\label{ssec:finitedescent}

As a first application of our methods, in this subsection we will consider the case of descent of stratification along a finite morphism $f\colon R \to S$. This provides, in essence, an abstract version of the main results of \cite{bik_finitegroups} and \cite{bik12}. To that end, we will introduce a technical property which is the key ingredient needed in the arguments of \cite{bik_finitegroups} to descend stratifications, as we will see shortly. In all of our examples, this condition follows from a strong form of Quillen stratification for $\Spec^h(\pi_*R)$.

\begin{defn}\sloppy\label{defn:qlifting}
A morphism of Noetherian commutative ring spectra $f\colon R \to S$ is said to satisfy Quillen lifting if for any two modules $M,N \in \Mod_R$ such that there is $\medmuskip=2mu \fp \in \res\supp_S(\Ind M) \cap \res\cosupp_S(\Coind N)$, there exists a homogeneous prime ideal $\fq \in \res^{-1}(\{\fp\})$ with $\medmuskip=2mu \fq \in \supp_S(\Ind M) \cap \cosupp_S(\Coind N)$. 
\end{defn}

Here is the primordial example of a morphism of ring spectra satisfying Quillen lifting; we will exhibit other examples in \Cref{prop:quillenliftingfusion}.  

\begin{ex}\label{ex:groups}
It is a consequence of the strong form of Quillen stratification for group cohomology~\cite{quillen_stratification} that the morphism 
\[
\xymatrix{C^*(BG,k) \ar[r] &  \prod \limits_{E \in \cE(G)}C^*(BE,k)}
\]
satisfies Quillen lifting, see \Cref{prop:quillenliftingfusion}. Here, $G$ is a compact Lie group, $k$ is a field of characteristic $p$, and $\cE(G)$ is a set of representatives of conjugacy classes of elementary abelian $p$-subgroups of $G$. 
\end{ex}

Suppose that $f \colon R \to S$ is a finite morphism in the sense of \Cref{defn:finite}. Using, for example, \cite[Thm.~1.7]{bds_wirth}, this is equivalent to the existence of a left adjoint to $\Ind$ and a right adjoint to $\Coind$. Moreover, their results give

\begin{lem}\label{lem:ambichouinard}
If $f\colon R \to S$ is a finite morphism of commutative ring spectra, then $\Ind$ is conservative if and only of $\Coind$ is conservative.
\end{lem}
\begin{proof}
By \cite[Eqns.~(3.4), (3.9)]{bds_wirth} and setting $\omega_f = \Coind(R)$, we have natural equivalences
\[
\xymatrix{\Coind \simeq \omega_f \otimes_S \Ind & \Ind \simeq \Hom_S(\omega_f,\Coind).}
\]
This implies that $\Ind(M) \simeq 0$ if and only if $\Coind(M) \simeq 0$ for any $M \in \Mod_R$.
\end{proof}

If $f\colon R \to S$ is a finite morphism which additionally satisfies $\Ind \simeq \Coind$, then as explained in \cite{bg_stratifyingcompactlie} the proof of \cite[Thm.~9.7]{bik_finitegroups} generalizes to give the next result. Conversely, applied to the morphism of \Cref{ex:groups} for a finite $p$-group $G$, we recover the situation studied in \cite{bik_finitegroups}. In the next subsection, we will see how our methods allow us to remove the finiteness hypothesis on the morphism, see \Cref{thm:stratdescent}.

\begin{prop}\label{prop:finitedescentstrat}
Let $f\colon R \to S$ be a finite morphism of Noetherian commutative ring spectra satisfying Quillen lifting and such that $\Ind$ or $\Coind$ are conservative. If $\Mod_S$ is canonically stratified, then $\Mod_R$ is canonically stratified as well. 
\end{prop}
\begin{proof}
By virtue of \Cref{lem:ambichouinard}, this is a special case of \Cref{thm:stratdescent}. 
\end{proof}

Turning to costratifications, the next result is an abstract version of \cite[Thm.~11.10]{bik12}, whose proof can be significantly simplified by virtue of our \Cref{prop:abstractsubgroupthm}. Before we can give the proof, we need one more construction. Recall that for any graded injective $\pi_*R$-module $I$ and compact object $C \in \Mod^\omega_R$ there is an $R$-module $T_C(I)$ with the property that 
\[
\pi_*\Hom_R(-,T_C(I)) \cong \Hom_{\pi_*R}(\pi_*\Hom_R(C,-),I), 
\]
see \cite[Sec.~6]{hps_axiomatic}, \cite{bk_pureinjectives}, or \cite[Sec.~4.1]{bhv2} for example. For $\fp \in \Spec^h(\pi_*R)$, let $I_{\fp}$ denote the injective hull of $\pi_*R/\fp$ in $\Mod_{\pi_*R}$. For simplicity, let us write $T_{\fp}I_{\fp}$ for $T_{\kos{R}{\fp}}(I_\fp)$. By \cite[Prop.~5.4]{bik12} this is a perfect cogenerator of $\Lambda^{\fp}\Mod_R$. In particular, $\Coloc(T_{\fp}I_{\fp}) = \Lambda^{\fp}\Mod_R$, see \cite[Prop.~5.2]{bik12}. 

\begin{prop}\label{prop:finitedescentcostrat}
	Suppose $f \colon R \to S$ is a finite morphism of Noetherian commutative ring spectra such that $\Ind$ or $\Coind$ are conservative. If $\Mod_S$ is canonically costratified and $\Mod_R$ is canonically stratified, then $\Mod_R$ is canonically costratified. 
\end{prop}
\begin{proof}
Let $\fp \in \Spec^h(\pi_*R)$, and let $M \in \Lambda^{\fp}\Mod_R$ be non-zero. Note that \Cref{prop:abstractsubgroupthm} and \Cref{prop:orthogonality}(1) give
	\[
	\cosupp_S(\Coind M) =  \res^{-1}\cosupp_R(M) = \res^{-1}(\{\fp\}),
	\]
so that $\Lambda^{\fq}(\Coind M) \not \simeq 0$ for each $\fq \in \Spec^h(\pi_*S)$ with ${\res(\fq) = \fp}$. 

	Now, as in \cite[Prop.~5.4]{bik12} we have $\cosupp_S(T_{\fq}I_{\fq}) = \{ \fq \}$. It follows from the paragraph above and costratification of $\Mod_S$ that for all $\fq \in \res^{-1}(\{\fp\})$:
\[
T_{\fq}I_{\fq} \in \Coloc(\Lambda^{\fq}(\Coind M)) \subseteq \Coloc(\Coind M).
\]
Since $\Res$ preserves limits, using \Cref{lem:preservecoloc} and \Cref{lem:cosuppcoind} we have
\[
\Res T_{\fq}I_{\fq} \in \Coloc(\Res \Coind M) \subseteq \Coloc(M).
\]
But the collection of objects $\{ \Res T_{\fq}I_{\fq} \mid \res(\fq) = \fp \}$ cogenerates the category $\Lambda^{\fp}\Mod_R$ by \cite[Prop.~5.4 and Lem.~7.12]{bik12}, so
\[
\Coloc(M) \subseteq \Lambda^{\fp}\Mod_R = \Coloc(\{ \Res T_{\fq}I_{\fq} \mid \res(\fq) = \fp \}) \subseteq \Coloc(M),
\]
where the middle equality uses \cite[Prop.~5.2]{bik12}. Therefore, $\Coloc(M) = \Lambda^{\fp}\Mod_R$ and thus $\Lambda^{\fp}\Mod_R$ is a minimal colocalizing subcategory. Since the local to global principle holds by \Cref{thm:localtoglobal}, $\Mod_R$ is canonically costratified.
\end{proof}

\begin{cor}
Suppose $f \colon R \to S$ is a finite morphism of Noetherian commutative ring spectra satisfying Quillen lifting and such that $\Ind$ or $\Coind$ are conservative. If $\Mod_S$ is canonically costratified, then $\Mod_R$ is canonically costratified. 
\end{cor}
\begin{proof}
This is an immediate consequence of \Cref{prop:finitedescentstrat} and \Cref{prop:finitedescentcostrat}. 
\end{proof}

\begin{rem}[Linearity of $\Coind$]
We claim that, if $R \to S$ is finite, then $\Coind$ is a linear functor $(\Mod_R,\pi_*R) \to (\Mod_S, \pi_*S)$ in the sense of \cite[Sec.~7]{bik12}. As a consequence, the proof of costratification given in \cite{bik12} also directly generalizes to give \Cref{prop:finitedescentcostrat}, where $\Coind$ plays the role of the functor they denote by $(-)\downarrow$.  

In order to verify the claim, must show that for any $R$-module $M$ the diagram
\[
\xymatrix@C=2cm{
\pi_*R \cong \End_R^*(R) \ar[d]_{M \otimes_R -} \ar[r]^-{\Ind} & \pi_*S \cong \End_S^*(S) \ar[d]^{\Coind (M) \otimes_S -} \\
\End_R^*(M) \ar[r]^-{\Coind} & \End_S^*(\Coind M)
}
\]
commutes, or in other words, for a morphism of $R$-modules $\alpha \colon \Sigma^iR \to R$, we must show that that 
\[\Sigma^i\Coind M \xr{\Coind(M \otimes_R \alpha)}  \Coind M\] 
is equivalent to 
\[
\Sigma^i \Coind (M) \otimes_S S \xr{\Coind (M) \otimes_S \Ind(\alpha)} \Coind (M) \otimes_S S. 
\]
Because of the finiteness assumption, we have a natural equivalence $\Coind(M \otimes_R N) \simeq \Coind(M) \otimes_S \Ind(N)$, see \cite[(3.6)]{bds_wirth}. It follows that $\Coind(M \otimes_R \alpha) \simeq \Coind(M) \otimes_S \Ind(\alpha)$, as required.
\end{rem}

If $G$ is a compact Lie group, then one would like to apply our abstract results to the morphism
\[
\xymatrix{f \colon C^*(BG,\F_p) \ar[r] & \prod \limits_{E \leq S} C^*(BE,\F_p),} 
\]
where $E$ runs over the collection of elementary abelian $p$-subgroups of $G$. Hence, we need to know about stratification and costratification for module spectra over $\Mod_{C^*(BE,\F_p)}$. This has previously been proved by Benson--Iyengar--Krause.
\begin{prop}[Benson--Iyengar--Krause]\label{prop:elementary}\sloppy
Let $E$ be an elementary abelian group, then $\Mod_{C^*(BE,\F_p)}$ is canonically stratified and costratified.
\end{prop}
\begin{proof}
By \cite[Thm.~4.2]{krausebenson_kg} and \cite[Thm.~7.8]{krausebenson_kg}, $\Mod_{C^*(BE,\F_p)}$ is equivalent as a symmetric monoidal stable category to $K(\Inj \F_p{E})$, the category of complexes of injective $\F_p{E}$-modules with homotopy equivalences as weak equivalences. Moreover, this equivalence is compatible with the canonical actions of $H^*(BE,\F_p) \cong \pi_{-*}C^*(BE,\F_p)$ on both sides. The stratification result then follows from \cite[Thm.~8.1]{bik_finitegroups}, while the costratification follows from \cite[Thm.~11.6]{bik12}.
\end{proof} 

\begin{rem}
If $G$ is a compact Lie group with $\pi_0G$ a finite $p$-group, then the arguments of \cite{bg_stratifyingcompactlie} show that the relevant induction and coinduction functors are naturally equivalent up to an invertible twist, so that the proof of costratification given in \cite{bik12} immediately yields canonical costratification for $\Mod_{C^*(BG,\F_p)}$ as well. However, in general, $f\colon C^*(BG,\F_p) \to C^*(BE,\F_p)$ is not finite, so that the corresponding induction functor does not admit a left adjoint. For example, as explained in \cite[Ex.~10.7]{greenlees_hi}, the homomorphism $C^*(BA_4,\F_2) \to C^*(B(\Z/2 \times \Z/2),\F_2)$ induced by the inclusion $\Z/2 \times \Z/2 \to A_4$ is not finite. Therefore, the methods of this subsection are insufficient to imply stratification for all compact Lie groups; the same restrictions apply to the arguments given in \cite{bg_stratifyingcompactlie}. We will return to the general case in \Cref{rem:compactliecostrat}.
\end{rem}

\subsection{Descent for stratification}
In this section, we give conditions on a morphism $f \colon R \to S$ such that we can descend a canonical stratification from $\Mod_S$ to $\Mod_R$. 

\begin{prop}\label{prop:cosuppcoind}
Assume that $\Coind$ is conservative and let $M \in \Gamma_{\fp}\Mod_R$ be a non-trivial module, then $\cosupp_S(\Coind M) \cap \res^{-1}(\{\fp\}) \neq \varnothing$.
\end{prop}
\begin{proof}
First note that, for any $N \in \Mod_S$, the non-triviality of $\Hom_S(\Gamma_{\fq}S,N) \simeq \Lambda^{\fq}N$ (see \Cref{prop:orthogonality}(2)) is equivalent to $\fq \in \cosupp_S(N)$. Therefore, for $M \in \Gamma_{\fp}\Mod_R$, it suffices to show that 
\[
\Hom_S(\Gamma_{\fq}S, \Coind M) \not\simeq 0
\]
for at least one $\fq \in \res^{-1}(\{\fp\})$. On the one hand, by local duality, there are equivalences
\begin{align*}
  \Hom_R(\Gamma_{\fp}\Res S, M) & \simeq \Hom_R(\Res S, \Lambda^{\fp}M) \\
  & \simeq \Hom_S(S, \Coind\Lambda^{\fp}M) \\
  & \simeq \Coind\Lambda^{\fp}M.
\end{align*}
Because $\supp_R(M)= \{\fp\}$, we get $\fp \in \cosupp_R(M)$ by \cite[Thm.~4.13]{bik12}, so $\Coind\Lambda^{\fp}M \not\simeq 0$ as $\Coind$ is conservative. On the other hand, \Cref{cor:1} gives that
\[
\Gamma_{\fp}\Res S \in \Loc(\Res \Gamma_{\fq}S\mid \res(\fq) = \fp).
\]
Since there is an equivalence of $R$-modules
\[
\Hom_S(\Gamma_{\fq}S, \Coind M) \simeq \Hom_R(\Res\Gamma_{\fq}S,M),
\]
$\Hom_S(\Gamma_{\fq}S, \Coind M) \simeq 0$ for all $\fq \in \res^{-1}(\{\fp\})$ would force $\Hom_R(\Gamma_{\fp}\Res S, M) \simeq 0$ as well. This contradicts $\Coind\Lambda^{\fp}M \not\simeq 0$ as proven above, hence the result follows. 
\end{proof}

\begin{thm}\label{thm:stratdescent}
Suppose that $f \colon R \to S$ is a morphism of Noetherian ring spectra satisfying Quillen lifting and such that $\Ind$ and $\Coind$ are conservative. If $\Mod_S$ is canonically stratified, then so is $\Mod_R$. 
\end{thm}
\begin{proof}
  Since the local to global principle holds for any Noetherian commutative ring spectrum by \Cref{thm:localtoglobal}, it remains to verify that $\Gamma_{\fp}\Mod_R$ is minimal for any homogeneous prime ideal $\fp \subseteq \pi_*R$. To this end, we will apply the criterion given in \cite[Lem.~4.1]{bik11}: it thus suffices to show that $\Hom_R(M,N) \not\simeq 0$ for all non-trivial modules $M,N \in \Gamma_{\fp}\Mod_R$. 

By \Cref{cor:7.8} and \Cref{prop:cosuppcoind} we have 
\[
\supp_S(\Ind M) \cap \res^{-1}(\{\fp\}) \ne \varnothing \ne \cosupp_S(\Coind N) \cap \res^{-1}(\{\fp\}). 
\]
  It follows that 
\[
\fp \in \res(\supp_S(\Ind M)) \cap \res(\cosupp_S(\Coind N)),
\]
  and so Quillen lifting yields a homogeneous prime ideal $\fq \in \res^{-1}(\{\fp\})$ such that 
\[
\fq \in \supp_S(\Ind M) \cap \cosupp_S(\Coind N). 
\]
  Since $\Mod_S$ is canonically stratified, \Cref{thm:cosupptrivialobjects} and \Cref{thm:tensor_hom_(co)supp} imply that 
\[
 0 \not\simeq \Hom_S(\Ind M,\Coind N) \simeq \Hom_R(\Res\Ind M,N) .
\]
Because $\Res\Ind M \in \Loc(M)$, it follows that $\Hom_R(M,N) \not\simeq 0$, as required.  
\end{proof}

\subsection{Descent for costratification}\label{sec:costrat}

For the following, we remind the reader of the $R$-module $T_{\fp}I_{\fp}$ from \Cref{ssec:finitedescent}. 

\begin{thm}\label{thm:costratdescent}
Suppose that $f\colon R \to S$ is a morphism of Noetherian commutative ring spectra which admits an $R$-module retract. If $\Mod_S$ is canonically costratified and $\Mod_R$ is canonically stratified, then $\Mod_R$ is also canonically costratified. 
\end{thm}
\begin{proof}
Since the local to global principle holds for Noetherian ring spectra (see \Cref{thm:localtoglobal}), it suffices to show that, given $\fp \in \Spec^h(\pi_*R)$, any non-trivial $M \in \Lambda^{\fp}\Mod_R$ satisfies $\Coloc(M) = \Lambda^{\fp}\Mod_R$. The inclusion $\subseteq$ holds because $\Lambda^{\fp}\Mod_R$ is colocalizing, so it remains to verify that a cogenerator for $\Lambda^{\fp}\Mod_R$ is contained in $\Coloc(M)$.

To this end, note that $\Lambda^{\fq}\Coind M \not\simeq 0$ for all $\fq \in \res^{-1}(\{\fp\})$ by \Cref{prop:abstractsubgroupthm}, hence the colocalizing subcategory generated by $\Lambda^{\fq}\Coind M$ coincides with $\Lambda^{\fq}\Mod_S$ for all $\fq\in\res^{-1}(\{\fp\})$, using that $\Mod_S$ is costratified. By the local to global principle, we thus obtain an inclusion
\[
\Coloc(\Lambda^{\fq}S\mid \fq \in \res^{-1}(\{\fp\})) = \Coloc(\Lambda^{\fq}\Coind M\mid \fq \in \res^{-1}(\{\fp\})) \subseteq \Coloc(\Coind M).
\]
Since $\cosupp_R(T_{\fp}I_{\fp}) = \{\fp\}$, \Cref{lem:cosuppcoind} shows that 
\[
\cosupp_S(\Coind T_{\fp}I_{\fp}) \subseteq \res^{-1}\cosupp_R(T_{\fp}I_{\fp}) = \res^{-1}(\{\fp\}),
\]
so $\Coind T_{\fp}I_{\fp} \in \Coloc(\Lambda^{\fq}S\mid \fq \in \res^{-1}(\{\fp\})) \subseteq \Coloc(\Coind M)$. It thus follows from the fact that $\Res$ preserves limits, \Cref{lem:preservecoloc}, and \Cref{lem:cosuppcoind} that
\[
\Res\Coind T_{\fp}I_{\fp}  \in \Coloc(\Res\Coind M) \subseteq \Coloc(M).
\]
The perfect cogenerator $T_{\fp}I_{\fp}$ of $\Lambda^{\fp}\Mod_R$ is a retract of $\Res\Coind T_{\fp}I_{\fp}$ by assumption on $f$, so $\Lambda^{\fp}\Mod_R \subseteq \Coloc(M)$, as desired. 
\end{proof}

\section{Homotopical groups}\label{sec:homotopicalgroups}
In this section we provide the background material on homotopical groups, introducing the concepts of $p$-compact groups, 
$p$-local finite groups, and $p$-local compact groups. We then prove a version of Chouinard's theorem for a large class of $p$-local compact groups 
by constructing a stable transfer.

\subsection{The homotopy theory of compact Lie groups and $p$-compact groups}\label{ssec:recollectionshomgps}
In order to study the homotopical properties of compact Lie groups, Rector suggested studying their classifying spaces; however, he observed that there are uncountably many distinct loop space structures on $S^3$  \cite{rector_loops}. 
It was then shown by Dwyer, Miller, and Wilkerson that this problem goes away after $p$-completion: there is a unique loop space 
structure on $\pc{(S^3)}$ \cite{dwyermillerwilk}.

This example suggests that the right category to isolate the homotopical properties of compact Lie groups is the category of 
$p$-complete loop spaces. Based on this, Dwyer and Wilkerson \cite{dwyerwilkerson_finloop} introduced the notion of a $p$-compact group. 

\begin{defn}\label{defn:pcompactgroup}
	A $p$-compact group is a triple $(X,BX,e)$ where $X$ is an $\F_p$-finite space, $BX$ is a $p$-complete space, and $e\colon X \to \Omega BX$ is an equivalence. 
\end{defn}
  
\begin{rem}\label{rem:defpcg}
	By \cite[Lem.~2.1]{dwyerwilkerson_finloop}, $BX$ is $p$-complete if and only if $X$ is $p$-complete and $\pi_0X$ is a finite $p$-group.  
\end{rem}

The notion of a $p$-compact group gives a homotopy theoretic version of connected compact Lie groups. If $G$ is a compact Lie group with $\pi_0G$ a finite $p$-group , then $(\pc{G},\pc{BG},\pc{e_G})$ is a $p$-compact group (see \Cref{rem:defpcg}). However, there are exotic examples of $p$-compact groups which are not the $p$-completion  of any compact Lie group. Connected $p$-compact groups were classified in \cite{AG_classification2} and \cite{AGMV_classification1}. 

Most of the geometric structure of a connected compact Lie group can be translated into this homotopy theoretic setting, see \cite{dwyerwilkerson_finloop}. A homomorphism $f \colon X \to Y$ of $p$-compact groups is a pointed map $Bf \colon BX \to BY$, and $f$ is said to be a monomorphism if 
$H^*(Y/f(X),\F_p)$ is finite where $Y/f(X)$ denotes the homotopy fiber of $Bf$ (equivalently, if $H^*(BX,\F_p)$ is finitely generated as a $H^*(BY,\F_p)$-module, 
see \cite[Prop.~9.11]{dwyerwilkerson_finloop}). In the latter case we say that the pair $(X,f)$ is a subgroup of $Y$, and we 
write $X \leq_f Y$.  

 A $p$-compact torus $T$ is the $p$-completion of a torus and a $p$-compact toral group $\widehat{S}$ is a $p$-compact group  that is a finite extension of a $p$-compact torus by a finite $p$-group. Especially relevant is the fact that $p$-compact groups admit a maximal torus,  an associated Weyl group, and a maximal $p$-compact toral subgroup (see \cite{dwyerwilkerson_finloop}).  We outline the description of $\widehat{S}$ in terms of the structure of the $p$-compact group.  In \cite{dwyerwilkerson_finloop}, the authors show that every $p$-compact group $\cG$ has a maximal torus $T_{\cG} \simeq K(\Z_p^{n},1)$ with classifying space $BT_{\cG}\simeq K(\Z_p^n,2)\simeq K((\Z/p^\infty)^n,1)^\wedge_p$, i.e., there is a homomorphism of $p$-compact groups $T_{\cG} \to \cG$ such that the homotopy fiber $\cG/T_{\cG}$ of  $BT_{\cG} \to B\cG$ is $\F_p$-finite. The rank of $\cG$ is defined then to be $n$. Replacing this map with an equivalent fibration, the Weyl space $W_{\cG} = W_{\cG}(T_{\cG})$ is defined as the space of all fiber maps over the identity, i.e., $(\cG/T_{\cG})^{hT_{\cG}}$.  By \cite[Prop.~9.5]{dwyerwilkerson_finloop} this is a homotopically discrete monoid, and the Weyl group $\pi_0W_{\cG}$ is a finite group whose order is the Euler characteristic  $\chi(\cG/T_{\cG})$ (see \cite[Sec.~4.3]{dwyerwilkerson_finloop}).

The normalizer of the maximal torus $N(T_{\cG})$ is the loop space whose classifying space $BN(T_{\cG})$ is the Borel construction of the action of $W_{\cG}$ on $BT_{\cG}$. The $p$-normalizer is then the loop space of $BN_p(T_{\cG})$ which is obtained by applying the Borel construction with respect to the action of $(W_{\cG})_p$, the union of components of $W_{\cG}$ corresponding to a $p$-Sylow subgroup of $W_{\cG}$. This fits into a fibration sequence
\[
BT_{\cG} \to BN_p(T_{\cG}) \to B(W_{\cG})_p.
\]

Then $\widehat{S}=N_p(T_{\cG})$ is a $p$-compact toral group being an extension of a $p$-complete torus by a finite $p$-group, and it is  the $\F_p$-completion of a discrete $p$-toral group $S$---that is, a discrete group which contains a normal subgroup of the form $(\Z/p^\infty)^r$ of $p$-power index, for finite $r \ge 0$, \cite[Prop.~6.9]{dwyerwilkerson_finloop}.  In \cite{dwyerwilkerson_finloop}, Dwyer and Wilkerson prove that $\widehat{S}=S^\wedge_p$ is a maximal $p$-compact toral subgroup and we call it a $p$-compact toral Sylow subgroup of $\cG$.

As noted, Lie groups in general do not give rise to $p$-compact groups unless $\pi_0(G)$ is a $p$-group. When $G$ is finite, based on the observation that the homotopy type of $\pc{BG}$ is strongly determined by the conjugacy data of $p$-subgroups in $G$, Broto, Levi, and Oliver \cite{blo_fusion} introduced the notion of a $p$-local finite group. Later, Broto, Levi, and Oliver \cite{blo_pcompact} defined the more general notion of a $p$-local compact group in order to recover the homotopy type of compact Lie groups. 

To define these, recall that a saturated fusion system $\mathcal F$ associated to a finite $p$-group $P$ is a small subcategory of the category of groups which encodes conjugacy data between 
subgroups of a fixed finite $p$-group $P$, as formalized by Puig \cite{puig_fusion}. The standard example is given by the fusion category of a finite group: 
Given a finite group $G$ with a fixed Sylow $p$-subgroup $S$, let $\mathcal F_S(G)$ be the category with 
 $\textrm{Mor}_{\mathcal F_S(G)} (P,Q) = \Hom_G (P,Q)$ for all $P,Q \leq S$, where $\Hom_G (P,Q) = \{\varphi \in \Hom(P,Q) \mid \varphi = c_g \mbox{ for some } g \in G\}$
 is the set of homomorphisms induced by subconjugation by an element in $G$.  This category $\mathcal F_S(G)$ gives rise to a saturated fusion system, see \cite[Proposition 1.3]{blo_fusion}.  Exotic examples (e.g., saturated fusion systems which are not equivalent to $\mathcal F_S(G)$ for any finite group $G$) are also known, see \cite{blo_fusion}, \cite{drv_exotic}, or \cite{lo_solomon}, for example.

In \cite{blo_pcompact} Broto, Levi, and Oliver extended the definition of fusion systems over finite $p$-groups to fusion systems defined over a discrete $p$-toral group. However, in order to recover the homotopy type of $\pc{BG}$ more structure is needed. To do so, Broto, Levi, and Oliver axiomatized a new category $\mathcal L$, the centric linking system,  containing the information needed to construct the classifying space.  Packing all this information together, we arrive at the definition of a $p$-local compact group. 

 \begin{defn}[Broto--Levi--Oliver]
 	A $p$-local compact group is a triple $\cG = (S,\cF,\mathcal{L})$ where $\cF$ is a saturated fusion system over a discrete $p$-toral group $S$ and $\mathcal{L}$ is a centric linking system associated to $\cF$. The classifying space $B\cG$ is defined as $|\mathcal{L}|^{\wedge}_{p}$, the $p$-completion of the nerve of the associated centric linking system. 
 \end{defn} 
 
A $p$-local finite group is then the special case where $S$ is a finite $p$-group. More recently, Chermak \cite{chermak_existence} proved the existence and uniqueness of centric linking systems associated to saturated fusion systems over a finite $p$-group (see also \cite{boboliver_existence}) and Levi--Libman \cite{ll_uniqueness} in general, and therefore there is a unique $p$-complete classifying space associated to a saturated fusion system over a discrete $p$-toral group. Hence, we will often denote a $p$-local compact group either as a pair $\cG = (S,\cF)$, or even just as $\cG$.  
 
 Among all the structure associated to a centric linking system, there is a canonical morphism $\theta \colon BS \to B\cG$, see the discussion on p.~826 of \cite{blo_fusion}, which can be thought of as the inclusion of a Sylow $p$-subgroup.  
 
\begin{ex}\label{rem:pcompactmodel}
For any compact Lie group $G$, there exists a maximal discrete $p$-toral subgroup $S$, and a $p$-local compact group $\cG = (S,\cF_S(G),\mathcal{L}_S(G))$ with $|\cal{L}_{S}(G)|^\wedge_{p}\simeq BG^\wedge_{p}$ \cite[Thm.~9.10]{blo_pcompact}. We say that the $p$-local compact group models the compact Lie group $G$. 
\end{ex}

\begin{ex}\label{rem:plocalmodels}
    Given a $p$-compact group $X$, a discrete $p$-toral subgroup is a pair $(S,f)$ where $S$ is a discrete $p$-toral subgroup and $f\colon \widehat{S} \to X$ is a monomorphism in the sense of Dwyer--Wilkerson.  By \cite{dwyerwilkerson_finloop} any $p$-compact group has a maximal discrete $p$-toral subgroup which corresponds to a discrete approximation of a maximal $p$-compact toral subgroup. We recall that a discrete approximation to a $p$-compact toral group $\widehat{S}$ is a pair $(S,f)$ where $S$ is a discrete $p$-toral group and $Bf \colon BS \to B\widehat{S}$ a morphism which induces an isomorphism in mod $p$ cohomology. Such discrete approximations always exist \cite[Prop.~6.9]{dwyerwilkerson_finloop}. Conversely, every discrete $p$-toral group $S$ is a discrete approximation of $(\Omega(BS^{\wedge}_p),BS^{\wedge}_p,\text{id})$. 

Moreover, there exists a $p$-local compact group $\cG = (S,\cF_{S,f}(X),\cal{L}_{S,f})$ with $|\cal{L}_{S,f}|^\wedge_{p} \simeq BX$ \cite[Thm.~10.7]{blo_pcompact}. We say that the $p$-local compact group $\cG$ models the $p$-compact group $X$.  More generally, finite loop spaces at a prime $p$ can also be modeled by $p$-local compact groups \cite{BLO_finiteloopspaces}.
\end{ex} 

Finally, the following definition will be important in the sequel. 
\begin{defn}
  Given a saturated fusion system $(S,\cF)$, we say that two subgroups $P,Q \leqslant S$ are $\cF$-conjugate if they are isomorphic in the category $\cF$. 
\end{defn}
In the case that the $p$-local compact group models a compact Lie group or a $p$-compact group, the notion of $\cF$-conjugacy has a familiar interpretation.
\begin{lem}
\begin{enumerate}
    \item  If a $p$-local compact group $\cG = (S,\cF)$ models a compact Lie group $G$, then two subgroups $P,Q \leqslant S$ are $\cF$-conjugate if and only if they are conjugate as subgroups of $G$. 
  \item  Fix a $p$-compact group $X$ and a maximal discrete $p$-toral subgroup $f \colon S \to X$. If a $p$-local compact group $\cG = (S,\cF)$ models $X$, then two subgroups $P,Q \leqslant S$ are $\cF$-conjugate if and only if there exists a monomorphism $\phi \colon P \to Q$ such that $Bf\mid_{BQ}\circ B \phi \simeq Bf \mid_{BP}$.  
\end{enumerate}
 \end{lem}
\begin{proof}
Both of these simply follow from the definition of the associated fusion system. For example, by \cite[Sec.~9]{blo_pcompact} the fusion system $\cF = \cF_S(G)$ associated to a compact Lie group $G$ has 
  \[
\Mor_{\cF_S(G)}(P,Q) = \Hom_G(P,Q),
  \]
  the set of homomorphisms induced by subconjugation by elements of $G$. The fusion system associated to a $p$-compact group is given in \cite[Def.~7.1]{blo_fusion} and has the claimed form. 
\end{proof}

\subsection{Unitary embeddings}

Unitary embeddings are useful tools to reduce questions about compact Lie groups to unitary groups. For compact Lie groups, the existence of unitary embeddings is a consequence
of the Peter--Weyl Theorem. For $p$-compact groups, the statement is in \cite[Sec.~5]{cancas_finloop}. This relies on the classification of $p$-compact groups, with the embedding of exotic $p$-compact groups given in \cite{castellana_gqrn,cancas_finloop} for $p>2$ and \cite{Ziemianski20091239} for $p=2$. 

In this short subsection, we review this notion for more general spaces and the main property which we will use in the sequel.  Since we work extensively with cochains in this section, it is worth pointing out the following.

\begin{lem}\label{lem:cochaincompletion}
If $X$ is a $p$-good space, and $k$ is a field of characteristic $p >0$, then $C^*(X,k) \simeq C^*(\pc{X},k)$ for both the stable and the unstable $p$-completions of $X$. 
\end{lem}
\begin{proof}
First note that, under the assumptions on $X$, $p$-completion coincides with $H\F_p$-localization, see \cite[Ch.~VII, Prop.~2.3]{bousfield_homotopy_1972} and \cite[Thm.~3.1]{bousfield_locspectra}. The natural localization maps $X \to \pc{X}$ and $\Sigma_+^{\infty}X \to \pc{(\Sigma_+^{\infty}X)}$ induce maps on cochains
\[
\xymatrix{C^*(\Sigma_+^{\infty}(\pc{X}),k) \ar[r] & C^*(\Sigma_+^{\infty}X,k)  & C^*(\pc{(\Sigma_+^{\infty}X)},k) \ar[r] & C^*((\Sigma_+^{\infty}X),k).}
\]
On homotopy groups, both maps induce isomorphisms of the corresponding cohomology groups, so they must be equivalences.
\end{proof}
\begin{rem}
  Examples of spaces which are $p$-good are pointed connected space such that $\pi_1(X)$ is finite \cite[Ch.~VII.5]{bousfield_homotopy_1972} and pointed connected nilpotent spaces \cite[Ch.~VI.5.3]{bousfield_homotopy_1972}. In particular, if $G$ is a compact Lie group or a $p$-compact group, then $BG$ is a $p$-good space.
\end{rem}

Suppose first that $G$ is a $p$-compact group. Then, one could reasonably define a unitary embedding to be a monomorphism $BG \to \pc{BU(n)}$ of $p$-compact groups. The following more general result implies that, in this case, the induced morphism $C^*(BU(n),\F_p) \to C^*(BG,\F_p)$ is finite (note that \Cref{lem:cochaincompletion} implies that we do not need to $p$-complete $BU(n)$ after taking cochains).

\begin{lem}\label{lem:monofinite}
  Suppose $H \leq_f G$ is a subgroup of a $p$-compact group $G$, then $C^*(BG,\F_p)$ is a finite $C^*(BH,\F_p)$-module via the induced morphism $(Bf)^*$.  
\end{lem}
\begin{proof}
  By definition of a subgroup of a $p$-compact group, the homotopy fiber of $(Bf)^*$ is $\F_p$-finite. The result then follows from \cite[Lem.~3.4]{shamir_pcochains}. 
\end{proof}
In general, if $G$ is a finite group, then the constant map $* \to \pc{BG}$ need not have $\F_p$-finite fiber. Based on this observation, \cite{cancas_finloop} suggested the notion of homotopy monomorphism, which only requires the homotopy fiber to be $B\Z/p$-null i.e., the space of based maps $\map_*(B\Z/p,F) \simeq \ast$ for all base-points in $F$. Consequently, we have the following notion of unitary embedding. 
\begin{defn}
A $p$-complete and path-connected space $X$ is said to have a unitary embedding if there exists a continuous map $f\colon X \to \pc{BU(n)}$ for some $n$ with $B\Z/p$-null homotopy fiber $F$.
\end{defn}
The next proposition shows that the finiteness result of \Cref{lem:monofinite} still holds with this definition, and in fact characterizes unitary embeddings. Note that the proof also implies that unitary embeddings do in fact have $\F_p$-finite homotopy fibers. 
\begin{prop}\label{prop:unitaryembbedingfinite}
Let $X$ be a connected $p$-complete  space with $\pi_1(X)$ a finite $p$-group. Assume $H^*(X,\F_p)$ is Noetherian. Then $f\colon X \to \pc{BU(n)}$ is a unitary embedding if and only if the morphism $f^*\colon C^*(BU(n),\F_p) \to C^*({X},\F_p)$ of commutative ring spectra is finite.
\end{prop}
\begin{proof}
We argue as in the proof of \cite[Prop.~2.2]{cancas_finloop}. Consider the fiber sequences
\[
\xymatrix{\Omega{X} \ar[r] & \pc{U(n)} \ar[r] & F \ar[r] & {X} \ar[r] & \pc{BU(n)}.}
\]

Since $\pc{BU(n)}$ is $1$-connected, $F$ is connected as well and also $p$-complete \cite[II.5.1]{bousfield_homotopy_1972}. Moreover, using that $\pc{U(n)}$ is $\F_p$-finite, the Serre spectral sequence shows that $H^*(F,\F_p)$ is a finitely generated $H^*(X,\F_p)$-module. Since $H^*(X,\F_p)$ is Noetherian, $H^*(F,\F_p)$ is Noetherian and thus finitely generated over $\F_p$.  

By assumption, $F$ is $B\Z/p$-null and also $p$-complete by construction, so $\map(B\Z/p,F) \simeq F$ and $p$-complete as well. Since $H^*(F,\F_p)$ is a finitely generated $\F_p$-algebra,  \cite[Cor.~3.4.3]{La} applies to give
\[ T_{\Z/p}(H^*(F,\F_p)) \cong H^*(\map(B\Z/p,F),\F_p) \cong H^*(F,\F_p), \]
where $T_{\Z/p}$ is Lannes's $T$-functor introduced in \cite{La}. By \cite[Thm.~6.2.1]{Schwartz_book}, it follows that $H^*(F,\F_p)$ is locally finite and, in particular, all 
of its elements are nilpotent. Because $H^*(F,\F_p)$ is finitely
generated as an $\F_p$-algebra, it must be finite. Therefore, \cite[Lem.~3.4]{shamir_pcochains} implies that $f^*\colon C^*(BU(n),\F_p) \to C^*(X,\F_p)$ is finite. 

Conversely, let $f\colon X\rightarrow \pc{BU(n)}$. If $f^*\colon C^*(BU(n),\F_p)\rightarrow C^*(X,\F_p)$ is a finite morphism, then $H^*(X,\F_p)$ is a finitely generated $H^*(BU(n),\F_p)$-module (see Remark \ref{rem:finite_fg}), so \cite[Lem.~3.5]{shamir_pcochains} applies to conclude that the fiber of $f$ is $\F_p$-finite. Therefore, by \cite[Thm 3.3.4.1]{La}, the pointed mapping space $\map_*(B\mathbb Z/p,F)$ is contractible and then $f$ is a unitary embedding of $X$.
\end{proof}
Note that we will not need the full strength of this theorem in our paper, since we only require it for $p$-compact groups, however we include it as it may be of independent interest.

\subsection{Chouinard's theorem for $p$-local compact groups}

The classical formulation of Choui-nard's theorem~\cite{chouinard} states that, for a finite group $G$ and any field $k$ of characteristic $p$, a $kG$-module $M$ is projective if and only if its restrictions $\Res_E^G(M) \in \Mod_{kE}$ to all elementary abelian subgroups of $G$ are projective. Translated to our setting, this says that an object $M \in \Mod_{C^*(BG,k)}$ is equivalent to $0$ if and only if $C^*(BE,k) \otimes_{C^*(BG,k)} M \simeq 0$ for all elementary subgroups $E \subseteq G$. Our goal in this subsection is to generalize this result to homotopical groups. 

\begin{defn}
A $p$-local compact group $\cG = (S,\cF)$ is said to satisfy Chouinard's theorem if induction and coinduction along the morphism induced by restriction
\[
\xymatrix{C^*(B\cG,k) \ar[r] & \prod_{\cE(\cG)}C^*(BE,k)}
\]
are conservative, where $\cE(\cG)$ denotes a set of representatives of $\cF$-conjugacy classes of elementary abelian subgroups of $S$.
\end{defn}

As the base case, we recall that Benson and Greenlees \cite[Thm.~3.1]{bg_stratifyingcompactlie} have proven Chouinard's theorem for all compact Lie groups. 

\begin{thm}[Benson--Greenlees]\label{thm:chouinard_compactlie}
If $G$ is a compact Lie group, then $G$ satisfies Chouinard's theorem.
\end{thm}

We will use this result together with a stable transfer to leverage Chouinard's theorem to a large class of homotopical, see \Cref{prop:plocalcompactchouniard}. To explain the idea of our argument, let us first consider the case of connected $p$-compact groups. We will use the terminology and notation introduced in \Cref{ssec:recollectionshomgps}.

\begin{lem}\label{lem:reduction}
Let $\cG$ be a $p$-compact group, and let $G'$ be a compact Lie group of the same rank. If there is a monomorphism $G' \hookrightarrow \cG$ such that 
$[W_{\cG}:W_{G'}]$ is prime to $p$, then $\cal{G}$ satisfies Chouinard's theorem.
\end{lem}
\begin{proof}
Since the index $[W_{\cG}:W_{G'}]$ is prime to $p$, we see that the $p$-completion of the maximal $p$-toral subgroup of $G$, $\widehat{S}=N_p(T)^\wedge_p$,  is a maximal $p$-compact toral subgroup of $\cG$.  The proof of Theorem 2.3 of \cite{dwyerwilkerson_finloop} shows that the Euler characteristic $\chi(\cG/\widehat{S})$ is prime to $p$, where $\cG/\widehat{S}$ is the homotopy fiber of $\Theta\colon B\widehat{S}\rightarrow B\cG$. Then, by a transfer argument, as in \cite{dwyerwilkerson_transfer}, the map on cochains $C^*(B\cG,\F_p) \to C^*(B\widehat{S},\F_p)$ has a retract. By \Cref{lem:retract}, induction and coinduction along $\Theta^*\colon C^*(B\cG,\F_p)\rightarrow C^*(B\widehat{S},\F_p)$ are conservative.

Now $\widehat{S}$ satisfies Chouinard's theorem because $N_p(T)$ is a compact Lie group and $C^*(B\widehat{S},\F_p)= C^*(BN_p(T)^\wedge_p,\F_p)\simeq C^*(BN_p(T),\F_p)$. Since  any elementary abelian $p$-subgroup of $\cG$ is subconjugated to $\widehat{S}$, the factorization of  $C^*(B\cG,\F_p)\rightarrow \prod_{\cE(\cG)} C^*(BE,\F_p)$ through
\[
\xymatrix{C^*(B\cG,\F_p) \ar[r] & C^*(B\widehat{S},\F_p) \ar[r] & \prod_{\cE(\widehat{S})} C^*(BE,\F_p)}
\]
shows that $\cG$ satisfies Chouinard theorem.
\end{proof}

This result allows us to deduce Chouinard's theorem for connected $p$-compact groups from the classification theorem of Andersen, Grodal, M{\o}ller, and Viruel. 

\begin{thm}\label{thm:Chouinard_pcompact}
	Let $\cG$ be a connected $p$-compact group, then $\cal{G}$ satisfies Chouinard's theorem. 
\end{thm}
\begin{proof}
	By the classification of connected $p$-compact groups \cite{AGMV_classification1,AG_classification2}, $\cG$ can be written as a product $\cal{G}_1 \times \cal{G}_2$ where $\cal{G}_1$ is the $\F_p$-completion of a compact connected Lie group, and $\cal{G}_2$ is a product of exotic $p$-compact groups. It thus suffices to prove the proposition for each exotic $p$-compact group, and we will use  \cref{lem:reduction} in every case.

	If $p=2$, then the only exotic $2$-compact group is the Dwyer--Wilkerson group $DI(4)$ \cite{dw_di4}. Let $\Spin(7)$ be the double cover of $SO(7)$, the construction of $DI(4)$ shows that $\Spin(7) \subset DI(4)$, both of rank $3$. The Weyl group index $[W_{DI(4)} : W_{\Spin(7)}]$ is odd (in fact, it is equal to 7 \cite[p.~168]{notbohm_di(4)}), and so, $\Spin(7)$ and $DI(4)$ have the same $2$-Sylow subgroup $S$, hence the result holds in this case. 

	When $p$ is odd, we split the result into two further cases: the non-modular $p$-compact groups (the Clark--Ewing $p$-compact groups), and the modular $p$-compact groups (Aguad\'e and Zabrodsky). Here non-modular means that $p$ does not divide the order of the Weyl group. In particular, the Sylow $p$-compact toral subgroup $N \cong (\Z_p)^r$, so that in particular $C^*(BS,\F_p) \simeq C^*(B(S^1)^r,\F_p)$, and we can apply \cref{lem:reduction} with $G'=(S^1)^r$.

	The final case is then that of the modular exotic connected $p$-compact groups for $p$ odd. These are precisely the $p$-compact groups constructed by Aguad\'e \cite{aguade} and Zabrodsky (named $G_{12},G_{29},G_{31}$, and $G_{34}$) and the generalized Grassmannians $G(m,n,s)$ of Notbohm--Oliver \cite{notbohm_grassmanian}. 
	
	For the examples constructed by Aguad\'e and Zabrodsky (see page 37 of \cite{aguade}) we see that, at the corresponding prime, there is a subgroup $SU(p)\subset \cG$ of the same rank such that $[W_{\cG}: \Sigma_p]$ is prime to $p$ (Table 1 in  \cite{aguade}).
	
	For the family of generalized Grassmannians, Notbohm in \cite{notbohm_grassmanian} gives a unified construction of the classifying space as a colimit of classifying spaces of unitary groups with a subgroup $U(n)\subset \cG$ of the same rank $n$ 
such that $[W_{\cG}: \Sigma_n]=q^{n-1}r$ where $r|q|p-1$.	
\end{proof}

It is worth pointing out that in all the cases considered in the proof of \cref{thm:Chouinard_pcompact}, we describe the $\F_p$-completion of $S$ as the $\F_p$-completion of a compact $p$-toral group which in particular is a compact Lie group. Therefore, using the classification, for any connected $p$-compact group $\cG$ with discrete $p$-toral Sylow $S$, there exists a $p$-toral group $N$ such that $BN^\wedge_p\simeq BS^\wedge_p$. This observation motivates the following definition and the subsequent proposition.

\begin{defn}\label{defn:geometric}
A discrete $p$-toral group $S$ is called geometric if there exists a $p$-toral group $S'$ together with an equivalence $BS^{\wedge}_p \simeq (BS')^{\wedge}_p$ of $p$-completed classifying spaces. That is, if $S$ is the discrete approximation of a $p$-toral group $S'$.
\end{defn}

\begin{prop}\label{prop:toralchouinard}
Suppose $S$ is a geometric discrete $p$-toral group, then $S$ satisfies Chouinard's theorem.
\end{prop}
\begin{proof}
By assumption, there exists a $p$-toral group $S'$ and a mod $p$ equivalence $BS\rightarrow BS^{\wedge}_p \simeq (BS')^{\wedge}_p$, which in particular induces an equivalence $f^*\colon C^*(BS,\F_p) \simeq C^*(BS',\F_p)$ by passing to cochains, as well as an isomorphism $H^*(BS,\F_p) \cong H^*(BS',\F_p)$ of unstable algebras over the mod $p$ Steenrod algebra. By \cite[Thm.~3.1.1]{La}, elementary abelian $p$-subgroups of $p$-toral groups or $p$-discrete groups $P$ are detected by morphisms of algebras over the Steenrod algebra in mod $p$ cohomology,
\begin{equation}
\Rep(E,P) \cong [BE,BP^\wedge_p]\cong \Hom_{\mathcal K}(H^*(BP,\F_p), H^*(BE,\F_p)),
\end{equation}
then  $f$ provides a bijection between elementary abelian subgroups of $S$ and $S'$ up to conjugacy, $\cE(S) \cong \cE(S')$. It consequently suffices to verify that $S'$ satisfies Chouinard's theorem, which has been established in \Cref{thm:chouinard_compactlie}.
\end{proof}

In \Cref{prop:plocalcompactchouniard} below, we will use this result to show that $\cG = (S,\cF)$ satisfies Chouinard's theorem whenever $S$ is geometric. In particular, we have:

\begin{cor}\label{cor:geometricexamples}
Let $\cG = (S,\cF)$ be a $p$-local compact group. The discrete $p$-toral subgroup $S$ is geometric if $\cG$ belongs to any of the following classes of homotopical groups:
	\begin{enumerate}
		\item compact Lie groups, or
		\item connected $p$-compact groups, or
		\item $p$-local finite groups.
	\end{enumerate}
\end{cor}
\begin{proof}
The first two examples have been covered above. The claim for finite $p$-local groups follows from the classical version of Chouinard's theorem, because in this case $S$ is a finite $p$-group.
\end{proof}

\subsection{Finiteness and a stable transfer for $C^*(B\cG,\F_p)$.}\label{ssec:splitting}
In order to produce a stable transfer for $p$-local compact groups, we rely on theorems of Gonzalez and Ragnarsson as well as a general result on limits of retracts whose proof is postponed to the end of this section.

In \cite{gonzalez_approxpcompact} Gonzalez proves that any $p$-local compact group $\cG = (S,\cF)$ can be approximated by a sequence of $p$-local finite groups $\cG_i = (S_i,\cF_i)_{i \ge 0}$ in such a way that the homotopy colimit over the associated sequence of classifying spaces is equivalent to the classifying space of $\cG = (S,\cF)$ after $p$-completion. Unpacking his construction, the following is \cite[Thm.~1]{gonzalez_approxpcompact}.

\begin{thm}[Gonzalez]\label{thm:approximation}
If $\cG = (S,\cF)$ is a $p$-local compact group, then there is a commutative diagram
\[
\xymatrix{\ldots \ar[r] & BS_i \ar[r] \ar[d] &  BS_{i+1} \ar[r] \ar[d] & \ldots \\
\ldots \ar[r] & B\cG_i \ar[r] & B\cG_{i+1} \ar[r] & \ldots}
\]
of $p$-local finite groups $\cG_i = (S_i,\cF_i)$ indexed on $i \in \N$ whose colimits recovers the map $BS \to B\cG$ after $p$-completion. 
\end{thm}

The next theorem of Ragnarsson shows that at each level of the induced tower of cochain ring spectra there is a transfer $C^*(BS_i,\F_p) \to C^*(B\cG_i,\F_p)$; however, note we do not know whether these transfer maps can be constructed in a way that is compatible with the structure maps in the corresponding towers. Recall that the suspension spectrum of an unbased space $X$ will be denoted by $X_{+} = \Sigma_{+}^{\infty}X$. The following is a combination of \cite[Thm.~D]{ragnarsson_transfer} and \cite[Cor.~6.7]{rag_transfer_diagram}.

\begin{thm}[Ragnarsson]\label{thm:transfer}
For any $p$-local finite group $\cG = (S,\cF)$, there is a natural transfer $s\colon B\cG_+ \to BS_+$. In particular, the canonical map $f\colon BS_+ \to B\cG_+$ is split, i.e.,  $f \circ s \simeq \id_{B\cG_+}$. Moreover, the following diagram commutes up to homotopy
\[
\xymatrix@C=3.55em{
  \ar[d]_s B\cG_+ \ar[r]^-{\Delta} & B\cG_+ \otimes B\cG_+ \ar[d]^{1 \otimes s} \\
BS_+ \ar[r]_-{(f \otimes 1)\circ \Delta} & B\cG_+ \otimes BS_+.
}
\]
In particular, the map $C^*(B\cG,\F_p) \to C^*(BS,\F_p)$ is split as a map of $C^*(B\cG,\F_p)$-modules.
\end{thm}

\begin{cor}\label{cor:degreewisefinite}
For any $p$-local compact group $\cG = (S,\cF)$, the homotopy groups $\pi_*C^*(B\cG,\F_p)$ are degreewise finite. 
\end{cor}
\begin{proof}
Let $\cG_i = (S_i,\cF_i)_{i\ge 0}$ be an approximation to $\cG$ as in \Cref{thm:approximation}. We follow Gonzalez's proof of \cite[Prop.~4.3 and Thm.~4.4]{gonzalez_approxpcompact}. Since $H^{-*}(BS_i,\F_p) \cong \pi_{*}C^*(BS_i,\F_p)$ is degreewise finite for all $i$, the same is true for $\pi_*C^*(B\cG_i,\F_p)$ by \Cref{thm:transfer}. It follows that both $(\pi_{*}C^*(BS_i,\F_p))_i$ and $(\pi_{*}C^*(B\cG_i,\F_p))_i$ are towers satisfying the Mittag-Leffler condition in each degree, so that the corresponding $\lim^1$-terms vanish. Using the Milnor sequence and \Cref{thm:transfer} again, we see that  
\[
\pi_*C^*(B\cG,\F_p)  \cong \pi_*\lim_iC^*(B\cG_i,\F_p) \cong \lim_i\pi_*C^*(B\cG_i,\F_p) \]
is a subgroup of $\lim_i\pi_*C^*(BS_i,\F_p) \cong \pi_*C^*(BS,\F_p)$, because $\lim$ is left-exact. By \cite[Thm.~12.1]{dwyerwilkerson_finloop}, $\pi_*C^*(BS,\F_p)$ is Noetherian, hence degreewise finite, and the claim follows.
\end{proof}

\begin{prop}\label{thm:cochainsretract}
Let $\cG=(S,\cF)$ be a $p$-local compact group and let $\theta \colon BS \to B\cG$ denote the canonical map. Then $\theta^*\colon C^*(B\cG,\F_p) \to C^*(BS,\F_p)$ is split as a map of $C^*(B\cG,\F_p)$-modules.  In particular, if $M\in \Mod_{C^*(B\cG,\F_p)}$ then $M$ is a retract of $\Res \Ind (M)$ and of $\Res \Coind M$.
\end{prop}
\begin{proof}
  Applying $C^*(-,\F_p)$ to Gonzalez's approximation from \Cref{thm:approximation} yields a map of towers $(g_i)_i\colon (C^*(B\cG_i,\F_p))_i \to (C^*(BS_i,\F_p))_i$ with limit $g\colon C^*(B\cG,\F_p) \to C^*(BS,\F_p)$. Since the structure maps of both towers as well as all $g_i$ are maps of commutative ring spectra, we can work entirely in $\Mod_{C^*(B\cG,\F_p)}$. Let $R = C^*(B\cG,\F_p)$, then \Cref{thm:transfer} and \Cref{cor:degreewisefinite} imply that the conditions of \Cref{prop:splitting} below are satisfied, implying the claim. The second statement follows from applying $- \otimes_{C^*(B\cG,\F_p)} M$ and $\Hom_{C^*(B\cG,\F_p)}(-, M)$ to $\id_{C^*(B\cG,\F_p)}\simeq t\circ \theta^*$, see \Cref{lem:retract}.
\end{proof}
As a consequence, we obtain a version of Chouinard's theorem for $\Mod_{C^*(B\cG,\F_p)}$ for suitable homotopical groups $\cG$, cf.~\Cref{cor:geometricexamples}.  

\begin{thm}\label{prop:plocalcompactchouniard}
Let $\cG=(S,\cF)$ be a $p$-local compact group such that $S$ is geometric and let $\theta \colon BS \to B\cG$ denote the canonical map. Then $\cG$ satisfies Chouinard's theorem, i.e., induction and coinduction along the morphism induced by restriction
\[
\xymatrix{f \colon C^*(B\cG,\F_p) \ar[r] &  \prod \limits_{\cE(S)}C^*(BE,\F_p)}
\]
are conservative, where $\cE(\cS)$ denotes a set of representatives of conjugacy classes of elementary abelian subgroups of $S$. 
\end{thm}

\begin{proof}
The morphism $f$ factors as the composite 
\[
\xymatrix{C^*(B\cG,\F_p)\ar[r]^-{\theta^*} & C^*(BS,\F_p) \ar[r]^-{\zeta} & \prod \limits_{\cE(S)}C^*(BE,\F_p),}
\]
so it suffices to show that induction and coinduction along both $\theta^\ast$ and $\zeta$ are conservative. By \Cref{thm:cochainsretract} and \Cref{lem:retract} this is true for $\theta^*$, while it holds for $\zeta$ by assumption on $S$ and \Cref{prop:toralchouinard}.
\end{proof}

Furthermore, we can deduce from the existence of a retract that the cohomology ring of any $p$-local compact group is Noetherian. 

\begin{cor}\label{cor:noetherian}
For any $p$-local compact group $\cG = (S,\cF)$, the ring spectrum $C^*(B\cG,\F_p)$ is Noetherian.
\end{cor}
\begin{proof}
The ring spectrum $C^*(BS,\F_p)$ is Noetherian by \cite[Thm.~12.1]{dwyerwilkerson_finloop}. Since $C^*(B\cG,\F_p)$ is a module retract of $C^*(BS,\F_p)$ by \Cref{thm:cochainsretract}, we conclude by \Cref{lem:noethretract}.
\end{proof}

\subsection{Phantom maps and limiting retracts}
Let $R$ be a (not necessarily Noetherian) commutative ring spectrum and $\Mod_R$ be the associated symmetric monoidal category of $R$-modules. In this section we always assume that $R$ is $p$-local for some fixed prime $p$. Our goal in this subsection is to prove a general result about when a tower of split $R$-linear maps $(g_i) \colon (M_i) \to (N_i)$ gives a splitting on the induced map $g \colon M \to N$ obtained under limits. 
 
We begin with a mild generalization of ($p$-local) Brown--Comenetz duality. Since $\Q/\Z_{(p)}$ is an injective abelian group, the functor
\[
\xymatrix{\pi_{-*}(-)^{\vee} = \Hom_{\Z_{(p)}}(\pi_{-*}(-),\Q/\Z_{(p)})\colon \Mod_R^{\op} \ar[r] & \Mod_{\Z_{(p)}}^{\mathrm{graded}}}
\]
is cohomological and hence representable by a spectrum $I_R \in \Mod_R$, the ($R$-linear) Brown--Comenetz dualizing module. By construction, it satisfies the universal property
\[
\xymatrix{\pi_*\Hom_R(M,I_R) \ar[r]^-{\simeq} & \Hom_{\Z_{(p)}}(\pi_{-*}(M),\Q/\Z_{(p)})}
\]
for any $M \in \Mod_R$. In particular, this implies that $I_R \simeq \Hom_{S^0}(R,I_{S^0})$, the coinduction of the usual Brown--Comenetz dual of the sphere spectrum. 

\begin{defn}
The ($R$-linear) Brown--Comenetz dual of an $R$-module $M$ is defined as the $R$-module $I_RM = \Hom_R(M,I_R)$.
\end{defn}

By adjunction, the evaluation map 
\[
\xymatrix{M \otimes_R I_RM = M \otimes_R \Hom(M,I_R) \ar[r] & I_R}
\]
corresponds to a canonical and natural $R$-linear duality map $\phi_M\colon M \to I_R^2M$. The next lemma gives a sufficient criterion for when $\phi_M$ is an equivalence. 

\begin{lem}\label{lem:bcdoubledual}
Suppose that $M \in \Mod_R$ has degreewise finite homotopy groups, then $\phi_M\colon M \to I_R^2M$ is an equivalence. 
\end{lem}
\begin{proof}
We first claim that the induced map $\pi_d \phi_M \colon \pi_d M \to \pi_d I_R^2M$ coincides with the map from $\pi_d(M)$ to its double $p$-local Pontryagin dual $\pi_d(M)^{\vee\vee}$. In the case that $R$ is the sphere spectrum, this is shown in \cite[Thm.~2.4]{bc_dual} or also follows from \cite[Prop.~5.1.4]{margolis}, while the general case can be proven similarly.

To finish the proof, it suffices to show that $\pi_*\phi_M$ is an isomorphism in each degree. Fix a degree $d \in \Z$. Because $\pi_dM$ is finite, its double $p$-local Pontryagin dual $\pi_d(M)^{\vee\vee}$ is canonically isomorphic to $\pi_dM$. By the universal property of $I_R$, it follows that $\phi_M$ is an equivalence.
\end{proof}

Recall that an $R$-module map $f\colon M \to N$ is called $R$-phantom if the composite $C \to M \xr{f} N$ is null for any $R$-linear map $C \to M$ with $C\in \Mod_R^{\omega}$. The next lemma is a straightforward generalization of a result due to Margolis~\cite[Prop.~5.1.2]{margolis}.

\begin{lem}[Margolis]
For any $M \in \Mod_R$, there are no non-trivial $R$-phantom maps with target $I_RM$. 
\end{lem}
\begin{proof}
For the convenience of the reader, we reproduce Margolis's argument in the context of $R$-modules. Any $R$-module $Z$ can be written canonically as a filtered colimit over all compact $R$-modules mapping to $Z$; more precisely, $Z \simeq \colim_{U \in \Lambda(Z)}U$, where $\Lambda(Z)$ is a set of representatives for maps $U \to Z$ with $U \in \Mod_R^{\omega}$. By the universal property of $I_R$, we thus obtain a commutative diagram
\[
\xymatrix{\pi_*\Hom_R(Z,I_RM) \ar[r]^-{\simeq} \ar[d]_-{\alpha} & \Hom_{\Z_{(p)}}(\pi_{-*}(Z\otimes_RM),\Q/\Z_{(p)}) \ar[d]^-{\beta} \\
\lim_{U \in \Lambda(Z)}\pi_*\Hom_R(U,I_RM) \ar[r]_-{\simeq} &\lim_{U \in \Lambda(Z)}\Hom_{\Z_{(p)}}(\pi_{-*}(U\otimes_RM),\Q/\Z_{(p)}).}
\]
Since $\otimes_R$ and $\pi_*$ commute with filtered colimits and $\Hom_{\Z_{(p)}}(-,\Q/\Z_{(p)})$ turns filtered colimits into cofiltered limits, the map $\beta$ is an isomorphism. It follows that $\alpha$ is an isomorphism as well, so we are done. 
\end{proof}

Combining the previous two lemmata, we obtain:

\begin{lem}\label{lem:mbcphantom}
If $M \in \Mod_R$ such that $\pi_*M$ is degreewise finite, then there are no non-trivial $R$-phantom maps with target $M$. 
\end{lem}
\begin{proof}
Indeed, Margolis's result shows that there are no non-trivial $R$-phantom maps with target $I_R(I_RM)$, which is equivalent to $M$ by \Cref{lem:bcdoubledual}. 
\end{proof}

For the proof of the next proposition, we also need the following standard lemma.

\begin{lem}\label{lem:finiteness}
If $M \in \Mod_R$ has degreewise finite homotopy groups, then so does $M \otimes_R C$ for any $C \in \Mod_R^{\omega}$.
\end{lem}
\begin{proof}
Any compact $R$-module is finitely built from $R$, so we can argue by induction on the number of $R$-cells of $C$. The claim is true by assumption for $R$ itself, and the induction step follows from the long exact sequence of homotopy groups. 
\end{proof}

In the next result, note that we do not assume that the splittings are compatible with the structure maps. 

\begin{prop}\label{prop:splitting}
Suppose that $(g_i)\colon (M_i) \to (N_i)$ is an $R$-linear map of towers of $R$-modules such that $\pi_*M_i$ and $\lim_i\pi_*M_i$ are degreewise finite for all $i$. If $g_i$ is split as a map of $R$-modules for all $i$, then the induced map $g\colon \lim_iM_i \to \lim_iN_i$ is split in $\Mod_R$ as well. 
\end{prop}
\begin{proof}
We fit the map $(g_i)$ into a fiber sequence of towers of $R$-modules,
\begin{equation}\label{eq:towerseq}
\xymatrix{(F_i)_i \ar[r]^-{(f_i)} & (M_i)_{i} \ar[r]^-{(g_i)} & (N_i)_{i}.}
\end{equation}
Since limits commute with limits, this induces a fiber sequence 
\[
\xymatrix{F \ar[r]^-{f} & M \ar[r]^-g & N}
\]
by passing to limits. It is easy to see that the map $g$ is split in $\Mod_R$ if and only if $f$ is null-homotopic in $\Mod_R$, and so we prove the latter. 

The Milnor sequence (see e.g.,~\cite[Prop.~A.5.15]{orangebook}) shows that $\pi_*M \cong \lim_i\pi_*M_i$, which is degreewise finite by assumption, so $M$ satisfies the hypothesis of \Cref{lem:mbcphantom}, i.e., there can be no non-trivial phantom maps with target $M$. It therefore suffices to show that $f$ is phantom, for then $f=0$.

In order to prove the claim, let $C \to F$ be an $R$-linear map with $C \in \Mod_R^{\omega}$; we need to prove that $C \to F \xr{f} M$ is null. Recall that $R$-linear Spanier--Whitehead duality $D_R=\Hom_R(-,R)$ establishes a natural equivalence of $R$-modules
\[
\Hom_R(C,X) \simeq \Hom_R(R,\Hom_R(C,X)) \simeq \Hom_R(R,X \otimes_RD_RC)
\]
for any compact $R$-module $C$ and any $R$-module $X$, see for example \cite[Ch.~IV, Cor.~1.8]{ekmm}. It follows that the composite $C \to F \xr{f} M$ is adjoint to the $R$-linear map
\[
\xymatrixcolsep{3.5pc}
\xymatrix{R \ar[r] & F\otimes_R D_R(C) \ar[r]^-{f\otimes_R D_R(C)} & M \otimes_R D_R(C),}
\]
so equivalently we have to show that this composite is null. Let us write $D=D_R(C)$ for simplicity. Since an $R$-linear map out of $R$ is just an element in $\pi_*$, it thus suffices to prove that
\[
\xymatrix{{f_*^D = \pi_*(f\otimes_R D)}\colon \pi_*(F \otimes_R D) \ar[r] & \pi_*(M \otimes_R D)}
\]
is zero. 

Consider the tower $(\pi_*(M_i \otimes_R D))_i$. This is a tower of graded abelian groups, which in each degree is a tower of finite abelian groups by \Cref{lem:finiteness}. Consequently, the Mittag-Leffler condition is satisfied, giving
\begin{equation}\label{eq:ml}
\lim_i^1\pi_*(M_i \otimes_R D) = 0. 
\end{equation}
Smashing the fiber sequence of towers \eqref{eq:towerseq} with $D$, the map $(f_i)$ induces a morphism of Milnor sequences associated to the first two towers (omitting the $\lim^1$-terms):
\[
\xymatrix{\pi_*(F\otimes_R D) \cong \pi_*\lim_i(F_i\otimes_R D) \ar[r] \ar[d]_{f_*^D} & \lim_i\pi_*(F_i\otimes_R D) \ar[r] \ar[d] & 0 \\
\pi_*(M \otimes_R D) \cong \pi_*\lim_i(M_i \otimes_R D) \ar[r]^-{\simeq} & \lim_i\pi_*(M_i \otimes_R D) \ar[r] & 0.}
\]
The bottom horizontal map is an isomorphism by \eqref{eq:ml}. This diagram implies $f_*^D = 0$ provided we can show that the right vertical map is zero. To this end, for fixed $i$, consider a segment of the long exact sequence in homotopy:
\begin{equation}\label{eq:finiteseq}
\xymatrix{\pi_*(F_i\otimes_R D) \ar[r] & \pi_*(M_i \otimes_R D) \ar[r] & \pi_*(N_i \otimes_R D).}
\end{equation}
By assumption, $M_i \to N_i$ is split for all $i$, so $M_i \otimes_R D \to N_i \otimes_R D$ is split as well, thus the right map in \eqref{eq:finiteseq} is injective, hence $\pi_*(F_i\otimes_R D) \to \pi_*(M_i \otimes_R D)$ is zero. This finishes the proof. 
\end{proof}

\section{Stratification}\label{sec:stratification}

In this section, we prove our main theorems about $p$-local compact groups.

\subsection{$F$-isomorphism and Quillen stratification}

In order to verify that $p$-local compact groups satisfy Quillen lifting, we will first establish an $F$-isomorphism theorem in this context. For a $p$-local compact group $\cG=(S,\cF)$, we let $\cal{F}^e$ denote the full subcategory of $\cal{F}$ consisting of elementary abelian $p$-groups. 

\begin{thm}\label{thm:fiso}
Let $\cG=(S,\cF)$ be a $p$-local compact group, then restriction to elementary abelian subgroups of $S$ induces an $F$-isomorphism in mod $p$ cohomology
\[
\xymatrix{
\lambda_{\cF}\colon H^*(B\cG,\F_p) \ar[r]^-{\sim} & \varprojlim \limits_{\cF^e} H^*(BE,\F_p).}
\]
\end{thm}
\begin{proof} The proof follows the strategy used for the finite case in \cite[Prop.~5.1]{blo_fusion}.
We first prove that the morphism 
\begin{equation}\label{eq:fisobs}
\xymatrix{R_{\Lambda}\colon H^*(BS,\F_p)\ar[r] & \varprojlim \limits_{{\cF_S(S)^e}} H^*(BE,\F_p)}
\end{equation}
is an $F$-isomorphism for any discrete $p$-toral group $S$, where the morphism is induced by restriction to elementary abelian $p$-subgroups and the inverse limit is taken with respect to subgroups and conjugacy relations. Here, the category ${\cF_S(S)^e}$ is the full subcategory of elementary abelian $p$-groups in the fusion category $\cF_S(S)$ of $S$. 

To prove this we use an abstract form of Quillen stratification due to Rector. Let $\cal{A}_p$ denote the mod $p$ Steenrod algebra. Given a Noetherian unstable algebra $\Lambda$ over $\mathcal{A}_p$, let $\cC(\Lambda)$ denote the category with objects pairs $(E,\Phi)$ where $E$ is a non-trivial elementary abelian $p$-group and $\Phi \colon \Lambda \to H^*(BE,\F_p)$ is a finite $\cal{A}_p$-algebra homomorphism. A morphism $f\colon (E,\Phi) \to (E',\Phi')$ in $\cC(\Lambda)$ is a $\mathcal{A}_p$-algebra homomorphism $H^*(BE,\F_p) \to H^*(BE',\F_p)$ making the obvious triangle commute. In this case, Rector \cite{rector_quillenstrat} for $p=2$, and Broto--Zarati \cite[Thm.~1.3 and Prop.~3.2]{brotozarati_steenrod} for $p$ odd, prove that 
\[
\xymatrix{r_{\Lambda}\colon \Lambda \ar[r] & \varprojlim\limits_{\cC(\Lambda)} H^*(BE,\F_p)}
\] 
is an $F$-isomorphism.

We wish to apply this result with $\Lambda = H^*(BS,\F_p)$. In \cite[Prop.~12.1]{dwyerwilkerson_finloop} Dwyer and Wilkerson showed that $H^*(BS,\F_p)$ is Noetherian unstable algebra over $\cal{A}_p$, and so \eqref{eq:fisobs} is an $F$-isomorphism if we can show that Rector's category $\cC(H^*(BS,\F_p))$ and $\cF_S(S)^e$ are equivalent. More precisely, we claim that $H^*(B-,\F_p)$ induces an equivalence of categories between $(\cF_S(S)^e)^{\op}$ and $\cC(H^*(BS,\F_p))$.

To this end, let $\Rep(E,S)$ denote the set of homomorphisms from $E$ to $S$, modulo conjugacy in $S$. Then, there are equivalences $\Rep(E,S)\cong [BE,BS]\cong[BE,BS^\wedge_p]$ where the first equivalence is just because $S$ and $E$ are discrete groups, and the second follows from the proof of \cite[Prop.~3.1]{dwyerzabrodsky_classifyingspaces}.  Lannes's theory applies to show that 
\begin{equation}\label{eq:morphisms}
\Rep(E,S) \cong [BE,BS^\wedge_p]\cong \Hom_{\mathcal K}(H^*(BS,\F_p), H^*(BE,\F_p)),
\end{equation}
where $\cal K$ is the category of unstable algebras over $\cal{A}_p$, see \cite[Thm.~3.1.1]{La}.

In order to see that $H^*(B-,\F_p)$ induces a functor as claimed, it remains to show that an algebra morphism $Bf^*\colon H^*(BS,\F_p)\rightarrow H^*(BE,\F_p)$ is finite if and only if $f \colon E \to S$ is injective. Both conditions are equivalent to the statement that $Bf^\wedge_p$ is a monomorphism, this is, the homotopy fiber of $Bf^\wedge_p$ is $\mathbb F_p$-finite. This together with \eqref{eq:morphisms} shows that $H^*(B-,\F_p)\colon (\cF_S(S)^e)^{\op} \to \cC(H^*(BS,\F_p))$ is well-defined and essentially surjective. If $E,E' \subseteq S$ are elementary abelian, then again Lannes's theory provides  isomorphisms
\[
\Hom(E,E') \cong [BE,BE'] \cong \Hom_{\cal K}(H^*(BE',\F_p),H^*(BE,\F_p)),
\]
so $H^*(B-,\F_p)$ is faithful. This also shows that any map $\phi\colon H^*(BE',\F_p) \to H^*(BE,\F_p)$ in $\cC(H^*(BS,\F_p))$ lifts to a unique map $f\colon E \to E'$ with $Bf^* = \phi$, and it follows easily from \eqref{eq:morphisms} that $f$ is given by subconjugation in $S$. Therefore Rector's category is equivalent to $\cal{F}_s(S)^e$, so that Quillen's $F$-isomorphism theorem holds for discrete $p$-toral groups, i.e., that \eqref{eq:fisobs} is an $F$-isomorphism. 

We now show that the $F$-isomorphism theorem holds for $B\cG$ itself.  The stable elements theorem holds for $p$-local compact groups by \cite[Thm.~2]{gonzalez_approxpcompact}, therefore we have 
\[
H^*(B\cG,\F_p)\cong \varprojlim_{P \in \cF} H^*(BP,\F_p).
\]

In \cite[Def.~3.1 and Lem.~3.2(a)]{blo_pcompact} Broto--Levi--Oliver  describe a full subcategory $\cF^\bullet$ of $\cF$ with finitely many isomorphism classes of objects and isomorphic limits over the corresponding orbit categories, see \cite[Prop.~5.2]{blo_pcompact}. By \cite[Rem.~4.5]{gonzalez_approxpcompact} we have that
 \[
H^*(B\cG,\F_p)\cong \varprojlim_{P \in \cF^\bullet} H^*(BP,\F_p).\] At this point, after replacing $\cF$ by $\cF^\bullet$, the proof of \cite[Prop.~5.1]{blo_fusion} applies formally. 
\end{proof}

We will now deduce the Quillen stratification of the cohomology ring $H^*(B\cG,\F_p)$ for any $p$-local compact group $\cG$, again utilizing the abstract approach developed by Rector~\cite{rector_quillenstrat}. In fact, we will deduce a strong version of stratification analogous to Quillen's decomposition of the prime ideal spectrum of the variety associated to $H^*(BG,\F_p)$ for $G$ a compact Lie group. 

Let $\cG = (\cF,S)$ be a $p$-local compact group, and denote by $\cV_{\cG}$ the homogeneous prime ideal spectrum of $H^*(B\cG,\F_p)$. For an elementary abelian subgroup $E \leqslant S$, let $\cV_{E}$ denote the homogeneous prime ideal spectrum of $H^*(BE,\F_p)$, and let $\cV_{E}^+ = \cV_{E} \setminus \bigcup_{E' < E} \res_{E}^{E'}\cV_{E'}$. Finally, we write $\cV_{\cG,E}^{+}= \res_{\cG}^{E} \cV_{E}^+$, where restriction is taken along the map induced by the composite $H^*(B\cG,\F_p) \hookrightarrow H^*(BS,\F_p) \to H^*(BE,\F_p)$. 
\begin{defn}\label{defn:varieties}
 We say that $\cG$ satisfies strong Quillen stratification if the variety $\cV_{\cG}$ is the disjoint union of the locally closed subvarieties $\cal{V}_{\cG,E}^{+}$, where $E$ runs over a set of representatives of $\cF$-conjugacy classes of elementary abelian subgroups of $S$.
\end{defn}

The $F$-isomorphism theorem is the key tool in \cite[Prop.~5.2]{blo_fusion} used to prove that the inclusion of the Sylow subgroup $S$ into a $p$-local finite group $(S,\cG)$ induces a finite algebra morphism, which implies that the induced map $H^*(B\cG,\F_p)\rightarrow H^*(BP,\F_p)$ is a finite morphism for any $P\leq S$. The next result generalizes this to $p$-local compact groups.

\begin{prop}\label{prop:elabcohomfingen}
Let $\cG = (\cF,S)$ be a $p$-local compact group, then the map $C^*(B\cG,\F_p) \to C^*(BS,\F_p)$ exhibits $H^*(BS,\F_p)$ as a finitely generated module over $H^*(B\cG,\F_p)$. Moreover, if $E$ is an elementary abelian subgroup of $S$, then $H^*(BE,\F_p)$ is a finitely generated module over $H^*(B\cG,\F_p)$ via the composite $C^*(B\cG,\F_p) \to C^*(BS,\F_p) \to C^*(BE,\F_p)$.
\end{prop}
\begin{proof}
We have seen in \Cref{thm:fiso} that the $F$-isomorphism theorem holds for $p$-local compact groups. Furthermore, \cite[Thm.~12.1]{dwyerwilkerson_finloop}  and the remark after Theorem 2.3 in the same paper show that $H^*(BS,\F_p)$ is a finitely generated algebra. The argument then proceeds as in the proof of \cite[Prop.~5.2]{blo_fusion}. 

We sketch the main ideas in their proof for completeness. Since $H^*(BS,\F_p)$ is a finitely generated $H^*(B\cG,\F_p)$-algebra, it is enough to show that $\theta^*\colon H^*(B\cG,\F_p)\rightarrow H^*(BS,\F_p)$ is integral.  Given $(S,\cF)$, let $\Lambda(\cF)$ be $\varprojlim_{E \in \cF^e} H^*(BE,\F_p)$. By  \cite[Prop.~2.5 and Thm.~2.6]{rector_quillenstrat}, $\Lambda(\cF)$ is a reduced Noetherian unstable $\cA_p$-algebra. Because of the $F$-isomorphism \Cref{thm:fiso}, it is enough to show that the inclusion $\Lambda(\cF_S(S))\rightarrow \Lambda(\cF)$ is integral. Let $V$ be a maximal rank elementary abelian $p$-subgroup
in $S$, then the algebra of invariants $H^*(BV,\F_p)^{GL(V)}$ is a finitely generated algebra and there is a sequence of inclusions $H^*(BV,\F_p)^{GL(V)}\rightarrow \Lambda(\cF_S)\rightarrow \Lambda(\cF)$ which exhibits both $\Lambda(\cF_S)$ and  $\Lambda(\cF)$ as finite $H^*(BV,\F_p)^{GL(V)}$-modules. Therefore, $\Lambda(\cF_S)\rightarrow \Lambda(\cF)$ is integral.

The final claim then follows from this together with \Cref{lem:monofinite}.
\end{proof}

The next theorem generalizes the strong Quillen stratification for $p$-local finite groups due to Linckelmann~\cite{linckelman_stratification}, and also provides an alternative argument for his result.

\begin{thm}\label{thm:qstratification}
A $p$-local compact group $\cG=(S,\cF)$ satisfies strong Quillen stratification, that is, the variety $\cal{V}_{\cG}$ admits a decomposition
\[
\cal{V}_{\cG} \cong \coprod_{E \in \cal{E}{(\cG)}}\cal{V}^+_{\cG,E},
\]
where $\cE(\cG)$ is a set of representatives of $\cF$-conjugacy classes of elementary abelian subgroups of $S$.  
\end{thm}
\begin{proof}
Throughout this proof, we will write $\cC(\Lambda)$ for Rector's category of a Noetherian unstable $\cA_p$-algebra $\Lambda$ and $\cE(\Lambda)$ for a set of representatives of isomorphism classes of objects in $\cC(\Lambda)$. Some of his results require minor modification for $p>2$, which are proven in \cite{brotozarati_steenrod}. Recall that $\cF^e$ is the full subcategory of $\cF$ on the elementary abelian subgroups of $S$. We start by checking that the pair $((\cF^{e})^{\op},H^*(B-,\F_p))$ satisfies the conditions (1)--(5) listed in \cite[Prop.~2.3]{rector_quillenstrat}. 

Indeed, it is clear that $\cF^e$ has finite skeleton, so Property (1) holds. The second property is satisfied by construction and the fact that all morphisms in $\cF^e$ are monomorphisms of elementary abelian groups.  The fourth property holds for $\cF_S(S)^e$ and hence also for $\cF^e$. The fifth property uses that every morphism in $\cF^e$ can be factored into an isomorphism followed by an inclusion, with the trivial group giving the final object. It remains to show Property (3); by the remark at the end of \cite[Sec.~2]{rector_quillenstrat}, it suffices to verify the factorization claim. To this end, let $E' \to E \gets E''$ be a span in $\cF^e$ such that $\ker(H^*(BE,\F_p) \to H^*(BE',\F_p)) \subseteq \ker(H^*(BE,\F_p) \to H^*(BE'',\F_p))$. Without loss of generality, we may assume that the maps $E' \to E$ and $E'' \to E$ are inclusions, again by the factorization of morphisms in fusion systems. The assumption on the cohomologies then forces $E'' \subseteq E'$, so the claim follows.

It thus follows from \cite[Prop.~2.5 and Thm.~2.6]{rector_quillenstrat} that $\Lambda = \varprojlim_{\cF^e}H^*(BE,\F_p)$ is a reduced Noetherian unstable $\cA_p$-algebra and that there is a natural equivalence $\cF^e \simeq \cC(\Lambda)$. As shown in \cite[Sec.~2]{rector_quillenstrat}, these properties suffice to establish a Quillen stratification for $\Lambda = \varprojlim_{\cF^e}H^*(BE,\F_p)$. Furthermore, the $F$-isomorphism $\Lambda \sim H^*(B\cG,\F_p)$ of \Cref{thm:fiso} induces an isomorphism of varieties. We thus obtain strong Quillen stratification
\[
\cV_{\cG} \cong \cV_{\Lambda} \cong \coprod_{E \in \cE{(\Lambda)}}\cal{V}^+_{\Lambda,E} \cong \coprod_{E \in \cE(\cG)}\cal{V}^+_{\cG,E},
\]
where the last isomorphism uses \Cref{prop:elabcohomfingen} to identify Rector's varieties $\cal{V}^+_{E,\Lambda}$ with the varieties $\cal{V}^+_{\cG,E}$ as constructed above \Cref{defn:varieties}.
\end{proof}

\begin{rem}
Since $H^*(B\cG,\F_p)$ is Noetherian by \Cref{cor:noetherian} and the morphisms $H^*(B\cG,\F_p)\rightarrow H^*(BE,\F_p)$ are finite for any $E\leq S$ by \Cref{prop:elabcohomfingen}, Rector's theorems \cite[Thm.~1.4, Thm.~1.7]{rector_quillenstrat} provide an alternative proof of \Cref{thm:qstratification} under the assumption that Rector's category $\cE{(H^*(B\cG,\F_p))}$ is equivalent to the full subcategory of elementary abelian $p$-subgroups $\cF^e$. This equivalence can be obtained by using Lannes's $T$-functor theory \cite[Thm.~3.1.1]{La} (following the proof of \Cref{thm:fiso}) and the identification of homotopy classes of maps $[BE,B\cG]$ with $\cF$-conjugacy classes of elementary abelian $p$-subgroups \cite[Thm.~6.3]{blo_pcompact}.
\end{rem}

\Cref{prop:elabcohomfingen} should be contrasted with the next proposition, which characterizes $p$-compact groups as those $p$-local compact groups for which this finiteness condition can be lifted to a homotopical statement. We first need an auxiliary result. 

\begin{lem}\label{lem:finitepartial2outof3}
Suppose $R \xrightarrow{f} S \xrightarrow{g} T$ are maps of commutative ring spectra such that:
  \begin{enumerate}
    \item $g$ is finite, and
    \item $\Coind_S^T\colon \Mod_S \to \Mod_T$ is conservative,
  \end{enumerate}
then $f$ is finite if and only if $gf$ is finite.
\end{lem}
\begin{proof}
The first implication holds because finite morphisms are closed under composition. Indeed, assume that $f$ is finite, then $S \in \Thick(R)$. Since $g$ is finite, this implies that $T \in \Thick(S) \subseteq \Thick(R)$ as claimed. 

Conversely, let $(M_i)_{i\in I}$ be a filtered system of $R$-modules, then we have to prove that the canonical map
\[
\xymatrix{\phi\colon \colim_{i \in I}\Hom_R(S,M_i) \ar[r] & \Hom_R(S,\colim_{i \in I}M_i)}
\]
is an equivalence. By Condition (2), it suffices to show $\Coind_S^T(\phi)$ is an equivalence; this latter map fits into a commutative diagram:
\[
\xymatrixcolsep{3pc}
\xymatrix{\Hom_S(T,\colim_{i \in I}\Hom_R(S,M_i)) \ar[r]^-{\Coind_S^T(\phi)} &\Hom_S(T,\Hom_R(S,\colim_{i \in I}M_i)) \ar[dd]^-{\sim} \\
\colim_{i \in I}\Hom_S(T,\Hom_R(S,M_i)) \ar[u]^-{\sim} \ar[d]_-{\sim} \\
\colim_{i \in I}\Hom_R(T,M_i) \ar[r] & \Hom_R(T,\colim_{i \in I}M_i).} 
\]
The top left vertical map is an equivalence by Condition (1), while the other two vertical maps are equivalences by the transitivity of coinduction. If $gf$ is finite, then the bottom horizontal map is an equivalence, which shows that $\Coind_S^T(\phi)$ is an equivalence as well. 
\end{proof}

\begin{prop}\label{prop:pcompactgroupchar}
The following statements are equivalent for a $p$-local compact group $\cG=(S,\cF)$:
  \begin{enumerate}
    \item For every elementary abelian $p$-subgroup $E$ of $S$, the canonical restriction map $C^*(B\cG,\F_p) \to C^*(BE,\F_p)$ is finite. 
    \item The canonical map $\theta^* \colon C^*(B\cG,\F_p) \to C^*(BS,\F_p)$ is finite.
    \item The space $B\cG$ is the classifying space of a $p$-compact group.
  \end{enumerate}
\end{prop}
\begin{proof}
The canonical restriction map in (1) factors through $\theta^*$ so that we have the following morphisms of commutative ring spectra:
\[
\xymatrix{C^*(B\cG,\F_p) \ar[r]^-{\theta^*} & C^*(BS,\F_p) \ar[r]^-g & \prod_{E \in \cF^e}C^*(BE,\F_p),}
\]
where $E$ runs through the finite set of elementary abelian subgroups of $S$. The morphism $g$ is finite and induction along $g$ is conservative by Chouinard's theorem for discrete $p$-toral groups (see \Cref{prop:plocalcompactchouniard}), so \Cref{lem:finitepartial2outof3} applies. This shows the equivalence of (1) and (2); namely, $g\theta^*$ is finite if and only if the composite $C^*(B\cG,\F_p) \to C^*(BS,\F_p) \to C^*(BE,\F_p)$ is finite for all $E \in \cF^e$.

To prove that (2) is equivalent to (3), we shall use the criterion in \cite[Lem.~3.4]{shamir_pcochains}: Let $f\colon X\rightarrow Y$ be a map between $p$-complete spaces whose fundamental groups are finite $p$-groups, then $f^*\colon C^*(Y,\F_p)\rightarrow C^*(X,\F_p)$ is finite if and only if the homotopy fiber of $f$ is $\mathbb F_p$-finite.  

In order to apply this statement, note that by \cite[Prop.~4.4]{blo_pcompact}, the classifying space $B\cG$ is $p$-complete and $\pi_1(B\cG)$ is a finite $p$-group. For any discrete $p$-toral group, the canonical augmentation $C^*(BS,\F_p) \to \F_p$ is induced by $*\rightarrow (BS)^\wedge_p$ since $C^*(BS,\F_p) \simeq C^*(BS^\wedge_p,\F_p)$. It is then finite, because $\Omega (BS)^\wedge_p\simeq S^\wedge_p$ is $\F_p$-finite. 

If $C^*(B\cG,\F_p) \to C^*(BS,\F_p)$ is finite, then by \Cref{lem:finitepartial2outof3} so is the composite
\[
\xymatrix{C^*(B\cG,\F_p) \ar[r] & C^*(BS,\F_p) \ar[r] & \F_p,}
\]
hence $\Omega (B\cG)$ is $\mathbb F_p$-finite and therefore $B\cG$ is equivalent to the classifying space of a $p$-compact group, by definition. Conversely, suppose $B\cG$ is the classifying space of a $p$-compact group. By \cite[Prop. 9.9]{dwyerwilkerson_finloop} $\theta^\wedge_p\colon (BS)^\wedge_p \rightarrow B\cG$ is a monomorphism of $p$-compact groups, that is, the homotopy fiber of $\theta^\wedge_p$ is $\F_p$-finite and both spaces are $p$-complete spaces whose fundamental groups are finite $p$-groups. Therefore $C^*(B\cG,\F_p) \to C^*(BS^\wedge_p,\F_p)\simeq C^*(BS,\F_p)$ is finite.
\end{proof}

\subsection{Quillen lifting}
In \Cref{defn:qlifting} we introduced the notion of Quillen lifting for a morphism $f \colon R \to S$. In this section we show that for a saturated fusion system over a geometric discrete $p$-toral group $S$ in the sense  of \Cref{defn:geometric}, the morphism $C^*(B\cG,\F_p) \to \prod_{E\in\cE(\cG)} C^*(BE,\F_p)$ satisfies Quillen lifting. This uses the strong form of Quillen stratification proven in the previous subsection. It also relies on the proof of stratification and costratification for elementary abelian subgroups, given in \Cref{prop:elementary}. 

In the following we will deal with products of ring spectra. We now briefly explain how the support and cosupport of modules over such product rings decompose. To that end, let $S$ be a product of ring spectra $S\simeq \prod_{i\in I} S_i$ where $I$ is finite. Then, the category of modules are related by a canonical equivalence $\Mod_S\simeq \prod_i \Mod_{S_i}$ which we use implicitly in the following. For any $M\in \Mod_R$ there are equivalences of $S$-modules $\Ind_R^SM \simeq (\prod_i S_i) \otimes_R M \simeq \prod_i (S_i \otimes_R M)$. Similarly, there are equivalences of $S$-modules $\Coind_R^SM \simeq \prod_i\Coind_R^{S_i}M \simeq \prod_i \Hom_R(S_i,M)$.  By the definition of support and cosupport, this gives rise to decompositions $\supp_S(\Ind_{R}^{S}M)=\bigcup_i \supp_{S_i}( \Ind_{R}^{S_i}M)$ and $\supp_S( \Coind_{R}^{S}M)=\bigcup_i \supp_{S_i} (\Coind_{R}^{S_i}M)$, where the unions are taken in $\Spec^h(\pi_*S)$. 

  Given a prime ideal $\fp \in \cV_{\cG}$, we say that $\fp$ originates in $E \leqslant S$ if $\fp$ is in the image of $\res_{\cG}^{E}$ but not of $\res_{\cG}^{E'}$ for $E'$ any proper subgroup of $E$. This terminology was introduced in the context of finite groups in \cite[Sec.~9]{bik_finitegroups}. As in \cite[Thm.~9.1]{bik_finitegroups} the strong form of Quillen stratification proven in \Cref{thm:qstratification} implies that for each $\fp \in \cal{V}_{\cG}$ the pairs $(E,\fq)$ where $\fp = \res(\fq)$ are all $\cF$-conjugate, and there is a bijection between primes in $\cal{V}_{\cG}$ and $\cF$-isomorphism classes of such pairs $(E,\fq)$. 
 
\begin{thm}\label{prop:quillenliftingfusion}
		Let $\cG=(\cF,S)$ be a $p$-local compact group, then the morphism
\[
\xymatrix{f \colon C^*(B\cG,\F_p) \ar[r] &  \prod \limits_{E \in \cE(\cG)}C^*(BE,\F_p)}
\]
satisfies Quillen lifting. Here $\cE(\cG)$ denotes a set of representatives of $\cF$-isomorphism classes of elementary abelian subgroups of $S$. 
\end{thm}
\begin{proof}\sloppy
  	  For brevity let us denote $\supp_{C^*(B\cG,\F_p)}$ by $\supp_{\cG}$, $\supp_{C^*(BE,\F_p)}$ by $\supp_E$ for any subgroup $E \leqslant S$, and $\supp_{\prod_{E \in \cE(\cG)}C^*(BE,\F_p)}$ by $\supp_{\cE(\cG)}$, 
  and similarly for cosupport. We also denote by  $\Mod_{\cG}$ and $\Mod_{\cE(\cG)}$ the corresponding categories of modules. Let $E\leq S$ be an elementary abelian subgroup, and denote by $\iota_E\colon BE\rightarrow B\cG$
  the corresponding inclusion. Note that for any morphism $f\in \Hom_\mathcal F(P,Q)$ in the fusion system, we obtain a homotopy $\iota_{P}\simeq \iota_{Q}\circ Bf$.
  
Let $M\in \Mod_{\cG}$ and assume there exists $\fp \in \res(\supp_{\cE(\cG)}(\Ind_{\cG}^{\cE(\cG)}M))$. By the discussion above there is a decomposition  
\[
\supp_{\cE(\cG)} (\Ind_{\cG}^{\cE(\cG)}(M))=\bigcup_{E \in \cE(\cG)} \supp_E \Ind_{\cG}^{E}(M),
\]
and similarly for cosupport and coinduction. Then there exists an elementary abelian subgroup $\widetilde{E} \leq S$ and 
  $\widetilde{\fq} \in \supp_{\widetilde{E}}(\Ind_{\cG}^{\widetilde{E}}M)$ with $\res(\widetilde{\fq})=\fp$. 
  Let $E \leq \widetilde{E}$ and $\fq \in \cV_E$ be a pair where $\fp$ originates. By \Cref{prop:elementary} and the subgroup theorem \cite[Thm.~4.4(ii)]{bg_stratifyingcompactlie} we obtain $\supp_{E}(\Ind_{\widetilde{E}}^{E}(L))=\res^{-1}(\supp_{\widetilde{E}}(L))$ for any $L \in \Mod_{\widetilde{E}}$, hence $\res(\fq)=\widetilde{\fq}$ and $\fq \in \supp_E(\Ind_{\cG}^{E}(M))$. 
  
  Let $(E',\fq')$ be another pair where $\fp$ originates. By \Cref{thm:qstratification}, $\cG$ satisfies strong Quillen stratification, and so any two such pairs are $\mathcal F$-conjugate. 
  That is, there is an isomorphism $f\in \Hom_{\mathcal F}(E',E)$ such that $Bf^*\colon C^*(BE,\F_p)\rightarrow C^*(BE',\F_p)$ is a homotopy equivalence with 
  $\Spec^h(Bf^*)(\fq)=\fq' $.  Let $\Ind_{Bf^*}$ denote the left adjoint to restriction along $Bf^*$. Since $\Ind_{Bf^*}\circ \Ind_{\cG}^{E}\simeq \Ind_{\cG}^{E'}$, we have that 
  $\fq'\in \supp_{E'} (\Ind_{\cG}^{E'}(M))\subseteq \supp_{\cE(\cG)}(\Ind_{\cG}^{\cE(\cG)}(M))$. In other words, for any pair $(E,\fq)$ where $\fp$ originates, we have $\fq \in \supp_E(\Ind_{\cG}^{E}M).$
  
  The same holds for coinduction, replacing induction by coinduction, support by cosupport, and using the subgroup theorem for cosupport \cite[Thm.~11.11]{bik12}. It follows that for any pair $(E,\fq)$ where $\fp$ originates, we have $\fq \in \cosupp_E(\Coind_{\cG}^{E}(N))$. Combining this with the previous paragraph, we obtain Quillen lifting along the morphism $f$. 
\end{proof}
By arguing along the same lines we can also obtain a stronger version of Chouinard's theorem than proven previously, cf.~\Cref{prop:plocalcompactchouniard}, by considering representatives of elementary abelian $p$-subgroups with respect to conjugation in $\cG$.  
\begin{cor}\label{cor:strongchouinard}
    Let $\cF$ be a saturated fusion system over a geometric discrete $p$-toral group $S$, then induction and coinduction along the morphism
\[
\xymatrix{f \colon C^*(B\cG,\F_p) \ar[r] &  \prod \limits_{E \in \cE(\cG)}C^*(BE,\F_p)}
\]
are conservative. 
\end{cor}
\begin{proof}
  Since support detects trivial objects by \Cref{thm:cosupptrivialobjects}, it suffices to prove that for $M \in \Mod_{C^*(B\cG,\F_p)}$ we have $\supp_{\cG}(M) \ne \varnothing $ if and only if $\bigcup_{E \in \cE(\cG)} \supp_E(\Ind_{\cG}^{E}M) \ne \varnothing$, and similarly for cosupport and coinduction. One direction is clear, so let us assume $\supp_{\cG}(M) \ne \varnothing$, which implies $M\not\simeq 0$. By \Cref{prop:toralchouinard}, there exists an elementary abelian subgroup $E\leq S$ such that $\Ind_{\cG}^E(M)\not\simeq 0$.  Let $\fp \in \res(\supp_{E}(\Ind_{\cG}^{E}M)) \subseteq \supp_{\cG}(M)$. As in the proof of \Cref{prop:quillenliftingfusion}, we see that all pairs $(E,\fq)$ in which $\fp$ originates are $\cF$-conjugate, so it follows that $\bigcup_{E \in \cE(\cG)} \supp_E(\Ind_{\cG}^{E}M) \ne \varnothing$, as required. A similar argument works for cosupport and coinduction. 
\end{proof}

\subsection{The proof of (co)stratification and its consequences}\label{sec:consequences}
Along with our abstract descent statements for stratification and costratification, the results of the previous two subsections now provide the necessary ingredients to prove (co)stratification for $p$-local compact groups over a geometric $p$-toral discrete group. We will repeatedly use without mention the fact that $C^*(B\cG,\F_p)$ is a Noetherian ring spectrum, which was proved in \Cref{cor:noetherian}, so that the abstract methods of \Cref{sec:basechange} apply. 

\begin{thm}\label{thm:plocalstrat}
Let $\cG=(S,\cF)$ be a $p$-local compact group with geometric $S$, then $\Mod_{C^*(B\cG,\F_p)}$ is canonically stratified. 
\end{thm}
\begin{proof}
Consider the morphism
\[
\xymatrix{f \colon C^*(B\cG,\F_p) \ar[r] & \prod \limits_{E \in \cE(\cG)} C^*(BE,\F_p).}
\]
By \Cref{prop:quillenliftingfusion} $f$ satisfies Quillen lifting and by \Cref{cor:strongchouinard} induction and coinduction along $f$ are conservative. The result then follows from \Cref{prop:elementary,thm:stratdescent}.
\end{proof}
In order to prove costratification for $p$-local compact groups, we first prove it for geometric discrete $p$-toral groups. 
\begin{prop}\label{prop:costratdiscreteptoral}
  Let $S$ be a geometric discrete $p$-toral group, then $\Mod_{C^*(BS,\F_p)}$ is canonically costratified. 
\end{prop}
\begin{proof}
 The $p$-completion of $BS$ is the classifying space of a $p$-compact toral group. The suspension spectrum $BS_+$ is $H\F_p$-equivalent to that of the classifying space of a $p$-compact group, and so by \Cref{lem:cochaincompletion} it suffices to prove the statement for the latter. In this case,  the morphism 
\[
\xymatrix{C^*(BS,\F_p) \ar[r] & \prod \limits_{E \in \cE(S)} C^*(BE,\F_p)}
\]
is finite by \Cref{prop:pcompactgroupchar}. Combining \Cref{prop:finitedescentcostrat,prop:elementary,prop:plocalcompactchouniard} gives the desired result. 
\end{proof}

\begin{thm}\label{thm:plocalcostrat}
Let $\cG=(S,\cF)$ be a $p$-local compact group with geometric $S$, then $\Mod_{C^*(B\cG,\F_p)}$ is canonically costratified. 
\end{thm}
\begin{proof}
We have shown in \Cref{thm:cochainsretract} that the canonical map $C^*(B\cG,\F_p) \to C^*(BS,\F_p)$ is split as a map of $C^*(B\cG,\F_p)$-modules, so that induction and coinduction along this map are conservative, see \Cref{lem:retract}. Since $\Mod_{C^*(B\cG,\F_p)}$ is canonically stratified by \Cref{thm:plocalstrat} and $\Mod_{C^*(BS,\F_p)}$ is canonically costratified by \Cref{prop:costratdiscreteptoral}, the result follows from \Cref{thm:costratdescent}. 
\end{proof}

The next corollary completes the work of Benson and Greenlees \cite{bg_stratifyingcompactlie}, who showed stratification in the special case that $G$ is a compact Lie group which has group of components a finite $p$-group. 

\begin{cor}\label{rem:compactliecostrat}
 Let $G$ be a compact Lie group, then $\Mod_{C^*(BG,\F_p)}$ is stratified and costratified by the canonical action of $H^*(BG,\F_p)$.
\end{cor}
\begin{proof}
By \Cref{lem:cochaincompletion} there is an equivalence of categories $\Mod_{C^*(BG,\F_p)} \simeq \Mod_{C^*(\pc{BG},\F_p)}$, so it suffices to prove the result for $\Mod_{C^*(\pc{BG},\F_p)}$. By \Cref{rem:pcompactmodel} there exists a $p$-local compact group $(S,\cF_s(G),\cal{L}_s(G))$ whose associated classifying space is homotopy equivalent to $\pc{BG}$. Moreover, $S$ is geometric by \Cref{cor:geometricexamples}. It follows from \Cref{thm:plocalstrat,thm:plocalcostrat} that $\Mod_{C^*(BG,\F_p)}$ is (co)stratified by the canonical action of $H^*(BG,\F_p)$. 
\end{proof}

We now list the consequences of (co)stratification for $p$-local compact groups that were stated abstractly in \Cref{sec:recollections} and \Cref{prop:abstractsubgroupthm}, while the list of examples given at the end of the theorem uses \Cref{cor:geometricexamples}.

\begin{thm}\label{thm:costratpgroupcons}
Let $\cG = (\cF,\cS)$ be a $p$-local compact group with geometric $S$. 
  \begin{enumerate}
    \item The map sending a localizing subcategory $\cL$ of $\Mod_{C^*(B\cG,\F_p)}$ to its left orthogonal $\cL^{\perp}$ induces a bijection between localizing and colocalizing subcategories of $\Mod_{C^*(B\cG)}$, and these are in bijection with subsets of $\Spec^h(H^*(B\cG,\F_p))$.
    \item There is a bijection between thick subcategories of compact $C^*(B\cG,\F_p)$-modules and specialization closed subsets of $\Spec^h(H^*(B\cG,\F_p))$, which takes a specialization closed subset $\cal{U}$ to the full-subcategory of compact $C^*(B\cG,\F_p)$-modules whose support is contained in $\cal{U}$. 
  \item The telescope conjecture holds in $\Mod_{C^*{(B\cG,\F_p)}}$: Let $\cal{L}$ be a localizing subcategory of $\Mod_{C^*{(B\cG,\F_p)}}$, then the following conditions are equivalent: 
  \begin{enumerate}
    \item The localizing subcategory $\cal{L}$ is smashing, i.e., the associated localizing endofunctor $L_{\cL}$ on $\Mod_{C^*{(B\cG,\F_p)}}$ preserves colimits. 
    \item The localizing subcategory $\cal{L}$ is generated by compact objects in $\Mod_{C^*(B\cG,\F_p)}$. 
    \item The support of $\cal{L}$ is specialization closed. 
  \end{enumerate}
  \item For any $M,N \in \Mod_{C^*(B\cG,\F_p)}$ there are identities:
  \begin{enumerate}
  	\item $\supp_{\cG}(M \otimes_{C^*(B\cG,\F_p)} N) = \supp_{\cG}(M) \cap \supp_{\cG}(N)$, and 
	\item $\cosupp_{\cG}(\Hom_{C^*(B\cG,\F_p)}(M,N)) = \supp_{\cG}(M) \cap \cosupp_{\cG}(N)$.
  \end{enumerate}
  \item Let $\cH=(S',\cF')$ be another $p$-local compact group with associated classifying space $B\cH$. For any morphism $f \colon C^*(B\cH,\F_p) \to C^*(B\cG,\F_p)$ of ring spectra we have 
  \[
\supp_{\cG}(\Ind M) = \res^{-1}\supp_{\cH}(M) \quad \text{ and } \quad \cosupp_{\cG}(\Coind M) = \res^{-1}\cosupp_{\cH}(M)
  \]
  for any $M \in \Mod_{C^*(B\cH,\F_p)}$. 
\end{enumerate}
In particular, each of these statements hold for compact Lie groups, connected $p$-compact groups, and $p$-local finite groups.
\end{thm}

\begin{rem}
Suppose that $G$ and $H$ are compact Lie groups. As in \Cref{rem:compactliecostrat} we may as well assume that $BG$ and $BH$ are $p$-complete. Then, any morphism $H \to G$ of compact Lie groups gives rise to a morphism $C^*(BG,\F_p) \to C^*(BH,\F_p)$ and, using \Cref{rem:compactliecostrat} again, we see that the subgroup theorem holds. Note that we need not assume that $H$ is a closed subgroup of $G$, in contrast to the subgroup theorem given in \cite[Thm.~5.2]{bg_stratifyingcompactlie}. 
\end{rem}

\section{Gorenstein duality for $p$-compact groups}

Let $G$ be a finite group and $\fp$ be a homogeneous prime ideal of $H^*(BG,k)$, where $k$ is a field of characteristic $p$. 
In the context of the stable module category of $G$, Benson conjectured that the object $T_{k}(I_{\fp})$ introduced by Benson and Krause in \cite{bk_pureinjectives} is stably equivalent to 
$\Gamma_{\fp}k$ up to a shift ~\cite{benson_moduleswithinjcohom}; here $I_{\fp}$ is the injective hull of $H^*(BG,k)/\fp$, and $T_{k}(I_{\fp})$ is 
constructed in the same way as the objects $T_C(I) \in \Mod_R$ introduced in \Cref{ssec:finitedescent}. This was proven by Benson and Greenlees 
in \cite{bg_localduality} by addressing the corresponding statement about $C^*(BG,k)$ in $\Mod_{C^*(BG,k)}$, where $G$ is a compact Lie group 
satisfying an orientability condition. Later, Benson gave an alternative proof for the stable module category of a finite group \cite{benson_shortproof}. 
Inspired by the latter, in \cite{bhv2} we studied the analogous conjecture for the category of modules over a commutative ring spectrum. 
To be specific, we studied the question when $R$ is \Gorenstein in the following sense, see \cite[Def.~5.6]{bhv2}. 

\begin{defn}\label{def:bc}
	Let $R$ be a ring spectrum. We say that $R$ is \Gorenstein with shift $\nu$ if, for each $\frak p \in \Spec^h (\pi_*R)$ of dimension $d=d(\fp)$, there is an equivalence $\Gamma_{\frak p}R \simeq \Sigma^{\nu+d} T_R(I_{\frak p})$. More generally, we say that $R$ is \Gorenstein with twist $J$ if there exists an invertible $R$-module $J$ such that there is an equivalence 
\[
\Gamma_{\frak p}R \otimes_R J \simeq \Sigma^{d} T_R(I_{\frak p})
\]
for any $\frak p \in \Spec^h(\pi_*R)$ of dimension $d$.
\end{defn}
In what follows, we will show that when $R = C^*(BG,k)$ for $G$ a $p$-compact group, then $R$ is absolute 
Gorenstein with shift given by the $\F_p$-cohomological dimension of $G$.

\subsection{$p$-compact groups from the stable perspective}
In this section we give an equivalent definition of $p$-compact group, following Bauer \cite{bauer_pcompact} and Rognes \cite{rognes_galois}. Since any loop space is equivalent to a topological group, we will work with topological groups $G$ instead of loop spaces. This allows us to more easily work with group actions by $G$. 
\begin{defn}[Rognes]\label{defn:rognespcompact}
	A topological group $G$ is $H\F_p$-locally stably dualizable if $(G_+)^{\wedge}_{p}$ is dualizable in the $p$-complete stable homotopy category. A $p$-compact group is an $H\F_p$-locally stably dualizable group whose (unstable) classifying space $BG$ is $p$-complete. 
\end{defn}
By \cite[Ex.~2.3.3]{rognes_galois} $G$ is $H\F_p$-locally stably dualizable if and only if $G$ is $\F_p$-finite. Moreover, the notion of $p$-compact group as defined in \Cref{defn:rognespcompact} agrees with that as given in \Cref{defn:pcompactgroup}, see Ex.~2.4.4 of \cite{rognes_galois}. 

Whenever equivariance plays a role in this section, we will work in the naive $p$-complete $G$-equivariant homotopy category. An object of this category is a $p$-complete spectrum $E$ together with a $G$-action on each space $E_n$ so that the structure maps $E_n \to \Omega E_{n+1}$ are $G$-equivariant homeomorphisms. We will write $S$ for the unit, which is the $p$-complete sphere spectrum with trivial $G$-action. Given objects $X,Y$ in this category, we write $X \otimes Y$ for the $p$-complete smash product of $X$ and $Y$, which is given the diagonal $G$-action, while the function spectrum $F(X,Y)$ has the conjugation $G$-action.

\subsection{The relative dualizing complex}
The following definition is due to Bauer~\cite{bauer_pcompact}.
\begin{defn}
	For a $p$-compact group $G$, define the dualizing spectrum $S_G$ by
	\[
S_G = (G_+)^{hG},
	\]
formed with respect to the standard right $G$ action on $G_+$ (thus we are left with a left $G$-action on $S_G$). 
\end{defn}
The dimension of a $p$-compact group $G$, $\dim_{\mathbb F_p}(G)$, is the $\F_p$-cohomological dimension of $G$. As shown in \Cref{adjspaceonecomponent}, $S_G$ is homotopy equivalent to a $H\F_p$-local sphere of dimension $\dim_{\mathbb F_p}(G)$.

Given a subgroup $H \leq_f G$ of a $p$-compact group, we define the relative dualizing object $\omega_f$ as the coinduced $C^*(BH,\F_p)$-module
\[
\omega_f := \Hom_{C^*(BG,\F_p)}(C^*(BH,\F_p),C^*(BG,\F_p)).
\]

\begin{thm}\label{thm:dualizing}
Suppose $H\leq_f G$ is a subgroup of a $p$-compact group $G$ of codimension $d = \dim_{\mathbb F_p}(G) - \dim_{\mathbb F_p}(H)$. Then, we have an equivalence of $C^*(BH,\F_p)$-modules
\[
\omega_f \simeq \Sigma^d C^*(BH,\F_p). 
\]
\end{thm}
\begin{proof}
The proof closely follows that of Benson--Greenlees \cite[Thm.~6.8]{bg_stratifyingcompactlie} . Let $h=\textrm{dim}_{\F_p}(H)$, and $g=\textrm{dim}_{\F_p}(G)$. Let $b = F(EG_+,H\F_p)$.
The quotient $G/H$ is non-equivariantly dualizable and so there is an equivalence $F(G/H_+,b) \simeq D(G/H_+) \otimes b$, which is even $G$-equivariant.

 For a monomorphism of connected $p$-compact groups $\alpha \colon H \to G$, Bauer showed \cite[Thm.~3]{bauer_pcompact} that there is a homotopy equivalence $G_+ \otimes_H S_H \simeq D(G/H_+) \otimes S_G$. By \Cref{relativeduality} this extends to not-necessarily connected $p$-compact groups. 

Combining the facts above, we see that there is a $G$-equivariant equivalence 
\[
F(G/H_+,b) \simeq D(G/H_+) \otimes b \simeq (G_+ \otimes_H S_H) \otimes D(S_G) \otimes b.
\]
Taking $G$-fixed points and using the definition of $b$, we see that 
\[
\begin{split}
	\Hom_b(F(G/H_+,b),b)^G &\simeq F((G_+ \otimes_H S_H) \otimes D(S_G),b)^G \\
	&\simeq F(EG_+ \otimes (G_+ \otimes_H S_H) \otimes D(S_G),H\mathbb F_p)^G\\
	&\simeq C^*(((G_+ \otimes_H S_H) \otimes D(S_G))_{hG},\mathbb F_p) \\
	&\simeq C^*((S_H \otimes D(S_G))_{hH},\mathbb F_p) \\
	&\simeq C^*(S_H \otimes D(S_G),\F_p)^{hH}. 
\end{split}
\]
It follows that there is a spectral sequence
\[
H^*(BH,H^*(S_H \otimes D(S_G),\F_p)) \implies H^*((S_H \otimes D(S_G))_{hH},\F_p).
\]

We now claim that the action of $H$ on the cohomology $H^*(S_H \otimes D(S_G),\F_p)$ is trivial.  Indeed, the action of $H$ arises from a morphism $\pi_0(H) \to \Z_p^\times \subseteq \pi_0S$ because the connected component must act trivially as it contains the identity. Since $\Z_p^\times \cong \Z_p \times \Z/(p-1)$, and $\pi_0(H)$ is a finite $p$-group, the action is trivial if $p$ is odd. Moreover, if $p = 2$, then the only possible non-trivial action is the sign action, but the sign representation is trivial on $\F_2$. Since $S_H \otimes D(S_G)$ is equivalent to a $p$-complete sphere of dimension $d$ we deduce that the reduced cohomology $H^*((S_H \otimes D(S_G))_{hH},\F_p)$ is a free $H^*(BH,\F_p)$-module on a class in degree $d$. We will produce a map of spectra that realizes this isomorphism in cohomology. To do this, we appeal to \cite[Lem.~4.1]{shamir_strat} with $R = H\F_p$, $A = C^*(BH,\F_p)$ and $M = C^*((S_H \otimes D(S_G))_{hH},\mathbb F_p)$. This result says that that map
\[
\pi_*\Hom_{H\F_p}(C^*(BH,\F_p),C^*((S_H \otimes D(S_G))_{hH},\mathbb F_p)) \longrightarrow \pi_*C^*((S_H \otimes D(S_G))_{hH},\mathbb F_p))
\]
induced by $H\F_p \to C^*(BH,\F_p)$ is surjective. It follows that there is a map of spectra $\Sigma^dC^*(BH,\F_p) \to C^*((S_H \otimes D(S_G))_{hH},\mathbb F_p)) $ which realizes the abstract isomorphism constructed previously. We conclude that 
\[
\Hom_b(F(G/H_+,b),b)^G \simeq \Sigma^d C^*(BH,\F_p). 
\]

On the other hand, as in \cite[Thm.~6.8]{bg_stratifyingcompactlie}, the natural map 
\[
\Hom_{b}(F(G/H_+,b),b)^G \to \Hom_{b^G}(F(G/H_+,b)^G,b^G) \simeq \Hom_{C^*(BG,\F_p)}(C^*(BH,\F_p),C^*(BG,\F_p))
\]
is an equivalence. Hence, we get
\[
\omega_f = \Hom_{C^*(BG,\F_p)}(C^*(BH,\F_p),C^*(BG,\F_p)) \simeq \Sigma^d C^*(BH,\F_p),\]
as desired. 
\end{proof}

\subsection{Gorenstein ascent}
We recall some basic definitions from \cite[Sec.~4]{bhv2}. 
\begin{defn}
We say that a ring spectrum $R$ with $\pi_*R$ local Noetherian of dimension $n$ is algebraically Gorenstein of shift $\nu$ if $\pi_*R$ is a graded Gorenstein ring; that is, the local cohomology $H_{\frak m}^i(\pi_*R)$ is non-zero only when $i = n$ and $(H^n_{\frak m}(\pi_*R))_t \cong (I_{\frak m})_{t-\nu-n}$. If $\pi_*R$ is non-local, then it is algebraically Gorenstein of shift $\nu$ if its localization at each maximal ideal is algebraically Gorenstein of shift $\nu$ in the above sense.
\end{defn}
For example, it is a consequence of work of Benson and Greenlees \cite{bg_commalg} that if $G$ is a compact Lie group whose adjoint representation is orientable and $H^*(BG,\F_p)$ is a Gorenstein ring, then $C^*(BG,\F_p)$ is algebraically Gorenstein with shift the dimension of $G$, see \cite[Exmp.~4.8]{bhv2}. By \cite[Prop.~4.7]{bhv2}, algebraically Gorenstein ring spectra are always \Gorenstein in the sense of \Cref{def:bc}.
\begin{defn}\label{def:nn}
We say that $S$ has a \Normalization (of shift $\nu$) if there exists a ring spectrum $R$ and a map of ring spectra $f \colon R \to S$ such that 
	\begin{enumerate}
		\item $R$ is \Gorenstein with shift $\nu$, 
		\item $S$ is a compact $R$-module via $f$, and
		\item $\omega_f$ is an invertible $S$-module. 
	\end{enumerate} 
\end{defn}
Our main result on detecting (twisted) absolute Gorenstein ring spectra is the following \cite[Thm~4.27]{bhv2}. 
\begin{thm}\label{thm:bc}
	Suppose $S$ has a \Normalization $f \colon R \to S$ of shift $\nu$. Then for each $\fp \in \Spec^h(\pi_*S)$ of dimension $d$, there is an equivalence	
\[
\Sigma^{-d} \Gamma_{\fp}S \otimes \omega_f \simeq \Sigma^{\nu} T_S(I_{\fp}).
\]	
 In particular, we obtain an isomorphism $\pi_\ast(\Gamma_{\fp}\omega_f)  \simeq (I_{\fp})_{\ast-d-\nu}$.
\end{thm}

We can now prove a version of Benson's conjecture for $p$-compact groups.
 \begin{thm}\label{thm:p-comp_goren}
	Let $G$ be a $p$-compact group of dimension $w$. Then for each $\fp \in \Spec^h (H^*(BG,\F_p))$ of dimension $d$, there is an equivalence	
\[
\Sigma^{-d} \Gamma_{\fp}C^*(BG,\F_p) \simeq \Sigma^{w} T_{C^*(BG,\F_p)}(I_{\fp}).
\]	
 In particular, we obtain an isomorphism $\pi_\ast(\Gamma_{\fp}C^*(BG,\F_p))  \simeq (I_{\fp})_{\ast-d-w}$.
\end{thm}
\begin{proof}
As noted previously, for any $p$-compact group $G$, there is always a unitary embedding  $ B\iota \colon BG \to BSU(n)^{\wedge}_p$. By \Cref{prop:unitaryembbedingfinite} we see that $C^*(BG,\F_p)$ is a compact $C^*(BSU(n),\F_p)$-module via $B\iota^*$. Moreover,  $C^*(BSU(n),\F_p)$ is absolute Gorenstein with shift $n^2-1$ \cite[Ex.~4.8(2)]{bhv2}.  

Note that $f \coloneqq B\iota^* \colon C^*(BSU(n),\F_p) \to C^*(BG,\F_p)$ satisfies the conditions of \Cref{def:nn}. By \Cref{thm:dualizing} we have $\omega_f \simeq \Sigma^{n^2-1-w}C^*(BG,\F_p)$. It then follows from \Cref{thm:bc} that
\[
\Sigma^{-d+n^2-1-w} \Gamma_{\fp}C^*(BG,\F_p) \simeq \Sigma^{n^2-1} T_{C^*(BG,\F_p)}(I_{\fp}),
\]
hence the result. 
\end{proof}
This has the following standard consequences. 
\begin{cor}\label{cor:pgroupss}
Let $G$ be a $p$-compact group of dimension $w$, and $\fp \in \Spec^h(H^*(BG,\F_p))$ a prime of dimension $d$.  
\begin{enumerate}
	\item There is a spectral sequence 
\[
E_2^{s,t} \cong (H_{\fp}^sH^*(BG,\F_p)_{\fp})_t \implies (I_{\fp})_{-t-s-w-d}.
		\]
In particular, if $H^*(BG,\F_p)$ is Cohen--Macaulay, then it is Gorenstein. 
\item $H^*(BG,\F_p)$ is generically Gorenstein (i.e., its localization at any minimal prime is Gorenstein). 
\end{enumerate}
\end{cor}
\begin{proof}
	The first follows from \cite[Prop.~3.19(1)]{bhv2}, using the previous result to identify the target. The second follows from \cite[Cor.~7.4]{green_lyub} where, in the language of the latter, $H^*(BG,\F_p)$ has a local cohomology theorem with shift $-w$. Alternatively, this can be proven directly using the spectral sequence from (1): if $\fp$ is minimal, then there is only one nonvanishing column in the spectral sequence, and it follows that $H^*(BG,\F_p)_{\fp}$ is Gorenstein. 
\end{proof}

\renewcommand {\sp}{\Sigma^\infty}
\newcommand{\spp}{\Sigma^\infty_+}
\newcommand {\sm}{\otimes}

\appendix
\section{Duality for stably dualizable groups, by Tilman Bauer}

The aim of this appendix is to generalize various needed results from \cite{bauer_pcompact} to the class of $H\F_p$-local $\F_p$-finite groups. In that paper, the author restricted attention to connected $p$-compact groups. This assumption is never really needed, but we want to give short proofs of the relevant results for the sake of completeness.

We work in the category of $H\F_p$-local spectra $X$ with a (left) $G$-action on every space $X_n$ and such that the structure maps are $G$-equivariant.  We will call a $G$-equivariant map $f\colon X\rightarrow Y$ between $G$-spectra a $hG$-equivalence if it is a weak equivalence of underlying spectra.  

\begin{lem}\label{freediscretehomotopyfixedpoints}
Let $H < G$ be an inclusion of $H\F_p$-local $\F_p$-finite groups, and let $X$ be a non-equivariant spectrum. Then the $H^{\op}$-action on the mapping spectrum $\map(G_+,X)$ gives a weak equivalence, natural in $H$ and $G$:
\[
\map(G_+,X)^{hH^{\op}} \simeq \map(G/H_+,X).
\]
\end{lem}
\begin{proof}
There are equivalences
\begin{align*}
\map(G_+,X)^{hH^{\op}} \simeq & \map_{H^{\op}}(EH_+,\map(G_+,X)) \simeq \map_{H^{\op}}((EH \times G)_+,X)\\
\simeq & \map((EH \times G)_+/H,X) \simeq \map(G/H_+,X).\qedhere
\end{align*}
\end{proof}

Recall the following definition due to Bauer~\cite{bauer_pcompact}. For a $p$-compact group $G$, define the dualizing spectrum $S_G$ by
	\[
S_G = (G_+)^{hG},
	\]
formed with respect to the standard right $G$ action on $G_+$ (thus we are left with a left $G$-action on $S_G$). 

The dimension of a $p$-compact group $G$, $\dim_{\mathbb F_p}(G)$, is the $\F_p$-cohomological dimension of $G$. In \cite{bauer_pcompact}, it is proven that $S_G$ is homotopy equivalent to a $H\F_p$-local sphere of dimension $\dim_{\mathbb F_p}(G)$ when $G$ is connected.

\begin{lem}\label{adjspaceonecomponent}
Let $G$ be an $H\F_p$-local $\F_p$-finite group of dimension $d$. Then the dualizing spectrum $S_G$ is equivalent to an $H\F_p$-local sphere of the same dimension $d$, and the inclusion of the identity component $G_0 \hookrightarrow G$ induces a $G_0$-equivariant equivalence $S_{G_0} \to S_G$.
\end{lem}
\begin{proof}
Let $\pi=\pi_0G$ be the finite group of components. Then $G_0$-equivariantly, $G \simeq \map(\pi_+,G_0)$ and since the suspension functor $\spp{-}$ sends coproducts to wedges,
\[
(\spp{G})^{hG_0} = \map_{G_0}((EG_{0})_+,\spp{G}) \simeq \map(\pi_+,\spp{G_0}^{hG_0}) = \map(\pi_+,S_{G_0}),
\]
which is a finite union of $H\F_p$-local spheres by \cite[Cor.~23]{bauer_pcompact}. Then
\[
\spp{G}^{hG} = \left(\spp{G}^{hG_0}\right)^{h\pi} \simeq \map(\pi_+,S_{G_0})^{h\pi} \simeq S_{G_0}
\]
by \Cref{freediscretehomotopyfixedpoints}.
Moreover, the inclusion of the unit component $G_0 \hookrightarrow G$ induces the inclusion of the identity factor $S_{G_0} \hookrightarrow \spp{G}^{hG_0}$ and hence an equivalence with $S_G$.
 \end{proof}

\begin{lem}\label{inducedsphere}
Let $H < G$ be an inclusion of $H\F_p$-local $\F_p$-finite groups. Then there is a $G$-equivariant weak equivalence, natural in $G$:
\[
G_+ \sm_H S_H \to \spp{G}^{hH^{\op}}.
\]
\end{lem}
\begin{proof}
In the case of connected $H$ and $G$, this is \cite[Lem.~19]{bauer_pcompact}. In general, the natural map
\[
G_+ \sm_H \map(EH_+,\spp{H}) \to \map(EH_+,G_+\sm_H \spp{H}) \simeq \map(EH_+,\spp{G}) 
\]
induces a $G$-equivariant map $\phi\colon G_+ \sm_H S_H \to \spp{G}^{hH^{\op}}$ by passing to $H^{\op}$-fixed points. 

After taking $H_0$-fixed points, the right hand side is $\spp{G}^{hH_0^{op}}\simeq \map(\pi_0(G)_+,\spp{G_0}^{hH_0})$ and the left hand side is equivalent to $\map(\pi_0(G)_+, (G_0)_+ \sm_{H_0} S_{H_0})$. 

Non-equivariantly then, $G_+ \sm_H S_H$ splits as $\map((\pi_0G/\pi_0H)_+, G_0 \sm_{H_0} S_{H_0})$ and 
\begin{align*}
\spp{G}^{hH^{\op}} \simeq & \left(\map(\pi_0 G_+, \spp{G_0})^{hH_0^{\op}} \right)^{h\pi_0H^{\op}}\\
 \simeq & \map(\pi_0 G_+, \spp{G_0}^{hH_0^{\op}})^{h\pi_0 H^{\op}} \simeq \map((\pi_0G/\pi_0 H)_+, \spp{G_0}^{hH_0^{\op}})
\end{align*}
by \Cref{freediscretehomotopyfixedpoints}, and $\phi$ respects this splitting. By \cite[Lem.~19]{bauer_pcompact}, $\phi$ is a weak equivalence on every wedge summand, hence a weak equivalence. 
\end{proof}

\begin{prop} \label{relativeduality}
Let $H<G$ be a monomorphism of $H\F_p$-local $\F_p$-finite groups. Then there is a relative $G$-equivariant duality weak equivalence
\[
G_+ \sm_H S_H \simeq D\left(\spp{G/H}\right) \sm S_G.
\]
Moreover, for inclusions $K < H < G$ of $H\F_p$-local $\F_p$-finite groups, the following diagram commutes:
\[
\xymatrix { \spp{G}^{hH^{\op}} \ar[d]^{\operatorname{res}} & G_+ 
\sm_H  S_H \ar@{<.>}[r]^-{\sim} \ar[l]_-{\sim} &  D(G/H_+) \sm S_G 
\ar[d]^{D(\operatorname{proj}) \sm \id}\\ \spp{G}^{hK^{\op}}  & G_+ 
\sm_K S_K \ar@{<.>}[r]^-{\sim} \ar[l]_-{\sim} & D(G/K_+) \sm S_G } 
\]
\end{prop}
\begin{proof}
In \cite[Prop.~22]{bauer_pcompact}, we constructed a weak equivalence
\[
S_G \sm DG_+ \to \spp{G}
\]
for connected $p$-compact groups $G$, which is equivariant with respect to two different $G$-actions. The first is multiplication on $DG_+$ and $\spp{G}$ and the standard (conjugation) action on $S_G$, and the second one is right multiplication on $DG_+$ and $\spp{G}$ and the trivial action on $S_G$. Rognes \cite[Thm.~3.1.4]{rognes_galois} extended this proof to stably dualizable groups, in particular to $H\F_p$-local $\F_p$-finite groups. Taking $H$-homotopy fixed points with respect to that second action, we obtain $hG$-equivalences
\[
D(G/H)_+ \sm S_G \xleftarrow{\sim} DG_+^{hH} \sm S_G \xrightarrow{\sim} (DG_+ \sm S_G)^{hH} \xrightarrow{\sim} \spp{G}^{hH},
\]
where the left hand map is the equivalence from \Cref{freediscretehomotopyfixedpoints}.
Composing with the natural equivalence of \Cref{inducedsphere} gives the result.

For the naturality statement, consider the following diagram: 
\[
\xymatrix@1 { D(G/H_+) \sm S_G \ar[d] & DG_+^{hH^{\op}} \sm S_G \ar[l] \ar[d] 
\ar[r] &  (DG_+ \sm S_G)^{hH^{\op}} \ar[d] \ar[r] & \spp{G}^{h H^{\op}} 
\ar[d] \\
D(G/K_+) \sm S_G & DG_+^{hK^{\op}} \sm S_G \ar[l] \ar[r] & (DG_+ \sm  S_G)^{hK^{\op}} \ar[r]& \spp{G}^{hK^{\op}} } 
\]
The left hand square commutes by \Cref{freediscretehomotopyfixedpoints}, the other two  for trivial reasons.
\end{proof}

\biblio
\bibliography{duality}\bibliographystyle{alpha}

\begin{thebibliography}{EKMM97}

\bibitem[AG09]{AG_classification2}
Kasper K.~S. Andersen and Jesper Grodal.
\newblock The classification of 2-compact groups.
\newblock {\em J. Amer. Math. Soc.}, 22(2):387--436, 2009.

\bibitem[AGMV08]{AGMV_classification1}
K.~K.~S. Andersen, J.~Grodal, J.~M. M{\o}ller, and A.~Viruel.
\newblock The classification of {$p$}-compact groups for {$p$} odd.
\newblock {\em Ann. of Math. (2)}, 167(1):95--210, 2008.

\bibitem[Agu89]{aguade}
J.~Aguad\'{e}.
\newblock Constructing modular classifying spaces.
\newblock {\em Israel J. Math.}, 66(1-3):23--40, 1989.

\bibitem[Bau04]{bauer_pcompact}
Tilman Bauer.
\newblock {$p$}-compact groups as framed manifolds.
\newblock {\em Topology}, 43(3):569--597, 2004.

\bibitem[BC76]{bc_dual}
Edgar~H. Brown, Jr. and Michael Comenetz.
\newblock Pontrjagin duality for generalized homology and cohomology theories.
\newblock {\em Amer. J. Math.}, 98(1):1--27, 1976.

\bibitem[BCR97]{bcr_thick}
D.~J. Benson, Jon~F. Carlson, and Jeremy Rickard.
\newblock Thick subcategories of the stable module category.
\newblock {\em Fund. Math.}, 153(1):59--80, 1997.

\bibitem[BDS16]{bds_wirth}
Paul Balmer, Ivo Dell'Ambrogio, and Beren Sanders.
\newblock Grothendieck--{N}eeman duality and the {W}irthm\"uller isomorphism.
\newblock {\em Compos. Math.}, 152(8):1740--1776, 2016.

\bibitem[Ben01]{benson_moduleswithinjcohom}
David Benson.
\newblock Modules with injective cohomology, and local duality for a finite
  group.
\newblock {\em New York J. Math.}, 7:201--215 (electronic), 2001.

\bibitem[Ben08]{benson_shortproof}
Dave Benson.
\newblock Idempotent {$kG$}-modules with injective cohomology.
\newblock {\em J. Pure Appl. Algebra}, 212(7):1744--1746, 2008.

\bibitem[BG97]{bg_commalg}
D.~J. Benson and J.~P.~C. Greenlees.
\newblock Commutative algebra for cohomology rings of classifying spaces of
  compact {L}ie groups.
\newblock {\em J. Pure Appl. Algebra}, 122(1-2):41--53, 1997.

\bibitem[BG08]{bg_localduality}
David~J. Benson and J.~P.~C. Greenlees.
\newblock Localization and duality in topology and modular representation
  theory.
\newblock {\em J. Pure Appl. Algebra}, 212(7):1716--1743, 2008.

\bibitem[BG14]{bg_stratifyingcompactlie}
David Benson and John Greenlees.
\newblock Stratifying the derived category of cochains on {$BG$} for {$G$} a
  compact {L}ie group.
\newblock {\em J. Pure Appl. Algebra}, 218(4):642--650, 2014.

\bibitem[BGS13]{shamir_pcochains}
David~J. Benson, John P.~C. Greenlees, and Shoham Shamir.
\newblock Complete intersections and mod {$p$} cochains.
\newblock {\em Algebr. Geom. Topol.}, 13(1):61--114, 2013.

\bibitem[BHV17]{bhv2}
T.~{Barthel}, D.~{Heard}, and G.~{Valenzuela}.
\newblock Local duality for structured ring spectra.
\newblock {\em Journal of Pure and Applied Algebra}, 2017.

\bibitem[BHV18]{bhv}
Tobias Barthel, Drew Heard, and Gabriel Valenzuela.
\newblock Local duality in algebra and topology.
\newblock {\em Adv. Math.}, 335:563--663, 2018.

\bibitem[BIK08]{benson_local_cohom_2008}
Dave Benson, Srikanth~B. Iyengar, and Henning Krause.
\newblock Local cohomology and support for triangulated categories.
\newblock {\em Ann. Sci. \'Ec. Norm. Sup\'er. (4)}, 41(4):573--619, 2008.

\bibitem[BIK11a]{bik11}
Dave Benson, Srikanth~B. Iyengar, and Henning Krause.
\newblock Stratifying triangulated categories.
\newblock {\em J. Topol.}, 4(3):641--666, 2011.

\bibitem[BIK11b]{bik_finitegroups}
David~J. Benson, Srikanth~B. Iyengar, and Henning Krause.
\newblock Stratifying modular representations of finite groups.
\newblock {\em Ann. of Math. (2)}, 174(3):1643--1684, 2011.

\bibitem[BIK12]{bik12}
David~J. Benson, Srikanth~B. Iyengar, and Henning Krause.
\newblock Colocalizing subcategories and cosupport.
\newblock {\em Journal f{{\"u}}r die Reine und Angewandte Mathematik. [Crelle's
  Journal]}, 673:161--207, 2012.

\bibitem[BK72]{bousfield_homotopy_1972}
A.~K. Bousfield and D.~M. Kan.
\newblock {\em Homotopy limits, completions and localizations}.
\newblock Lecture {Notes} in {Mathematics}, {Vol}. 304. Springer-Verlag,
  Berlin-New York, 1972.

\bibitem[BK02]{bk_pureinjectives}
David Benson and Henning Krause.
\newblock Pure injectives and the spectrum of the cohomology ring of a finite
  group.
\newblock {\em J. Reine Angew. Math.}, 542:23--51, 2002.

\bibitem[BK08]{krausebenson_kg}
David~John Benson and Henning Krause.
\newblock Complexes of injective {$kG$}-modules.
\newblock {\em Algebra Number Theory}, 2(1):1--30, 2008.

\bibitem[BLO03]{blo_fusion}
Carles Broto, Ran Levi, and Bob Oliver.
\newblock The homotopy theory of fusion systems.
\newblock {\em J. Amer. Math. Soc.}, 16(4):779--856, 2003.

\bibitem[BLO07]{blo_pcompact}
Carles Broto, Ran Levi, and Bob Oliver.
\newblock Discrete models for the {$p$}-local homotopy theory of compact {L}ie
  groups and {$p$}-compact groups.
\newblock {\em Geom. Topol.}, 11:315--427, 2007.

\bibitem[BLO14]{BLO_finiteloopspaces}
Carles Broto, Ran Levi, and Bob Oliver.
\newblock An algebraic model for finite loop spaces.
\newblock {\em Algebr. Geom. Topol.}, 14(5):2915--2981, 2014.

\bibitem[Bou79]{bousfield_locspectra}
A.~K. Bousfield.
\newblock The localization of spectra with respect to homology.
\newblock {\em Topology}, 18(4):257--281, 1979.

\bibitem[BZ88]{brotozarati_steenrod}
C.~Broto and S.~Zarati.
\newblock Nil-localization of unstable algebras over the {S}teenrod algebra.
\newblock {\em Math. Z.}, 199(4):525--537, 1988.

\bibitem[Cas06]{castellana_gqrn}
Nat\`alia Castellana.
\newblock On the {$p$}-compact groups corresponding to the {$p$}-adic
  reflection groups {$G(q,r,n)$}.
\newblock {\em Trans. Amer. Math. Soc.}, 358(7):2799--2819, 2006.

\bibitem[CC17]{cancas_finloop}
Jos\'e Cantarero and Nat\`alia Castellana.
\newblock Unitary embeddings of finite loop spaces.
\newblock {\em Forum Math.}, 29(2):287--311, 2017.

\bibitem[Che13]{chermak_existence}
Andrew Chermak.
\newblock Fusion systems and localities.
\newblock {\em Acta Math.}, 211(1):47--139, 2013.

\bibitem[Cho76]{chouinard}
Leo~G. Chouinard.
\newblock Projectivity and relative projectivity over group rings.
\newblock {\em J. Pure Appl. Algebra}, 7(3):287--302, 1976.

\bibitem[DMW87]{dwyermillerwilk}
William~G. Dwyer, Haynes~R. Miller, and Clarence~W. Wilkerson.
\newblock The homotopic uniqueness of {$BS^3$}.
\newblock In {\em Algebraic topology, {B}arcelona, 1986}, volume 1298 of {\em
  Lecture Notes in Math.}, pages 90--105. Springer, Berlin, 1987.

\bibitem[DRV07]{drv_exotic}
Antonio D{\'\i}az, Albert Ruiz, and Antonio Viruel.
\newblock All {$p$}-local finite groups of rank two for odd prime {$p$}.
\newblock {\em Trans. Amer. Math. Soc.}, 359(4):1725--1764, 2007.

\bibitem[DW93]{dw_di4}
W.~G. Dwyer and C.~W. Wilkerson.
\newblock A new finite loop space at the prime two.
\newblock {\em J. Amer. Math. Soc.}, 6(1):37--64, 1993.

\bibitem[DW94]{dwyerwilkerson_finloop}
W.~G. Dwyer and C.~W. Wilkerson.
\newblock Homotopy fixed-point methods for {L}ie groups and finite loop spaces.
\newblock {\em Ann. of Math. (2)}, 139(2):395--442, 1994.

\bibitem[DW09]{dwyerwilkerson_transfer}
W.~G. Dwyer and C.~W. Wilkerson.
\newblock The fundamental group of a {$p$}-compact group.
\newblock {\em Bull. Lond. Math. Soc.}, 41(3):385--395, 2009.

\bibitem[DZ87]{dwyerzabrodsky_classifyingspaces}
W.~Dwyer and A.~Zabrodsky.
\newblock Maps between classifying spaces.
\newblock In {\em Algebraic topology, {B}arcelona, 1986}, volume 1298 of {\em
  Lecture Notes in Math.}, pages 106--119. Springer, Berlin, 1987.

\bibitem[EKMM97]{ekmm}
A.~D. Elmendorf, I.~Kriz, M.~A. Mandell, and J.~P. May.
\newblock {\em Rings, modules, and algebras in stable homotopy theory},
  volume~47 of {\em Mathematical Surveys and Monographs}.
\newblock American Mathematical Society, Providence, RI, 1997.
\newblock With an appendix by M. Cole.

\bibitem[Eve61]{evens}
Leonard Evens.
\newblock The cohomology ring of a finite group.
\newblock {\em Trans. Amer. Math. Soc.}, 101:224--239, 1961.

\bibitem[GL00]{green_lyub}
J.~P.~C. Greenlees and G.~Lyubeznik.
\newblock Rings with a local cohomology theorem and applications to cohomology
  rings of groups.
\newblock {\em J. Pure Appl. Algebra}, 149(3):267--285, 2000.

\bibitem[Gon16]{gonzalez_approxpcompact}
Alex Gonzalez.
\newblock Finite approximations of {$p$}-local compact groups.
\newblock {\em Geom. Topol.}, 20(5):2923--2995, 2016.

\bibitem[{Gre}16]{greenlees_hi}
J.~P.~C. {Greenlees}.
\newblock {Homotopy Invariant Commutative Algebra over fields}.
\newblock {\em arXiv preprint \url{https://arxiv.org/abs/1601.02473}}, January
  2016.

\bibitem[GY83]{gy_noetheriangraded}
Shiro Goto and Kikumichi Yamagishi.
\newblock Finite generation of {N}oetherian graded rings.
\newblock {\em Proc. Amer. Math. Soc.}, 89(1):41--44, 1983.

\bibitem[HPS97]{hps_axiomatic}
Mark Hovey, John~H. Palmieri, and Neil~P. Strickland.
\newblock Axiomatic stable homotopy theory.
\newblock {\em Mem. Amer. Math. Soc.}, 128(610):x+114, 1997.

\bibitem[HSS00]{sym_spectra}
Mark Hovey, Brooke Shipley, and Jeff Smith.
\newblock Symmetric spectra.
\newblock {\em J. Amer. Math. Soc.}, 13(1):149--208, 2000.

\bibitem[Lan92]{La}
Jean Lannes.
\newblock Sur les espaces fonctionnels dont la source est le classifiant d'un
  {$p$}-groupe ab\'elien \'el\'ementaire.
\newblock {\em Inst. Hautes \'Etudes Sci. Publ. Math.}, (75):135--244, 1992.
\newblock With an appendix by Michel Zisman.

\bibitem[Lin17]{linckelman_stratification}
Markus Linckelmann.
\newblock Quillen's stratification for fusion systems.
\newblock {\em Comm. Algebra}, 45(12):5227--5229, 2017.

\bibitem[LL15]{ll_uniqueness}
Ran Levi and Assaf Libman.
\newblock Existence and uniqueness of classifying spaces for fusion systems
  over discrete {$p$}-toral groups.
\newblock {\em J. Lond. Math. Soc. (2)}, 91(1):47--70, 2015.

\bibitem[LO02]{lo_solomon}
Ran Levi and Bob Oliver.
\newblock Construction of 2-local finite groups of a type studied by {S}olomon
  and {B}enson.
\newblock {\em Geom. Topol.}, 6:917--990, 2002.

\bibitem[Lur17]{ha}
Jacob Lurie.
\newblock {\em Higher {A}lgebra}.
\newblock 2017.
\newblock Draft available from author's website as
  \url{http://www.math.harvard.edu/~lurie/papers/HA.pdf}.

\bibitem[Mar83]{margolis}
H.~R. Margolis.
\newblock {\em Spectra and the {S}teenrod algebra}, volume~29 of {\em
  North-Holland Mathematical Library}.
\newblock North-Holland Publishing Co., Amsterdam, 1983.
\newblock Modules over the Steenrod algebra and the stable homotopy category.

\bibitem[Not98]{notbohm_grassmanian}
Dietrich Notbohm.
\newblock Topological realization of a family of pseudoreflection groups.
\newblock {\em Fund. Math.}, 155(1):1--31, 1998.

\bibitem[Not03]{notbohm_di(4)}
Dietrich Notbohm.
\newblock On the 2-compact group {${\rm DI}(4)$}.
\newblock {\em J. Reine Angew. Math.}, 555:163--185, 2003.

\bibitem[Oli13]{boboliver_existence}
Bob Oliver.
\newblock Existence and uniqueness of linking systems: {C}hermak's proof via
  obstruction theory.
\newblock {\em Acta Math.}, 211(1):141--175, 2013.

\bibitem[Pui06]{puig_fusion}
Lluis Puig.
\newblock Frobenius categories.
\newblock {\em J. Algebra}, 303(1):309--357, 2006.

\bibitem[Qui71]{quillen_stratification}
Daniel Quillen.
\newblock The spectrum of an equivariant cohomology ring. {I}, {II}.
\newblock {\em Ann. of Math. (2)}, 94:549--572; ibid. (2) 94 (1971), 573--602,
  1971.

\bibitem[Rag06]{ragnarsson_transfer}
K\'ari Ragnarsson.
\newblock Classifying spectra of saturated fusion systems.
\newblock {\em Algebr. Geom. Topol.}, 6:195--252, 2006.

\bibitem[Rag08]{rag_transfer_diagram}
K\'ari Ragnarsson.
\newblock Retractive transfers and {$p$}-local finite groups.
\newblock {\em Proc. Edinb. Math. Soc. (2)}, 51(2):465--487, 2008.

\bibitem[Rav92]{orangebook}
Douglas~C. Ravenel.
\newblock {\em Nilpotence and periodicity in stable homotopy theory}, volume
  128 of {\em Annals of Mathematics Studies}.
\newblock Princeton University Press, Princeton, NJ, 1992.
\newblock Appendix C by Jeff Smith.

\bibitem[Rec71]{rector_loops}
David~L. Rector.
\newblock Loop structures on the homotopy type of {$S\sp{3}$}.
\newblock pages 99--105. Lecture Notes in Math., Vol. 249, 1971.

\bibitem[Rec84]{rector_quillenstrat}
D.~L. Rector.
\newblock Noetherian cohomology rings and finite loop spaces with torsion.
\newblock {\em J. Pure Appl. Algebra}, 32(2):191--217, 1984.

\bibitem[Rog08]{rognes_galois}
John Rognes.
\newblock Galois extensions of structured ring spectra. {S}tably dualizable
  groups.
\newblock {\em Mem. Amer. Math. Soc.}, 192(898):viii+137, 2008.

\bibitem[Sch94]{Schwartz_book}
Lionel Schwartz.
\newblock {\em Unstable modules over the {S}teenrod algebra and {S}ullivan's
  fixed point set conjecture}.
\newblock Chicago Lectures in Mathematics. University of Chicago Press,
  Chicago, IL, 1994.

\bibitem[Sha12]{shamir_strat}
Shoham Shamir.
\newblock Stratifying derived categories of cochains on certain spaces.
\newblock {\em Math. Z.}, 272(3-4):839--868, 2012.

\bibitem[Ven59]{venkov}
B.~B. Venkov.
\newblock Cohomology algebras for some classifying spaces.
\newblock {\em Dokl. Akad. Nauk SSSR}, 127:943--944, 1959.

\bibitem[Zie09]{Ziemianski20091239}
Krzysztof Ziemia\'nski.
\newblock A faithful unitary representation of the 2-compact group {${\rm
  DI}(4)$}.
\newblock {\em J. Pure Appl. Algebra}, 213(7):1239--1253, 2009.

\end{thebibliography}
\end{document}